\documentclass[11pt,reqno]{article}
\usepackage{amssymb,amsmath,amsthm,amstext,amscd}
\usepackage{amsfonts}
\usepackage{mathrsfs}
\usepackage{graphicx}
 \usepackage{subcaption}
\usepackage[english]{babel}

\usepackage{bbm}
\usepackage{url}
\usepackage[textwidth=18cm,textheight=24cm]{geometry}
\newtheorem{thm}{Theorem}
\newtheorem{lem}{Lemma}
\newtheorem{cor}{Corollary}
\theoremstyle{definition}

\newtheorem{assumptionA}{A-\hspace{-1.2mm}}

\newtheorem{example}{Example}[section]
\theoremstyle{remark}
\newtheorem{remark}{Remark}

\begin{document}

\title{On Tamed Milstein Schemes of SDEs Driven by L\'{e}vy Noise}
\author{Chaman Kumar and Sotirios Sabanis \\ \small School of Mathematics,  University of Edinburgh, \\ \small Edinburgh, EH9 3JZ, U.K.}

\pagenumbering{arabic}
\date{\today}
\maketitle
\begin{abstract}
We extend the taming techniques developed in \cite{konstantinos2014,sabanis2013} to construct explicit Milstein schemes that numerically approximate L\'evy driven stochastic differential equations with super-linearly growing drift coefficients. The classical rate of convergence is recovered when the first derivative of the drift coefficient satisfies a polynomial Lipschitz condition.
\bigskip

\noindent



\bigskip

\end{abstract}

\section{Introduction}
The models incorporating event-driven uncertainties are increasingly becoming popular in  finance, economics, engineering, medical sciences, ecology  and many other branches of sciences. For example, in finance, often the stock price movements are suddenly and significantly influenced by market crashes, market booms, announcements made by central banks, changes in credit ratings, defaults, etc. The stochastic differential equations (SDEs) driven by L\'{e}vy noise are more realistic models to be used in such event-driven phenomena. Some of the references are \cite{cont2004},  \cite{oksendal2005}, \cite{situ2005}  and references therein.

The L\'evy driven SDEs rarely  possess explicit solutions, that necessitates the development of numerical schemes to approximate their  solutions.  Recently, many explicit and implicit numerical schemes of L\'{e}vy driven SDEs  have been studied and their convergence in  strong as well as in weak sense have been proved. The interested reader may refer to \cite{bruti2007}, \cite{Dereich}, \cite{Dereich_Heidenreich}, \cite{higham2005}, \cite{higham2006}, \cite{jacod2005}, \cite{platen2010} along with their references for a comprehensive discussion on numerous numerical schemes of different orders.

In \cite{hutzenthaler2010}, authors have shown that the efficient Euler schemes of SDEs (without jumps) with super-linearly growing drift coefficients diverge in $\mathcal{L}^p$-sense. Several numerical schemes based on taming techniques have been proposed in the past few years to overcome these difficulties. One could refer to  \cite{HJ}, \cite{hutzenthaler2014}, \cite{hutzenthaler2012}, \cite{sabanis2013}, \cite{Tretyakov-Zhang}, \cite{zhang2014} and \cite{sabanis2015} for order $0.5$ tamed Euler schemes and \cite{wanga2013} and \cite{zhang2014} for order $1.0$ tamed Milstein schemes of SDEs (without jumps). For $0.5$ order tamed schemes of SDEs driven by L\'{e}vy noise, the only reference available to us is \cite{konstantinos2014}. The taming techniques proposed in this article can be extended to any desired order of convergence. However, for mathematical simplicity, we discuss order $1.0$ tamed Milstein scheme of L\'{e}vy driven SDEs with super-linearly growing drift coefficients. It can also be concluded that our methodology and techniques can be extended to a tamed scheme of any order of convergence when drift coefficient grows super-linearly. Further, our calculations are motivated by the classical methodology (see also \cite{konstantinos2014}, \cite{sabanis2013}) and hence are more refined than those present in the literature. To the best of our knowledge, this is the first article on Milstein  schemes of SDEs driven by L\'{e}vy noise when drift coefficient grows super-linearly.

We conclude this section by introducing some basic notations. We use $|a|$ to denote the Euclidean norm of a $d$-dimensional vector $a$, whereas $|A|$ for the Hilbert-Schmidt norm of a $d\times m$ matrix $A$. Also, the transpose of the matrix $A$ is denoted  by $A^*$. The inner product of two $d$-dimensional vectors $x$ and $y$ is denoted by $xy$. The $i$th element of a $d$-dimensional vector is denoted by $a^i$, whereas $A^{(i,j)}$ and $A^{(j)}$ stand for $(i,j)$th element and $j$th column of a $d \times m$ matrix $A$ respectively, for $i=1,\ldots,d$, $j=1,\ldots,m$. The notation $\lfloor x \rfloor$ stands for the integer part of a positive real number $x$. Finally, we use $K>0$ to denote a generic constant, which is  independent of $n$ and varies from place to place.

\section{Main Results}
Let $(\Omega, \{\mathscr{F}_t\}_{t \geq 0}, \mathscr{F}, P)$  be a complete filtered probability space satisfying the usual conditions, i.e. the filtration is  right continuous and $\mathscr{F}_0$ contains all $P$-null sets. Suppose that $T>0$ is a  fixed constant. We assume that  $w_t:=(w_t^i)_{i=1}^m$ is an ${\mathbb{R}}^m-$valued  standard Wiener process. Also let $(Z,\mathscr{Z}, \nu)$ be $\sigma$-finite measure space and $N(dt,dz)$ be Poisson random measure defined on $(Z,\mathscr{Z}, \nu)$ with intensity $\nu$ satisfying $\nu(Z)<\infty$. We set $\tilde N(dt,dz):=N(dt,dz)-\nu(dz)dt$. Let $b(x)$ and $\sigma(x)$ be  $\mathscr{B}(\mathbb R^d)$-measurable functions with values in ${\mathbb{R}}^d$ and ${\mathbb{R}}^{d\times m}$ respectively. Also, let $\gamma(x,z)$ be a $\mathscr{B}(\mathbb{R}^d)\otimes \mathscr{Z}$-measurable function with values in $\mathbb{R}^d$. Moreover, $b(x)$, $\sigma(x)$ and $\gamma(x,z)$ are assumed to be twice differentiable functions in $x\in \mathbb{R}^d$. For the purpose of this article,  we consider the following $d$-dimensional SDE,
\begin{align} \label{eq:sde}
x_t=& \xi + \int_{0}^t b(x_s)ds+\int_{0}^t \sigma(x_s)dw_s+ \int_{0}^t\int_{Z}\gamma(x_s,z)\tilde N(ds,dz),
\end{align}
almost surely for any $t \in [0,T] $, where $\xi$ is an $\mathscr{F}_{0}$-measurable  random variable in $\mathbb{R}^d$.
\begin{remark}
In the right hand side of the above equation \eqref{eq:sde}, we write $x_t$ instead of $x_{t-}$ in order to ease the notation. This does not cause any problem, since the compensator of the martingales driving the equation are continuous. This notational convenience shall be followed throughout this article.
\end{remark}

We make the following assumptions.

\begin{assumptionA} \label{as:sde:ini}
For a fixed $p\geq 2$, $E|\xi|^p < \infty$.
\end{assumptionA}

\begin{assumptionA} \label{as:sde:lipschitz}
There exists a constant $L>0$ such that
$$
(x-\bar{x})(b(x)-b(\bar{x}))\vee |\sigma(x)-\sigma(\bar{ x})|^2 \vee \int_Z|\gamma(x,z)-\gamma(\bar{x},z)|^2\nu(dz) \leq L|x-\bar{x}|^2
$$
for any $x, \bar{x} \in \mathbb{R}^d$.
\end{assumptionA}

\begin{assumptionA} \label{as:sde:lip:der}
There exists a constant $L>0$ such that
\begin{align*}
\big|\frac{\partial \sigma^{(i,j)}(x)}{\partial x^k}  -\frac{\partial \sigma^{(i,j)}(\bar{x})}{\partial x^k}\big|^2 \vee \int_Z \big|\frac{\partial \gamma^{i}(x,z)}{\partial x^k}  -\frac{\partial \gamma^{i}(\bar{x},z)}{\partial x^k}\big|^2\nu(dz) \leq L|x-\bar{x}|^2
\end{align*}
for any $x, \bar{x} \in \mathbb{R}^d$  and $i,k=1,\ldots,d$ and $j=1,\ldots,m$.
\end{assumptionA}

\begin{assumptionA} \label{as:sde:growth-p:jump}
There exists a constant $L>0$ such that
$$
\int_Z |\gamma(x,z)|^p\nu(dz) \leq L (1+|x|^p)
$$
for any $x \in \mathbb{R}^d$.
\end{assumptionA}

%\begin{assumptionA} \label{as:a:continuity}
%The function $b(x)$ is continuous in $x\in \mathbb{R}^d$.
%\end{assumptionA}
\begin{assumptionA} \label{as:sde:ploylip}
There exist  constants $L>0$ and $\chi>0$ such that
$$
\big|\frac{\partial b^i(x)}{\partial x^k}-\frac{\partial b^i(\bar{x})}{\partial x^k}\big| \leq L(1+|x|^{\chi}+|\bar{x}|^{\chi})|x-\bar{x}|
$$
for any $x,\bar{x}  \in \mathbb{R}^d$, $i,k=1,\ldots,d$.
\end{assumptionA}

The main results of this article on the convergence of the newly proposed tamed Milstein scheme, which is formally introduced in  Section \ref{sec:tamed:milstein}, follows.

\begin{thm} \label{thm:main:thm}
Let Assumptions A-1 to  A-5 hold with $p \geq 6(\chi+2)$. Then, the tamed Milstein scheme \eqref{eq:milstein} converges to the solution of SDE \eqref{eq:sde} in $\mathcal{L}^2$ with rate of convergence arbitrarily close to $1$, i.e.,
$$
\sup_{0 \leq t \leq T}E|x_t-x_t^n|^2 \leq K n^{-\frac{2}{2+\delta}-1}
$$
with $\delta \in (4/(p-2),1)$.
\end{thm}
\begin{remark}
If $E|\xi|^p < \infty$ for all $p>0$, then one observes that the $\delta>0$ appearing in Theorem \ref{thm:main:thm} can be as small as one wishes and  the obtained rate is arbitrarily close to $1.0$ (which agrees with the classical results, see \cite{platen2010}).
\end{remark}
The proof of Theorem \ref{thm:main:thm} can be found in Section \ref{sec:rate}. When one considers only the continuous case, then the rate is improved as stated in the theorems below.
%-------------------------------------------------------------
\begin{thm} \label{thm:main:theorem}
%-------------------------------------------------------------
Let Assumptions A-1 to
%, A-4, A-2, A-5,A-3 , A-\ref{as:a:continuity} and
A-5 ($\gamma\equiv0$) with $p \geq 6(\chi+2)$ hold. Then, the tamed Milstein scheme \eqref{eq:milstein:continuous} converges to the solution of SDE \eqref{eq:sde:continuous} in $\mathcal{L}^q$ with rate $1$, i.e.
$$
\sup_{0 \leq t \leq T}E|x_t-x_t^n|^q \leq K n^{-q}
$$
where $0 <q \leq \max(2, p/(3\chi+6)-2)$.
\end{thm}

%-------------------------------------------------------------
\begin{thm} \label{thm:main:theorem:sup:inside}
%-------------------------------------------------------------
Let Assumptions A-1 to
%, A-4, A-2, A-5,  A-3 , A-\ref{as:a:continuity} and
A-5 ($\gamma\equiv0$) with $p \geq 6(\chi+2)$ hold. Then, the tamed Milstein scheme \eqref{eq:milstein:continuous} converges to the solution of SDE \eqref{eq:sde:continuous} in $\mathcal{L}^q$ with rate $1$, i.e.
$$
E\sup_{0 \leq t \leq T}|x_t-x_t^n|^q \leq K n^{-q}
$$
where $0 < q < \max(2, p/(3\chi+6)-2)$.
\end{thm}
\begin{remark}
If one assumes that $E|\xi|^p<\infty$ is true for all $p>0$ as assumed in \cite{wanga2013}, then  Theorem \ref{thm:main:theorem:sup:inside} holds for any $q>0$, which is in agreement with the findings of \cite{wanga2013}.
\end{remark}
The proofs of Theorems [\ref{thm:main:theorem}, \ref{thm:main:theorem:sup:inside}] as stated above are given in Section \ref{sec:con}.
%-----------------------------------------------------
\subsection{Preliminary Observations}
%-------------------------------------------------------
The proof of the following lemma can be found in \cite{Mik}.
\begin{lem} \label{lem:BGD:jump}
Let $r \geq 2$. There exists a constant $K$, depending only on $r$, such that for every real-valued, $\mathscr{B}([0,T]) \otimes \mathscr Z-$measurable function $g$ satisfying
$$
\int_0^T\int_Z |g_t(z)|^2\nu(dz) dt < \infty
$$
the following estimate holds,
\begin{align*}
E\sup_{0 \leq t \leq T} \Big|\int_0^t \int_Z & g_s(z)\tilde{N}(ds,dz)\Big|^r \leq K E\Big(\int_0^T \int_Z|g_t(z)|^2\nu(dz)dt\Big)^{r/2}
\\
&+K E\int_0^T \int_Z |g_t(z)|^r \nu(dz)dt.
\end{align*}
It is  known  that if $1\leq r \leq 2$, then the second term of the right hand side can be dropped.
\end{lem}

\begin{remark} \label{as:sde:growth}
Due to Assumption A-2, there exits a constant $L>0$ such that
$$
xb(x)\vee |\sigma(x)|^2 \vee \int_Z|\gamma(x,z)|^2 \nu(dz) \leq L(1+|x|^2)
$$
for any $x \in \mathbb{R}^d$.
\end{remark}

The proof of the following lemma can be found in \cite{konstantinos2014} and \cite{situ2005}.
\begin{lem} \label{lem:sde:momentbound}
Let $b(x)$ be a continuous function in $x \in \mathbb{R}^d$. Also, suppose that Assumptions A-1, A-2  and A-4 are satisfied.  Then, there exists a unique solution to SDE \eqref{eq:sde} and moreover
$$
E\sup_{0 \leq t \leq T}|x_t|^p \leq K,
$$
where $K:=K(L,T,p,E|\xi|^p)$.
\end{lem}

\begin{remark} \label{rem:der:bounded}
Due to Assumptions A-2 and A-3, there exists a constant $L>0$ such that
$$
\Big|\frac{\partial \sigma^{(i,j)}(x)}{\partial x^k}\Big| \vee \Big|\frac{\partial^2 \sigma^{(i,j)}(x)}{\partial x^k \partial x^u}\Big| \vee \int_Z\Big|\frac{\partial \gamma^i(x,z)}{\partial x^k}\Big|\nu(dz) \vee \int_Z\Big|\frac{\partial^2 \gamma^i(x,z)}{\partial x^k \partial x^u}\Big|\nu(dz)\leq L
$$
for any $x \in \mathbb{R}^d$, $i,k,u=1,\ldots,d$ and $j=1,\ldots,m$.
\end{remark}

\begin{remark} \label{as:sde:poly:Der}
Due to Assumption A-5, there exist constants $L>0$ and $\chi>0$ such that
$$
\Big|\frac{\partial b^i(x)}{\partial x^k}\Big| \leq  L(1+|x|^{\chi+1}) \quad \mbox{and} \quad \Big|\frac{\partial^2 b^i(x)}{\partial x^j \partial x^k}\Big|\leq  L(1+|x|^{\chi})
$$
for any $x \in \mathbb{R}^d$ and $i,j, k=1,\ldots,d$.
\end{remark}

\begin{remark} \label{as:sde:poly}
Due to Remark \ref{as:sde:poly:Der}, there exist constants $L>0$ and $\chi>0$ such that
$$
|b(x)-b(\bar{x})| \leq L(1+|x|^{\chi+1}+|\bar{x}|^{\chi+1})|x-\bar{x}|
$$
which further implies that
$$
|b(x)| \leq L(1+|x|^{\chi+2})
$$
for any $x,\bar{x}  \in \mathbb{R}^d$.
\end{remark}
\begin{remark}
One notices that Assumption 3.1 in \cite{wanga2013} implies Assumption A-3 of this article due to the fact that, for a function $f:\mathbb{R}^d \to \mathbb{R}^d$, the condition $\sup_{h_1,h_2 \in \mathbb{R}^d,|h_1|\leq 1, |h_1|\leq 1} |\sum_{i,j=1}^d \frac{\partial^2 f(x)}{\partial x^i \partial x^j}h_1^ih_2^j| \leq K$ implies that $|\frac{\partial^2 f^k(x)}{\partial x^i \partial x^j}|\leq K$ for every $i,j,k=1,\ldots,d$. One can make similar conclusion for Assumption  A-5.  Thus, the assumptions under consideration in this article are not stronger than those in \cite{wanga2013}.
\end{remark}

\begin{remark}
For the practical implementation of the schemes \eqref{eq:milstein} and \eqref{eq:milstein:continuous} which have been introduced in this article, one requires commutative conditions on the coefficients, which have been discussed \cite{kloeden1999} and \cite{platen2010}.
\end{remark}

\begin{remark}
No comparison is made here with \cite{zhang2014} as the main theorem for convergence in that article i.e. Theorem 2.2. p. 4 reproduces the corresponding theorem from \cite{Tretyakov-Zhang}, i.e. Theorem 2.1, p. 3137, which is known to contain an imprecise statement about the moments requirement for the $\mathcal{L}^{2p}$-convergence of their proposed numerical scheme.
\end{remark}

The proof of the following lemma can be found in  \cite{yor}.
\begin{lem} \label{lem:yor}
Let $X$ be a positive, adapted right continuous process and $A$ be a continuous increasing process such that
$$
E[X_\tau| \mathscr{F}_0] \leq E[A_\tau | \mathscr{F}_0]
$$
for any bounded stopping time $\tau$. Then for any $\varsigma \in (0,1)$,
$$
E\big[\big( X_\infty^*\big)^\varsigma\big] \leq \frac{2-\varsigma}{1-\varsigma} E\big[\big( A_\infty\big)^\varsigma\big]
$$
where $X_\infty^*=\sup_{t\geq 0} X_t$.
\end{lem}

%-----------------------------------------------------------------
\section{Tamed Milstein Scheme} \label{sec:tamed:milstein}
%-----------------------------------------------------------------
For every $n \in \mathbb{N}$, we propose the following form for taming the super-linearly growing drift coefficient of  SDE \eqref{eq:sde},
\begin{equation} \label{eq:an:tilde}
\tilde b^n(x)= \frac{b(x)}{1+n^{-\theta}|b(x)|^{2 \theta }}
\end{equation}
for any $\theta \geq \frac{1}{2}$,  and $x \in \mathbb{R}^d$. One notes that $\tilde b^n$ is an $\mathbb{R}^d$ -valued function and its $i$-th element is denoted by $\tilde b^{n,i}$ for $i=1,\ldots,d$.
\begin{remark}
One observes that when $\theta=\frac{1}{2}$, then we obtain  tamed Euler schemes similar to those discussed in \cite{konstantinos2014} and \cite{sabanis2013}. In this article, we  discuss a tamed Milstein scheme of SDE \eqref{eq:sde} by taking $\theta=1$. It is important to note that by assigning different values to $\theta$ and appropriately including multiple stochastic integrals in the scheme, one could write a tamed scheme and then  perform  calculations similar to the methodology developed in this article to achieve an order $\theta$. For the purpose of this article, we only discuss the case when $\theta=1$ i.e. the tamed Milstein scheme of L\'{e}vy driven SDE with super-linear drift coefficient.  From now onward, throughout this article, we take $\theta=1$.
\end{remark}

\begin{remark} \label{as:scheme:a^n}
For every $n \in \mathbb{N}$, equation \eqref{eq:an:tilde} implies
$$
|\tilde{b}^n(x)| \leq \min(n^{1/2}, |b(x)|)
$$
for any $x \in \mathbb{R}^d$ which is in the same spirit as in \cite{konstantinos2014} and \cite{sabanis2013}.
\end{remark}
For every $n \in \mathbb{N}$, we propose the following  tamed Milstein scheme,
\begin{align}  \label{eq:milstein}
x_t^n =\xi + \int_{0}^t \tilde{b}^n(x_{\kappa(n,s)}^n)ds+\int_{0}^t \tilde{\sigma}(x_{\kappa(n,s)}^n)dw_s+\int_{0}^t\int_Z \tilde{\gamma}(x_{\kappa(n,s)}^n, z_2)\tilde N(ds,dz_2)
\end{align}
almost surely for any $t \in [0,T]$ where  $\kappa(n,t):=\lfloor nt \rfloor/n$.  The drift coefficient $\tilde{b}^n$ of scheme \eqref{eq:milstein} is defined in equation \eqref{eq:an:tilde} whereas $\tilde{\sigma}$ and  $\tilde{\gamma}$ are diffusion and jump coefficients of the scheme \eqref{eq:milstein} which are defined below. For every $n \in \mathbb{N}$ and  $ s \in [0,T]$, the diffusion coefficient $\tilde{\sigma}(x_{\kappa(n,s)}^n)$ of scheme \eqref{eq:milstein} is given by
\begin{equation}
\tilde{\sigma}(x_{\kappa(n,s)}^n):=\sigma(x_{\kappa(n,s)}^n)+\sigma_{1}(x_{\kappa(n,s)}^n)+\sigma_{2}(x_{\kappa(n,s)}^n)+\sigma_{3}(x_{\kappa(n,s)}^n) \label{eq:bn:tilde}
\end{equation}
where $\sigma_{1}(x_{\kappa(n,s)}^n)$, $\sigma_{2}(x_{\kappa(n,s)}^n)$  and $\sigma_{3}(x_{\kappa(n,s)}^n)$ are $d\times m$ matrices with their $(i,k)$-th elements given by
\begin{align}
{\sigma}^{(i,k)}_{1}(x_{\kappa(n,s)}^n):=& \sum_{j=1}^m\sum_{u=1}^d \int_{\kappa(n,s)}^s  \sigma^{(u,j)}(x_{\kappa(n,r)}^n)\frac{\partial \sigma^{(i,k)} (x_{\kappa(n,r)}^n)}{\partial x^u} dw_r^j \label{eq:sigma:1}
\\
\sigma^{(i,k)}_{2}(x_{\kappa(n,s)}^n):=& \sum_{u=1}^d \int_{\kappa(n, s)}^{s} \int_Z \frac{\partial \sigma^{(i,k)}(x_{\kappa(n,r)}^n)}{\partial x^u} \gamma^u(x_{\kappa(n,r)}^n,z_1) \tilde N(dr, dz_1) \notag
\\
\sigma^{(i,k)}_{3}(x_{\kappa(n,s)}^n):=& \int_{\kappa(n,s)}^s \int_{Z}\Big(\sigma^{(i,k)}\big(x_{\kappa(n,r)}^n+\gamma(x_{\kappa(n,r)}^n, z_1)\big)-\sigma^{(i,k)}\big(x_{\kappa(n,r)}^n\big) \notag
\\
& \hspace{1cm}-\sum_{u=1}^d  \frac{\partial \sigma^{(i,k)}(x_{\kappa(n,r)}^n)}{\partial x^u} \gamma^u(x_{\kappa(n,r)}^n,z_1)\Big)N(dr,dz_1) \notag
\end{align}
respectively, for every $i=1,\ldots,d$, $k=1,\ldots,m$. Similarly, for every $n \in \mathbb{N}$, $s\in [0,T]$ and $z_2 \in Z$, the jump coefficient $\tilde \gamma(x_{\kappa(n,s)}^n, z_2)$ of scheme \eqref{eq:milstein} is given by
\begin{equation}
\tilde{\gamma}(x_{\kappa(n,s)}^n, z_2):=\gamma(x_{\kappa(n,s)}^n, z_2)+\gamma_1(x_{\kappa(n,s)}^n, z_2)+\gamma_2(x_{\kappa(n,s)}^n, z_2)+\gamma_3(x_{\kappa(n,s)}^n, z_2) \label{eq:cn:tilde}
\end{equation}
where  $\gamma_1(x_{\kappa(n,s)}^n, z_2)$, $\gamma_2(x_{\kappa(n,s)}^n, z_2)$ and $\gamma_3(x_{\kappa(n,s)}^n, z_2)$   are $d$-dimensional vectors with their $i$th elements given by,
\begin{align*}
\gamma_1^{i}(x_{\kappa(n,s)}^n, z_2):= & \sum_{j=1}^m \sum_{u=1}^d  \int_{\kappa(n,s)}^s \frac{\partial \gamma^{i}(x_{\kappa(n,r)}^n, z_2)}{\partial x^u} \sigma^{(u,j)}(x_{\kappa(n,r)}^n)dw_r^j
\\
\gamma^{i}_{2}(x_{\kappa(n,s)}^n,z_2):=& \sum_{u=1}^d \int_{\kappa(n, s)}^{s} \int_Z \frac{\partial \gamma^{i}(x_{\kappa(n,r)}^n, z_2)}{\partial x^u} \gamma^u(x_{\kappa(n,r)}^n,z_1) \tilde N(dr, dz_1)
\\
\gamma_3^{i}(x_{\kappa(n,s)}^n, z_2):= & \int_{\kappa(n,s)}^s \int_{Z} \Big( \gamma^{i}(x_{\kappa(n,r)}^n+\gamma(x_{\kappa(n,r)}^n, z_1), z_2)-\gamma^{i}(x_{\kappa(n,r)}^n, z_2)
\\
& \hspace{1cm}-\sum_{u=1}^d \frac{\partial \gamma^{i}(x_{\kappa(n,r)}^n,z_2)}{\partial x^u} \gamma^u(x_{\kappa(n,r)}^n,z_1)
 \Big) N(dr,dz_1)
\end{align*}
respectively, for every $i=1,\ldots,d$.

%-----------------------------------------------------------------
\subsection{Moment Bounds}
%-----------------------------------------------------------------
One observes that due to Remark \ref{as:scheme:a^n}, for a fixed $n \in \mathbb{N}$, the drift coefficient of the scheme \eqref{eq:milstein} is bounded. Also, the diffusion and jump coefficients grow linearly due to Remark \ref{as:sde:growth}. Hence, one can refer to \cite{platen2010} to conclude that Assumptions A-1,  A-2 and A-4 along with equation \eqref{eq:an:tilde} imply that, for a fixed $n \in \mathbb{N}$,
\begin{align} \label{eq:scheme:mb:finite}
E\sup_{0 \leq t \leq T}|x_t^n|^p < \infty.
\end{align}
Clearly, one can not claim at this stage that the bound is independent of $n$. However, it guarantees that all local martingales appearing henceforth are in fact true martingales and thus the use of stopping time arguments is avoided. Before proving the moment bounds of the scheme \eqref{eq:milstein} in Lemma \ref{lem:scheme:momentbound}, one requires to show the following     lemmas.

%-----------------------------------------------------------------
%                Lemma b1:rate
%-----------------------------------------------------------------
\begin{lem} \label{lem:b1:rate}
Let Assumption A-2 hold, then
$$
E|\sigma_{1}(x_{\kappa(n,s)}^n)|^p \leq Kn^{-\frac{p}{2}}  \big(1+E|x_{\kappa(n,s)}^n|^p\big),
$$
for every $n \in \mathbb{N}$ and $s \in [0,T]$, where  $K:=K(L,p, m,d)$.
\end{lem}
\begin{proof}
First one writes,
\begin{align*}
E|\sigma_{1}(x_{\kappa(n,s)}^n)|^p  = & E\Big(\sum_{k=1}^m \sum_{i=1}^d |\sigma^{(i,k)}_{1}(x_{\kappa(n,s)}^n)|^2\Big)^\frac{p}{2}
%\\
%&=KE\Big(\sum_{k=1}^m \sum_{i=1}^d \Big|\sum_{j=1}^m\sum_{u=1}^d \int_{\kappa(n,s)}^s  \sigma^{(u,j)}(x_{\kappa(n,r)}^n)\frac{\partial \sigma^{(i,k)} (x_{\kappa(n,r)}^n)}{\partial x^u} dw_r^j\Big|^2\Big)^\frac{p}{2}
%\\
%&\leq KE\Big(\sum_{k,j=1}^m \sum_{i,u=1}^d \Big|\int_{\kappa(n,s)}^s  \sigma^{(u,j)}(x_{\kappa(n,r)}^n)\frac{\partial \sigma^{(i,k)} (x_{\kappa(n,r)}^n)}{\partial x^u} dw_r^j\Big|^2\Big)^\frac{p}{2}
\\
\leq & K\sum_{k,j=1}^m \sum_{i,u=1}^d E\Big|\int_{\kappa(n,s)}^s  \sigma^{(u,j)}(x_{\kappa(n,r)}^n)\frac{\partial \sigma^{(i,k)} (x_{\kappa(n,r)}^n)}{\partial x^u} dw_r^j\Big|^p
\end{align*}
which on the application of an elementary inequality for stochastic integrals and H\"{o}lder's  inequality yields
\begin{align*}
E|\sigma_{1}(x_{\kappa(n,s)}^n)&|^p  \leq K\sum_{k,j=1}^m \sum_{i,u=1}^d E\Big(\int_{\kappa(n,s)}^s  \Big|\sigma^{(u,j)}(x_{\kappa(n,r)}^n)\frac{\partial \sigma^{(i,k)} (x_{\kappa(n,r)}^n)}{\partial x^u}\Big|^2 dr\Big)^\frac{p}{2}
\\
& \leq Kn^{-(\frac{p}{2}-1)}\sum_{k,j=1}^m \sum_{i, u=1}^d E\int_{\kappa(n,s)}^s  |\sigma^{(u,j)}(x_{\kappa(n,r)}^n)|^p\Big|\frac{\partial \sigma^{(i,k)} (x_{\kappa(n,r)}^n)}{\partial x^u}\Big|^p dr
\end{align*}
and then due to Remarks [\ref{as:sde:growth},  \ref{rem:der:bounded}], one obtains
\begin{align}
E|\sigma_{1}(x_{\kappa(n,s)}^n)|^p & \leq Kn^{-(\frac{p}{2}-1)} E\int_{\kappa(n,s)}^s  \big(1+|x_{\kappa(n,r)}^n|^p\big) dr = Kn^{-\frac{p}{2}}  \big(1+E|x_{\kappa(n,s)}^n|^p\big) \notag
\end{align}
for every $n \in \mathbb{N}$ and $s \in [0,T]$. This completes the proof.
\end{proof}
%-----------------------------------------------------------------
%                Lemma b2:rate
%-----------------------------------------------------------------
\begin{lem} \label{lem:b2:rate}
Let Assumptions A-2 and A-4 hold, then
$$
E|\sigma_{2}(x_{\kappa(n,s)}^n)|^p \leq  K n^{-1} \big(1+E|x_{\kappa(n,s)}^n|^p\big),
$$
for every $n \in \mathbb{N}$ and $s \in [0,T]$, where  $K:=K(L,p, m,d)$.
\end{lem}
\begin{proof} One notes that,
\begin{align*}
E|\sigma_{2}(&x_{\kappa(n,s)}^n)|^p  = E\Big(\sum_{i=1}^d\sum_{k=1}^m |\sigma^{(i,k)}_{2}(x_{\kappa(n,s)}^n)|^2\Big)^\frac{p}{2} \leq K\sum_{i=1}^d\sum_{k=1}^m E|\sigma^{(i,k)}_{2}(x_{\kappa(n,s)}^n)|^p
%\\
%&\leq K\sum_{i=1}^d\sum_{k=1}^m E\Big| \sum_{u=1}^d \int_{\kappa(n, s)}^{s} \int_Z \frac{\partial \sigma^{(i,k)}(x_{\kappa(n,r)}^n)}{\partial x^u} \gamma^u(x_{\kappa(n,r)}^n,z_1) \tilde N(dr, dz_1)\Big|^p
\\
&\leq K\sum_{i=1}^d\sum_{k=1}^m  \sum_{u=1}^d E\Big|\int_{\kappa(n, s)}^{s} \int_Z \frac{\partial \sigma^{(i,k)}(x_{\kappa(n,r)}^n)}{\partial x^u} \gamma^u(x_{\kappa(n,r)}^n,z_1) \tilde N(dr, dz_1)\Big|^p
\end{align*}
which due to Lemma \ref{lem:BGD:jump} yields
\begin{align*}
& \qquad \qquad \qquad \qquad \qquad \qquad E|\sigma_{2}  (x_{\kappa(n,s)}^n)|^p
\\
&\leq  K \sum_{i=1}^d\sum_{k=1}^m  \sum_{u=1}^d E\Big(\int_{\kappa(n, s)}^{s} \int_Z \Big|\frac{\partial \sigma^{(i,k)}(x_{\kappa(n,r)}^n)}{\partial x^u} \gamma^u(x_{\kappa(n,r)}^n,z_1)\Big|^2 \nu(dz_1)dr\Big)^\frac{p}{2}
\\
& +K \sum_{i=1}^d\sum_{k=1}^m  \sum_{u=1}^d E\int_{\kappa(n, s)}^{s} \int_Z \Big|\frac{\partial \sigma^{(i,k)}(x_{\kappa(n,r)}^n)}{\partial x^u} \gamma^u(x_{\kappa(n,r)}^n,z_1)\Big|^p \nu(dz_1)dr
\end{align*}
and then on using Remarks [\ref{as:sde:growth},  \ref{rem:der:bounded}] and Assumption A-4, one obtains
\begin{align*}
E|\sigma_{2}(x_{\kappa(n,s)}^n)|^p & \leq  K n^{-\frac{p}{2}} \big(1+E|x_{\kappa(n,s)}^n|^p\big)+ K n^{-1} \big(1+E|x_{\kappa(n,s)}^n|^p\big)
\end{align*}
for every $n \in \mathbb{N}$ and $s \in [0,T]$. This finishes the proof.
\end{proof}
%-----------------------------------------------------------------
%                Lemma b3:rate
%-----------------------------------------------------------------
\begin{lem} \label{lem:b3:rate}
Let Assumptions A-2 to A-4 hold, then
$$
E|\sigma_{3}(x_{\kappa(n,s)}^n)|^p \leq K n^{-1} \big(1+E|x_{\kappa(n,s)}^n|^p),
$$
for every $n \in \mathbb{N}$ and $s \in [0,T]$, where $K:=K(L,p, m,d)$.
\end{lem}
\begin{proof}
By using Assumption A-2 and Remark \ref{rem:der:bounded}, one obtains
\begin{align*}
E&|\sigma_{3}(x_{\kappa(n,s)}^n)|^p= E\Big(\sum_{i=1}^{d}\sum_{k=1}^{m}|\sigma^{(i,k)}_{3}(x_{\kappa(n,s)}^n)|^2\Big)^\frac{p}{2} %\leq K E\sum_{i=1}^{d}\sum_{k=1}^{m}|\sigma^{(i,k)}_{3}(x_{\kappa(n,s)}^n)|^p
%\\
%=& K E\sum_{i=1}^{d}\sum_{k=1}^{m}\Big|\int_{\kappa(n,s)}^s \int_{Z}\Big\{\sigma^{(i,k)}\big(x_{\kappa(n,r)}^n+\gamma(x_{\kappa(n,r)}^n, z_1)\big)-\sigma^{(i,k)}\big(x_{\kappa(n,r)}^n\big)
%\\
%&\hspace{4cm}-\sum_{u=1}^d  \frac{\partial \sigma^{(i,k)}(x_{\kappa(n,r)}^n)}{\partial x^u} \gamma^u(x_{\kappa(n,r)}^n,z_1)\Big\}N(dr,dz_1)\Big|^p
\\
\leq & K E\sum_{i=1}^{d}\sum_{k=1}^{m}\Big|\int_{\kappa(n,s)}^s \int_{Z}\Big(|\sigma^{(i,k)}\big(x_{\kappa(n,r)}^n+\gamma(x_{\kappa(n,r)}^n, z_1)\big)-\sigma^{(i,k)}\big(x_{\kappa(n,r)}^n\big)|
\\
&\hspace{3cm}+\sum_{u=1}^d  \Big|\frac{\partial \sigma^{(i,k)}(x_{\kappa(n,r)}^n)}{\partial x^u}\Big| |\gamma^u(x_{\kappa(n,r)}^n,z_1)|\Big)N(dr,dz_1)\Big|^p
\\
\leq & K E\Big(\int_{\kappa(n,s)}^s \int_{Z}|\gamma(x_{\kappa(n,r)}^n, z_1)|N(dr,dz_1)\Big)^p
\end{align*}
which due to Lemma \ref{lem:BGD:jump} gives
\begin{align*}
E|\sigma_{3}(x_{\kappa(n,s)}^n)|^p \leq & KE\Big(\int_{\kappa(n,s)}^s \int_{Z}|\gamma(x_{\kappa(n,r)}^n, z_1)| \tilde N(dr,dz_1)\Big)^p
\\
&+ K E\Big(\int_{\kappa(n,s)}^s \int_{Z}|\gamma(x_{\kappa(n,r)}^n, z_1)| \nu(dz_1)dr\Big)^p
\\
\leq & KE\Big(\int_{\kappa(n,s)}^s \int_{Z}|\gamma(x_{\kappa(n,r)}^n, z_1)|^2 \nu(dz_1)dr \Big)^\frac{p}{2}
\\
& + K E\int_{\kappa(n,s)}^s \int_{Z}|\gamma(x_{\kappa(n,r)}^n, z_1)|^p \nu(dz_1)dr
\\
& \hspace{2cm} + K E\Big(\int_{\kappa(n,s)}^s \int_{Z}|\gamma(x_{\kappa(n,r)}^n, z_1)| \nu(dz_1)dr\Big)^p
\end{align*}
and then on applying Remark \ref{as:sde:growth} and Assumption  A-4, one obtains
\begin{align*}
E|\sigma_{3}(x_{\kappa(n,s)}^n)|^p & \leq K n^{-\frac{p}{2}} \big(1+E|x_{\kappa(n,s)}^n|^p)+K n^{-1} \big(1+E|x_{\kappa(n,s)}^n|^p)
\\
&+K n^{-p} \big(1+E|x_{\kappa(n,s)}^n|^p)
\end{align*}
for every $n \in \mathbb{N}$ and $s \in [0,T]$.  This completes the proof.
\end{proof}
%-----------------------------------------------------------------
%                Lemma tilde b:bound
%-----------------------------------------------------------------
The following corollary is a consequence of Lemmas [\ref{lem:b1:rate}, \ref{lem:b2:rate}, \ref{lem:b3:rate}].

\begin{cor} \label{lem:tildeb:momentbound}
Let Assumptions A-2  to A-4 hold, then
$$
E| \tilde{\sigma}(x_{\kappa(n,s)}^n)|^p \leq  K(1+E|x_{\kappa(n,s)}^n|^p),
$$
for every $n \in \mathbb{N}$ and $s \in [0,T]$, where $K:=K(L,p, m,d)$.
\end{cor}
%-----------------------------------------------------------------
%                Lemma c1:rate
%-----------------------------------------------------------------
\begin{lem} \label{lem:c1:rate}
Let Assumptions A-2 and A-4 hold, then
$$
E\int_Z|\gamma_1(x_{\kappa(n,s)}^n, z_2)|^p \nu(dz_2) \leq K n^{-\frac{p}{2}} (1+E|x_{\kappa(n,s)}^n|^{p}),
$$
for every $n \in \mathbb{N}$ and $s \in [0,T]$, where $K:=K(L,p, m,d)$.
\end{lem}
\begin{proof}
One observes that
\begin{align*}
E\int_Z|\gamma_1(&x_{\kappa(n,s)}^n, z_2)|^p \nu(dz_2)=E\int_Z\Big(\sum_{i=1}^d|\gamma_1^{i}(x_{\kappa(n,s)}^n, z_2)|^2\Big)^\frac{p}{2} \nu(dz_2)
%\\
%\leq & K \sum_{i=1}^dE\int_Z|\gamma_1^{i}(x_{\kappa(n,s)}^n, z_2)|^p \nu(dz_2)
%\\
%\leq & K \sum_{i=1}^dE\int_Z\Big|\sum_{j=1}^m \sum_{u=1}^d  \int_{\kappa(n,s)}^s \frac{\partial \gamma^{i}(x_{\kappa(n,r)}^n, z_2)}{\partial x^u} \sigma^{(u,j)}(x_{\kappa(n,r)}^n)dw_r^j\Big|^p \nu(dz_2)
%\\
%\leq & K \sum_{i=1}^dE\int_Z\sum_{j=1}^m \sum_{u=1}^d \Big| \int_{\kappa(n,s)}^s \frac{\partial \gamma^{i}(x_{\kappa(n,r)}^n, z_2)}{\partial x^u} \sigma^{(u,j)}(x_{\kappa(n,r)}^n)dw_r^j\Big|^p \nu(dz_2)
\\
\leq & K \sum_{i=1}^d\int_Z\sum_{j=1}^m \sum_{u=1}^d E\Big| \int_{\kappa(n,s)}^s \frac{\partial \gamma^{i}(x_{\kappa(n,r)}^n, z_2)}{\partial x^u} \sigma^{(u,j)}(x_{\kappa(n,r)}^n)dw_r^j\Big|^p \nu(dz_2)
\end{align*}
which on using an elementary inequality for stochastic integrals  and H\"{o}lder's inequality implies
\begin{align*}
& E\int_Z|\gamma_1(x_{\kappa(n,s)}^n, z_2)|^p \nu(dz_2)
\\
& \leq  K \sum_{i,u=1}^d\sum_{j=1}^m \int_Z E\Big( \int_{\kappa(n,s)}^s \Big|\frac{\partial \gamma^{i}(x_{\kappa(n,r)}^n, z_2)}{\partial x^u} \sigma^{(u,j)}(x_{\kappa(n,r)}^n)\Big|^2dr\Big)^\frac{p}{2} \nu(dz_2)
\\
& \leq  K n^{-(\frac{p}{2}-1)}\sum_{i,u=1}^d\sum_{j=1}^m \int_Z E\int_{\kappa(n,s)}^s \Big|\frac{\partial \gamma^{i}(x_{\kappa(n,r)}^n, z_2)}{\partial x^u}\Big|^{p} |\sigma^{(u,j)}(x_{\kappa(n,r)}^n)|^{p} dr \nu(dz_2)
\end{align*}
and then due to Remarks [\ref{as:sde:growth},  \ref{rem:der:bounded}], one obtains
\begin{align}
E\int_Z|\gamma_1(x_{\kappa(n,s)}^n, z_2)|^p \nu(dz_2) \leq  K n^{-\frac{p}{2}} (1+E|x_{\kappa(n,s)}^n|^p) \notag
\end{align}
for every $n \in \mathbb{N}$ and $s \in [0,T]$. Hence the proof follows.
\end{proof}
%-----------------------------------------------------------------
%                Lemma c2:rate
%-----------------------------------------------------------------
\begin{lem} \label{lem:c2:rate}
Let Assumptions A-2 and A-4 hold, then
$$
E\int_Z|\gamma_2(x_{\kappa(n,s)}^n, z_2)|^p \nu(dz_2) \leq K n^{-1} (1+E|x_{\kappa(n,s)}^n|^p ),
$$
for every $n \in \mathbb{N}$ and $s \in [0,T]$, where $K:=K(L,p, m,d)$.
\end{lem}
\begin{proof}
One observes that
\begin{align*}
& E\int_Z|\gamma_2(x_{\kappa(n,s)}^n, z_2)|^p  \nu(dz_2)= E\int_Z\Big(\sum_{i=1}^d|\gamma_2^{i}(x_{\kappa(n,s)}^n, z_2)|^2\Big)^\frac{p}{2} \nu(dz_2)
%\\
%&\leq K E\int_Z\sum_{i=1}^d|\gamma_2^{i}(x_{\kappa(n,s)}^n, z_2)|^p \nu(dz_2)
%\\
%& \leq K E\int_Z\sum_{i=1}^d \Big|\sum_{u=1}^d \int_{\kappa(n, s)}^{s} \int_Z \frac{\partial \gamma^{i}(x_{\kappa(n,r)}^n, z_2)}{\partial x^u} \gamma^u(x_{\kappa(n,r)}^n,z_1) \tilde N(dr, dz_1)\Big|^p \nu(dz_2)
\\
& \leq K \int_Z\sum_{i=1}^d \sum_{u=1}^d E\Big|\int_{\kappa(n, s)}^{s} \int_Z \frac{\partial \gamma^{i}(x_{\kappa(n,r)}^n, z_2)}{\partial x^u} \gamma^u(x_{\kappa(n,r)}^n,z_1) \tilde N(dr, dz_1)\Big|^p \nu(dz_2)
\end{align*}
which due to Lemma \ref{lem:BGD:jump} yields
\begin{align*}
& \hspace{3cm} E\int_Z|\gamma_2(x_{\kappa(n,s)}^n, z_2)|^p  \nu(dz_2)
\\
& \leq K \int_Z\sum_{i,u=1}^d E\Big(\int_{\kappa(n, s)}^{s} \int_Z \Big| \frac{\partial \gamma^{i}(x_{\kappa(n,r)}^n, z_2)}{\partial x^u} \gamma^u(x_{\kappa(n,r)}^n,z_1)\Big|^2  \nu(dz_1) dr \Big)^{\frac{p}{2}} \nu(dz_2)
\\
& + K \int_Z\sum_{i,u=1}^d  E\int_{\kappa(n, s)}^{s} \int_Z \Big| \frac{\partial \gamma^{i}(x_{\kappa(n,r)}^n, z_2)}{\partial x^u} \gamma^u(x_{\kappa(n,r)}^n,z_1)\Big|^p  \nu(dz_1) dr  \nu(dz_2)
\end{align*}
and then on using Remarks [\ref{as:sde:growth}, \ref{rem:der:bounded}] and Assumption A-4, one obtains,
\begin{align*}
E\int_Z|\gamma_2(x_{\kappa(n,s)}^n, z_2)|^p \nu(dz_2) & \leq K n^{-\frac{p}{2}} (1+E|x_{\kappa(n,s)}^n|^p ) +K n^{-1} (1+E|x_{\kappa(n,s)}^n|^p)
\end{align*}
for every $n \in \mathbb{N}$ and $s \in [0,T]$. Thus the proof finishes.
\end{proof}
%-----------------------------------------------------------------
%                Lemma c3:rate
%-----------------------------------------------------------------
\begin{lem} \label{lem:c3:rate}
Let Assumptions A-2 to A-4 hold, then
$$
E\int_Z|\gamma_3(x_{\kappa(n,s)}^n, z_2)|^p \nu(dz_2) \leq K n^{-1} \big(1+E|x_{\kappa(n,s)}^n|^p \big),
$$
for every $n \in \mathbb{N}$ and $s \in [0,T]$, where $K:=K(L,p, m,d)$.
\end{lem}
\begin{proof}
On using Assumption A-2 and Remark \ref{rem:der:bounded}, one writes
\begin{align*}
&E\int_Z|\gamma_3(x_{\kappa(n,s)}^n, z_2)|^p \nu(dz_2) \leq K E\int_Z\Big(\sum_{i=1}^d|\gamma_3^{i}(x_{\kappa(n,s)}^n, z_2)|^2\Big)^\frac{p}{2} \nu(dz_2)
\\
&= K E\int_Z\sum_{i=1}^d\Big| \int_{\kappa(n,s)}^s \int_{Z} \Big( \gamma^{i}(x_{\kappa(n,r)}^n+\gamma(x_{\kappa(n,r)}^n, z_1), z_2)-\gamma^{i}(x_{\kappa(n,r)}^n, z_2)
\\
& \hspace{2cm}-\sum_{u=1}^d \frac{\partial \gamma^{i}(x_{\kappa(n,r)}^n,z_2)}{\partial x^u} \gamma^u(x_{\kappa(n,r)}^n,z_1)
 \Big) N(dr,dz_1)\Big|^p \nu(dz_2)
\\
&\leq K E\int_Z\sum_{i=1}^d\Big( \int_{\kappa(n,s)}^s \int_{Z} \Big(| \gamma^{i}(x_{\kappa(n,r)}^n+\gamma(x_{\kappa(n,r)}^n, z_1), z_2)-\gamma^{i}(x_{\kappa(n,r)}^n, z_2)|
\\
& \hspace{2cm}+\sum_{u=1}^d \big|\frac{\partial \gamma^{i}(x_{\kappa(n,r)}^n,z_2)}{\partial x^u}\big| |\gamma^u(x_{\kappa(n,r)}^n,z_1)|
 \Big) N(dr,dz_1)\Big)^p\nu(dz_2)
\\
&\leq K E\int_Z \Big( \int_{\kappa(n,s)}^s \int_{Z} |\gamma(x_{\kappa(n,r)}^n, z_1)| N(dr,dz_1)\Big)^p \nu(dz_2)
\end{align*}
which due to Assumption A-4, Lemma \ref{lem:BGD:jump} and Remark \ref{as:sde:growth} gives,
\begin{align*}
&  E\int_Z|\gamma_3(x_{\kappa(n,s)}^n, z_2)|^p  \nu(dz_2)
\\
& \leq K  \int_ZE\Big( \int_{\kappa(n,s)}^s \int_{Z} |\gamma(x_{\kappa(n,r)}^n,z_1)|  \tilde N(dr,dz_1)\Big)^p \nu(dz_2)
\\
& +   K \int_Z E\Big( \int_{\kappa(n,s)}^s \int_{Z} |\gamma(x_{\kappa(n,r)}^n,z_1)| \nu(dz_1)dr\Big)^p \nu(dz_2)
\\
& \leq  K n^{-\frac{p}{2}} (1+E|x_{\kappa(n,s)}^n|^p ) +K n^{-1} (1+E|x_{\kappa(n,s)}^n|^p) +K n^{-p} \big(1+E|x_{\kappa(n,s)}^n|^p)
\end{align*}
for every $n \in \mathbb{N}$ and $s\in [0,T]$. This completes the proof.
\end{proof}
%-----------------------------------------------------------------
%                Lemma tilde c:momen tbound
%-----------------------------------------------------------------
The following corollary is a consequence of Lemmas [\ref{lem:c1:rate}, \ref{lem:c2:rate}, \ref{lem:c3:rate}].

\begin{cor} \label{lem:tilde c:momentbound}
Let Assumptions A-2 to A-4 hold, then
\begin{align*}
E\int_Z|\tilde{\gamma}(x_{\kappa(n,s)}^n, z_2)|^p  \nu(dz_2) & \leq   K(1+E|x_{\kappa(n,s)}^n|^p),
\end{align*}
for every $n \in \mathbb{N}$ and $s \in [0,T]$, where $K:=K(L,p,m,d)$.
\end{cor}
%-----------------------------------------------------------------
%                Lemma : one-step : no rate
%-----------------------------------------------------------------
\begin{lem} \label{lem:one-step:norate}
Let Assumptions A-2 to A-4 hold, then
$$
\int_0^uE|x_s^n-x_{\kappa(n,s)}^n|^p ds \leq K n^{-1}\big(1+  \int_0^u  E|x_{\kappa(n,s)}^n|^p ds\big),
$$
for any $u \in [0,T]$, $n \in \mathbb{N}$, where $K:=K(L,p,m,d)$.
\end{lem}
\begin{proof}
Due to H\"{o}lder's inequality,  Lemma \ref{lem:BGD:jump} and Remark \ref{as:scheme:a^n},   one observes that
\begin{align*}
E|&x_s^n-x_{\kappa(n,s)}^n|^p  = E\Big|\int^{s}_{\kappa(n,s)} \tilde{b}^n(x_{\kappa(n,r)}^n)dr+\int^{s}_{\kappa(n,s)} \tilde{\sigma}(x_{\kappa(n,r)}^n)dw_r
\\
& +\int^{s}_{\kappa(n,s)}\int_Z \tilde{\gamma}(x_{\kappa(n,r)}^n, z_2)\tilde N(dr,dz_2)\Big|^p
\\
& \leq K E\Big|\int^{s}_{\kappa(n,s)} \tilde{b}^n(x_{\kappa(n,r)}^n)dr\Big|^p+K E\Big|\int^{s}_{\kappa(n,s)} \tilde{\sigma}(x_{\kappa(n,r)}^n)dw_r\Big|^p
\\
& + K E\Big|\int^{s}_{\kappa(n,s)}\int_Z \tilde{\gamma}(x_{\kappa(n,r)}^n, z_2)\tilde N(dr,dz_2)\Big|^p
\\
& \leq K |s-\kappa(n,s)|^{p-1}E \int_{\kappa(n,s)}^{s}|\tilde{b}^n(x_{\kappa(n,r)}^n)|^pdr
\\
&+ K |s-\kappa(n,s)|^{\frac{p}{2}-1}\int_{\kappa(n,s)}^{s}E|\tilde{\sigma}(x_{\kappa(n,r)}^n)|^pdr
\\
& \quad + K |s-\kappa(n,s)|^{\frac{p}{2}-1}\int^{s}_{\kappa(n,s)}E\Big(\int_Z |\tilde{\gamma}(x_{\kappa(n,r)}^n, z_2)|^2  \nu(dz_2)\Big)^\frac{p}{2} dr
\\
& +K \int^{s}_{\kappa(n,s)}E\int_Z |\tilde{\gamma}(x_{\kappa(n,r)}^n, z_2)|^p  \nu(dz_2)dr
\end{align*}
and then on using Corollaries [\ref{lem:tildeb:momentbound}, \ref{lem:tilde c:momentbound}], one obtains
\begin{align*}
E|x_s^n&-x_{\kappa(n,s)}^n|^p  \leq K n^{-\frac{p}{2}}+ K n^{-\frac{p}{2}} \big(1+E|x_{\kappa(n,s)}^n|^p\big)
\\
& +K n^{-\frac{p}{2}}\big(1+E|x_{\kappa(n,s)}^n|^p\big)+K n^{-1}\big(1+E|x_{\kappa(n,s)}^n|^p\big)
\end{align*}
for any $s \in [0,T]$, $n \in \mathbb{N}$.  This completes the proof.
\end{proof}
%-----------------------------------------------------------------
%                Lemma : scheme :momen bound
%-----------------------------------------------------------------
\begin{lem} \label{lem:scheme:momentbound}
Let Assumptions A-1 to A-4 hold, then
$$
\sup_{n \in \mathbb{N}}E\sup_{0 \leq t \leq T}|x_t^n|^p \leq K,
$$
where $K:=K(L,T,p,m,d,E|\xi|^p)$.
\end{lem}
\begin{proof}
By  It\^{o}'s formula,
\begin{align}
|x_t^n|^p = &|\xi|^p+ p \int_0^{t} |x_s^n|^{p-2} x_s^n \tilde{b}^n(x_{\kappa(n,s)}^n) ds + p\int_{0}^{t} |x_s^n|^{p-2} x_s^n \tilde{\sigma}(x_{\kappa(n,s)}^n)  dw_s  \notag
\\
& + \frac{p(p-2)}{2} \int_{0}^{t} |x_s^n|^{p-4}| \tilde{\sigma}^{*}(x_{\kappa(n,s)}^n)x_s^n|^2 ds  + \frac{p}{2} \int_{0}^{t} |x_s^n|^{p-2}|\tilde{\sigma}(x_{\kappa(n,s)}^n)|^2 ds  \notag
\\
&+  p \int_{0}^{t} \int_{Z} |x_s^n|^{p-2} x_{s}^n \tilde  \gamma(x_{\kappa(n,s)}^n,z_2)   \tilde N(ds,dz_2) \notag
\\
+\int_{0}^{t} \int_{Z}& \big( |x_{s}^n+\tilde \gamma(x_{\kappa(n,s)}^n,z_2)|^p-|x_{s}^n|^p-p|x_{s}^n|^{p-2} x_{s}^n \tilde \gamma(x_{\kappa(n,s)}^n,z_2) \big)N(ds,dz_2)
\label{eq:rr}
\end{align}
almost surely for any $t \in [0,T]$ and $n \in \mathbb{N}$. For estimating the second term on the right hand side of the above equation, one writes
$$
x_s^n \tilde{b}^n(x_{\kappa(n,s)}^n)=(x_s^n-x_{\kappa(n,s)}^n) \tilde{b}^n(x_{\kappa(n,s)}^n)+x_{\kappa(n,s)}^n \tilde{b}^n(x_{\kappa(n,s)}^n)
$$
which on using  Remark \ref{as:sde:growth} yields
\begin{align}
x_s^n \tilde{b}^n(x_{\kappa(n,s)}^n)&\leq\Big(\int_{\kappa(n,s)}^{s}\tilde{b}^n(x_{\kappa(n,r)}^n)dr + \int_{\kappa(n,s)}^{s}\tilde{\sigma}(x_{\kappa(n,r)}^n)dw_r \notag
\\
& +\int_{\kappa(n,s)}^{s}\int_{Z}\tilde{\gamma}(x_{\kappa(n,r)}^n,z_2)\tilde N(dr,dz_2)\Big)\tilde{b}^n(x_{\kappa(n,s)}^n) + L(1+|x_{\kappa(n,s)}^n|^2)\notag
\end{align}
and then by using  Remark \ref{as:scheme:a^n} and Schwartz inequality, one obtains
\begin{align}
x_s^n \tilde{b}^n(x_{\kappa(n,s)}^n) & \leq |\tilde{b}^n(x_{\kappa(n,s)}^n)|\int_{\kappa(n,s)}^{s}|\tilde{b}^n(x_{\kappa(n,r)}^n)|dr \notag
\\
& + |\tilde{b}^n(x_{\kappa(n,s)}^n)|\Big|\int_{\kappa(n,s)}^{s}\tilde{\sigma}(x_{\kappa(n,r)}^n)dw_r\Big| \notag
\\
& + |\tilde{b}^n(x_{\kappa(n,s)}^n)|\Big|\int_{\kappa(n,s)}^{s}\int_{Z}\tilde{\gamma}(x_{\kappa(n,r)}^n,z_2)\tilde N(dr,dz_2)\Big|+L(1+|x_{\kappa(n,s)}^n|^2)\notag
%\\
%&\leq n^{\frac{1}{2}}\int_{\kappa(n,s)}^{s}n^{\frac{1}{2}}dr + n^{\frac{1}{2}}\Big|\int_{\kappa(n,s)}^{s}\tilde{\sigma}(x_{\kappa(n,r)}^n)dw_r\Big|+ n^{\frac{1}{2}} \Big|\int_{\kappa(n,s)}^{s}\int_{Z}\tilde{\gamma}(x_{\kappa(n,r)}^n,z_2)\tilde N(dr,dz_2)\Big|+L(1+|x_{\kappa(n,s)}^n|^2)\notag
\\
&\leq 1 + n^{\frac{1}{2}}\Big|\int_{\kappa(n,s)}^{s}\tilde{\sigma}(x_{\kappa(n,r)}^n)dw_r\Big| \notag
\\
& + n^{\frac{1}{2}} \Big|\int_{\kappa(n,s)}^{s}\int_{Z}\tilde{\gamma}(x_{\kappa(n,r)}^n,z_2)\tilde N(dr,dz_2)\Big| +L(1+|x_{\kappa(n,s)}^n|^2)\notag
\end{align}
almost surely for any $s \in [0,T]$ and $n \in \mathbb{N}$. Further, due to Young's inequality,
\begin{align}
|x_s^n|^{p-2}x_s^n \tilde{b}^n(x_{\kappa(n,s)}^n) & \leq 1+ K|x_s^n|^{p} +K|x_{\kappa(n,s)}^n|^p + n^{\frac{p}{4}}\Big|\int_{\kappa(n,s)}^{s}\tilde{\sigma}(x_{\kappa(n,r)}^n)dw_r\Big|^{\frac{p}{2}}\notag
\\
&+ n^{\frac{1}{2}} |x_s^n|^{p-2} \Big|\int_{\kappa(n,s)}^{s}\int_{Z}\tilde{\gamma}(x_{\kappa(n,r)}^n,z_2)\tilde N(dr,dz_2)\Big| \label{eq:scheme:moment:1}
\end{align}
almost surely for any $s \in [0,T]$ and $n \in \mathbb{N}$.

Moreover, since the map $y \rightarrow |y|^p$ is of class $C^2$, by the formula for the remainder,
\begin{equation}
|y_1+y_2|^p-|y_1|^p-p|y_1|^{p-2}y_1y_2 \leq K \int_0^1 |y_1+qy_2|^{p-2}|y_2|^2 dq \leq K(|y_1|^{p-2}|y_2|^2+|y_2|^p) \label{eq:y1y2}
\end{equation}
for any $y_1,y_2 \in \mathbb{R}^d$.

Hence, one first substitutes the estimates from \eqref{eq:scheme:moment:1} in \eqref{eq:rr} and uses \eqref{eq:y1y2} to estimate the last term of  \eqref{eq:rr} which  on taking suprema  and expectation gives
%\begin{align}
%|x_t^n|^p \leq  &|\xi|^p+K+ K \int_{0}^{t}|x_s^n|^{p}ds +K \int_{0}^{t}|x_{\kappa(n,s)}^n|^pds + K n^{\frac{p}{4}} \int_{0}^{t} \Big|\int_{\kappa(n,s)}^{s}\tilde{\sigma}(x_{\kappa(n,r)}^n)dw_r\Big|^{\frac{p}{2}}ds \notag
%\\
%& +K n^{\frac{1}{2}} \int_{0}^{t} |x_s^n|^{p-2} \Big|\int_{\kappa(n,s)}^{s}\int_{Z}\tilde{\gamma}(x_{\kappa(n,r)}^n,z_2)\tilde N(dr,dz_2)\Big|ds \notag
%\\
%&+ p\int_{0}^{t} |x_s^n|^{p-2} x_s^n \tilde{\sigma}(x_{\kappa(n,s)}^n)  dw_s   + \frac{p(p-2)}{2} \int_{0}^{t} |x_s^n|^{p-4}| \tilde{\sigma}^{*}(x_{\kappa(n,s)}^n)x_s^n|^2 ds \notag
%\\
%&+ \frac{p}{2} \int_{0}^{t} |x_s^n|^{p-2}|\tilde{\sigma}(x_{\kappa(n,s)}^n)|^2 ds +  p \int_{0}^{t} \int_{Z} |x_s^n|^{p-2} x_{s}^n \tilde  \gamma(x_{\kappa(n,s)}^n,z_2)   \tilde N(ds,dz_2) \notag
%\\
%&+\int_{0}^{t} \int_{Z}\{ |x_{s}^n+\tilde \gamma(x_{\kappa(n,s)}^n,z_2)|^p-|x_{s}^n|^p-p|x_{s}^n|^{p-2} x_{s}^n \tilde \gamma(x_{\kappa(n,s)}^n,z_2) \}N(ds,dz_2) \label{eq:rt}
%\end{align}
\begin{align}
E\sup_{0 \leq t \leq u}|x_t^n|^p & \leq  E|\xi|^p+K+ K \int_{0}^{u}E\sup_{0 \leq r \leq s}|x_r^n|^{p}ds + K E\int_{0}^{u} |x_s^n|^{p-2}| \tilde{\sigma}(x_{\kappa(n,s)}^n)|^2 ds \notag
\\
&+ K n^{\frac{p}{4}} \int_{0}^{u} E\Big|\int_{\kappa(n,s)}^{s}\tilde{\sigma}(x_{\kappa(n,r)}^n)dw_r\Big|^{\frac{p}{2}}ds \notag
\\
& + pE \sup_{0 \leq t \leq u}\Big|\int_{0}^{t} |x_s^n|^{p-2} x_s^n \tilde{\sigma}(x_{\kappa(n,s)}^n)  dw_s\Big|    \notag
\\
&+p E  \sup_{0 \leq t \leq u} \Big|\int_{0}^{t} \int_{Z} |x_s^n|^{p-2} x_s^n\tilde  \gamma(x_{\kappa(n,s)}^n,z_2)  \tilde N(ds,dz_2)\Big|  \notag
\\
&+E\int_{0}^{u} \int_{Z}\big\{ |x_{s}^n|^{p-2} |\tilde \gamma(x_{\kappa(n,s)}^n,z_2)|^2+ |\tilde \gamma(x_{\kappa(n,s)}^n,z_2)|^p\big\}N(ds,dz_2) \notag
\\
& +K n^{\frac{1}{2}} \int_{0}^{u} E|x_s^n|^{p-2} \Big|\int_{\kappa(n,s)}^{s}\int_{Z}\tilde{\gamma}(x_{\kappa(n,r)}^n,z_2) \tilde N(dr,dz_2)\Big|ds \notag
\\
= & C_1+C_2+C_3+C_4+C_5+C_6+C_7. \label{eq:C1+...+C7}
\end{align}
Here, $C_1$ is given by
$$
C_1:=E|\xi|^p+K+ K \int_{0}^{u}E\sup_{0 \leq r \leq s}|x_r^n|^{p}ds
$$
for any $u \in [0,T]$. Due to Young's inequality and Corollary \ref{lem:tildeb:momentbound},  $C_2$ can be estimated by
\begin{align}
C_2:= & K E\int_{0}^{u} |x_s^n|^{p-2}| \tilde{\sigma}(x_{\kappa(n,s)}^n)|^2 ds \leq K \int_{0}^{u} E|x_s^n|^p ds+K\int_{0}^{u}E| \tilde{\sigma}(x_{\kappa(n,s)}^n)|^p ds \notag
\\
\leq & K+K\int_{0}^{u} E|x_{\kappa(n,s)}^n|^p ds. \label{eq:C1}
\end{align}
For $C_3$, one uses an elementary stochastic inequality and Corollary \ref{lem:tildeb:momentbound} to obtain
\begin{align}
C_3& := K n^{\frac{p}{4}} \int_{0}^{u} E\Big|\int_{\kappa(n,s)}^{s}\tilde{\sigma}(x_{\kappa(n,r)}^n)dw_r\Big|^{\frac{p}{2}}ds \notag
\\
& \leq  K n^{\frac{p}{4}} \int_{0}^{u} E\Big(\int_{\kappa(n,s)}^{s}|\tilde{\sigma}(x_{\kappa(n,r)}^n)|^2dr\Big)^{\frac{p}{4}}ds \notag
\\
& \leq 1 + K n^{\frac{p}{2}} \int_{0}^{u} E\Big(\int_{\kappa(n,s)}^{s}|\tilde{\sigma}(x_{\kappa(n,r)}^n)|^2dr\Big)^{\frac{p}{2}}ds \notag
\\
& \leq  K n \int_{0}^{u} \int_{\kappa(n,s)}^{s}E|\tilde{\sigma}(x_{\kappa(n,r)}^n)|^p drds \notag
\\
& \leq  K  + K\int_{0}^{u} E|x_{\kappa(n,s)}^n|^p ds. \label{eq:C2}
\end{align}
Also, by using Burkholder-Gundy-Davis, Young's  and H\"{o}lder's inequalities,  $C_4$ can be estimated as
\begin{align*}
C_4 & := pE \sup_{0 \leq t \leq u}\Big|\int_{0}^{t} |x_s^n|^{p-2} x_s^n \tilde{\sigma}(x_{\kappa(n,s)}^n)  dw_s\Big|\leq K E \Big(\int_{0}^{u} |x_s^n|^{2p-2} |\tilde{\sigma}(x_{\kappa(n,s)}^n)|^2  ds\Big)^\frac{1}{2}
\\
\leq & K E \sup_{0 \leq t \leq u}|x_t^n|^{p-1}\Big(\int_{0}^{u}  |\tilde{\sigma}(x_{\kappa(n,s)}^n)|^2  ds\Big)^\frac{1}{2} \notag
\\
 \leq & \frac{1}{8} E \sup_{0 \leq t \leq u}|x_t^n|^p+K E\Big(\int_{0}^{u}  |\tilde{\sigma}(x_{\kappa(n,s)}^n)|^2  ds\Big)^\frac{p}{2}
\\
\leq & \frac{1}{8} E \sup_{0 \leq t \leq u}|x_t^n|^p+K \int_{0}^{u}  E|\tilde{\sigma}(x_{\kappa(n,s)}^n)|^p  ds
\end{align*}
which due to Corollary \ref{lem:tildeb:momentbound} gives
\begin{align}
C_4 \leq  \frac{1}{8} E \sup_{0 \leq t \leq u}|x_t^n|^p+ K+K\int_{0}^{u} E|x_{\kappa(n,s)}^n|^p ds. \label{eq:C3}
\end{align}
Further, due to Lemma \ref{lem:BGD:jump}, Young's  and H\"{o}lder's inequalities, one obtains
\begin{align*}
C_5:= & E \sup_{0 \leq t \leq u} \int_{0}^{t} \int_{Z} |x_s^n|^{p-2} x_s^n \tilde  \gamma(x_{\kappa(n,s)}^n,z_2)|   \tilde N(ds,dz_2)
\\
\leq & K E  \int_{0}^{u} \Big(\int_{Z} |x_s^n|^{2p-2} |\tilde  \gamma(x_{\kappa(n,s)}^n,z_2)|^2 \nu(dz_2)  \Big)^\frac{1}{2}ds
\\
\leq & K E \sup_{0 \leq t \leq u}|x_t^n|^{p-1} \int_{0}^{u} \Big(\int_{Z}  |\tilde  \gamma(x_{\kappa(n,s)}^n,z_2)|^2 \nu(dz_2)  \Big)^\frac{1}{2}ds
\\
\leq & \frac{1}{8} E \sup_{0 \leq t \leq u}|x_t^n|^p+ K  \int_{0}^{u} E\int_{Z}  |\tilde  \gamma(x_{\kappa(n,s)}^n,z_2)|^p \nu(dz_2) ds
\end{align*}
which due to Corollary \ref{lem:tilde c:momentbound} gives
\begin{align}
C_5 \leq &\frac{1}{8} E \sup_{0 \leq t \leq u}|x_t^n|^p + K +K\int_{0}^{u}E|x_{\kappa(n,s)}^n|^pds. \label{eq:C4}
\end{align}
Moreover, $C_6$ can be estimated as
\begin{align*}
C_6:=&E\int_{0}^{u} \int_{Z}\big( |x_{s}^n|^{p-2} |\tilde \gamma(x_{\kappa(n,s)}^n,z_2)|^2+ |\tilde \gamma(x_{\kappa(n,s)}^n,z_2)|^p\big){N}(ds,dz_2)
\\
\leq & E\int_{0}^{u} \int_{Z}\big( |x_{s}^n|^{p-2} |\tilde \gamma(x_{\kappa(n,s)}^n,z_2)|^2+ |\tilde \gamma(x_{\kappa(n,s)}^n,z_2)|^p\big)\nu(dz_2) ds
\\
\leq & K E\int_{0}^{u} \int_{Z} |x_{s}^n|^{p-2} |\tilde \gamma(x_{\kappa(n,s)}^n,z_2)|^2 \nu(dz_2) ds
\\
& + K E\int_{0}^{u} \int_{Z}|\tilde \gamma(x_{\kappa(n,s)}^n,z_2)|^p \nu(dz_2) ds
\\
\leq & K E\sup_{0 \leq t \leq u}|x_{t}^n|^{p-2}\int_{0}^{u} \int_{Z}  |\tilde \gamma(x_{\kappa(n,s)}^n,z_2)|^2 \nu(dz_2) ds
\\
& + K E\int_{0}^{u} \int_{Z}|\tilde \gamma(x_{\kappa(n,s)}^n,z_2)|^p \nu(dz_2) ds
\end{align*}
which on the application of Young's inequality, H\"{o}lder's inequality and Corollary \ref{lem:tilde c:momentbound} implies
\begin{align}
C_6 \leq & \frac{1}{8}E\sup_{0 \leq t \leq u}|x_{t}^n|^p + K \int_{0}^{u} E \int_{Z}  |\tilde \gamma(x_{\kappa(n,s)}^n,z_2)|^p \nu(dz_2) ds  \notag
\\
\leq & \frac{1}{8}E\sup_{0 \leq t \leq u}|x_{t}^n|^p+K +K\int_{0}^{u}E|x_{\kappa(n,s)}^n|^pds. \label{eq:C5}
\end{align}
Finally, for estimating $C_7$, one uses the inequality $|x_{s}^n|^{p-2} \leq 2^{p-3}|x_{\kappa(n,s)}^n|^{p-2}+2^{p-3}|x_{s}^n-x_{\kappa(n,s)}^n|^{p-2}$ to obtain,
\begin{align}
C_7:= &K n^{\frac{1}{2}} \int_{0}^{u} E|x_s^n|^{p-2} \Big|\int_{\kappa(n,s)}^{s}\int_{Z}\tilde{\gamma}(x_{\kappa(n,r)}^n,z_2)\tilde N(dr,dz_2)\Big|ds \notag
\\
\leq & K n^{\frac{1}{2}} \int_{0}^{u} E|x_{\kappa(n,s)}^n|^{p-2} \Big|\int_{\kappa(n,s)}^{s}\int_{Z}\tilde{\gamma}(x_{\kappa(n,r)}^n,z_2)\tilde N(dr,dz_2)\Big|ds \notag
\\
&+K n^{\frac{1}{2}} \int_{0}^{u} E|x_s^n-x_{\kappa(n,s)}^n|^{p-2} \Big|\int_{\kappa(n,s)}^{s}\int_{Z}\tilde{\gamma}(x_{\kappa(n,r)}^n,z_2)\tilde N(dr,dz_2)\Big|ds \notag
\\
& =:C_{7a}+C_{7b}. \label{eq:C6a+C6b}
\end{align}
To estimate $C_{7a}$, one uses Lemma \ref{lem:BGD:jump} to write
\begin{align*}
C_{7a}:= & K n^{\frac{1}{2}} \int_{0}^{u} E \Big|\int_{\kappa(n,s)}^{s}\int_{Z}|x_{\kappa(n,s)}^n|^{p-2}\tilde{\gamma}(x_{\kappa(n,r)}^n,z_2)\tilde N(dr,dz_2)\Big|ds
\\
\leq & K n^{\frac{1}{2}} \int_{0}^{u} E \Big(\int_{\kappa(n,s)}^{s}\int_{Z}|x_{\kappa(n,s)}^n|^{2p-4}|\tilde{\gamma}(x_{\kappa(n,r)}^n,z_2)|^2\nu(dz_2)dr\Big)^\frac{1}{2} ds
\end{align*}
which on using Young's and H\"{o}lder's inequalities gives
\begin{align*}
C_{7a} \leq & K E \sup_{0 \leq t \leq u}|x_{t}^n|^{p-2} n^{\frac{1}{2}} \int_{0}^{u}  \Big(\int_{\kappa(n,s)}^{s}\int_{Z}|\tilde{\gamma}(x_{\kappa(n,r)}^n,z_2)|^2\nu(dz_2)dr \Big)^\frac{1}{2} ds
\\
\leq &\frac{1}{8} E \sup_{0 \leq t \leq u}|x_{t}^n|^p + K n^{\frac{p}{4}} E  \int_{0}^{u} \Big(\int_{\kappa(n,s)}^{s}\int_{Z}|\tilde{\gamma}(x_{\kappa(n,r)}^n,z_2)|^2\nu(dz_2)dr ds\Big)^\frac{p}{4}
\\
\leq &\frac{1}{8} E \sup_{0 \leq t \leq u}|x_{t}^n|^p+ 1+ K n^{\frac{p}{2}} E  \int_{0}^{u} \Big(\int_{\kappa(n,s)}^{s}\int_{Z}|\tilde{\gamma}(x_{\kappa(n,r)}^n,z_2)|^2\nu(dz_2)dr ds\Big)^\frac{p}{2}
\\
\leq & \frac{1}{8} E \sup_{0 \leq t \leq u}|x_{t}^n|^p + 1+ K n   \int_{0}^{u} \int_{\kappa(n,s)}^{s}E\int_{Z}|\tilde{\gamma}(x_{\kappa(n,r)}^n,z_2)|^p\nu(dz_2) dr ds
\end{align*}
and then using Corollary \ref{lem:tilde c:momentbound},
\begin{align}
C_{7a} \leq & \frac{1}{8} E \sup_{0 \leq t \leq u}|x_{t}^n|^p +K+K\int_{0}^{u}E|x_{\kappa(n,s)}^n|^p ds.  \label{eq:C6a}
\end{align}
Further, to estimate $C_{7b}$, one observes that due to Young's inequality,
\begin{align}
C_{7b}:= &K n^{\frac{1}{2}} \int_{0}^{u} E|x_s^n-x_{\kappa(n,s)}^n|^{p-2} \Big|\int_{\kappa(n,s)}^{s}\int_{Z}\tilde{\gamma}(x_{\kappa(n,r)}^n,z_2)\tilde N(dr,dz_2)\Big|ds \notag
\\
\leq & K n^{\frac{1}{2}} \int_{0}^{u} E|x_s^n-x_{\kappa(n,s)}^n|^p ds \notag
\\
& + K n^{\frac{1}{2}} \int_{0}^{u} E\Big|\int_{\kappa(n,s)}^{s}\int_{Z}\tilde{\gamma}(x_{\kappa(n,r)}^n,z_2)\tilde N(dr,dz_2)\Big|^\frac{p}{2}ds \notag
\\
\leq & K n^{\frac{1}{2}} \int_{0}^{u} E|x_s^n-x_{\kappa(n,s)}^n|^p ds \notag
\\
& +1+ K n \int_{0}^{u} E\Big|\int_{\kappa(n,s)}^{s}\int_{Z}\tilde{\gamma}(x_{\kappa(n,r)}^n,z_2)\tilde N(dr,dz_2)\Big|^pds \notag
\end{align}
which on using Lemmas [\ref{lem:BGD:jump}, \ref{lem:one-step:norate}] gives
\begin{align}
C_{7b} \leq & K n^{-\frac{1}{2}}  \big(1+  \int_{0}^{u}E|x_{\kappa(n,s)}^n|^p ds \big) \notag
\\
& + 1 + K n \int_{0}^{u} E\Big(\int_{\kappa(n,s)}^{s}\int_{Z} |\tilde{\gamma}(x_{\kappa(n,r)}^n,z_2)|^2 \nu(dz_2) dr\Big)^\frac{p}{2}ds \notag
\\
&+ K n \int_{0}^{u} E\int_{\kappa(n,s)}^{s}\int_{Z} |\tilde{\gamma}(x_{\kappa(n,r)}^n,z_2)|^p \nu(dz_2) drds \notag
\\
\leq & K + K \int_{0}^{u} E|x_{\kappa(n,s)}^n|^p ds \notag
\\
& + K n^{2-\frac{p}{2}} \int_{0}^{u} \int_{\kappa(n,s)}^{s}E\int_{Z} |\tilde{\gamma}(x_{\kappa(n,r)}^n,z_2)|^p \nu(dz_2)  dr ds \notag
\\
& + K n \int_{0}^{u} \int_{\kappa(n,s)}^{s}E\int_{Z} |\tilde{\gamma}(x_{\kappa(n,r)}^n,z_2)|^\frac{p}{2} \nu(dz_2) drds \notag
\end{align}
and finally due to Corollary \ref{lem:tilde c:momentbound}, one obtains
\begin{align}
C_{7b} \leq & K + K  \int_{0}^{u} E|x_{\kappa(n,s)}^n|^p ds. \label{eq:C6b}
\end{align}
Thus, on using estimates from \eqref{eq:C6a} and \eqref{eq:C6b} in \eqref{eq:C6a+C6b}, one has
\begin{align}
C_7 \leq \frac{1}{8} E \sup_{0 \leq t \leq u}|x_{t}^n|^p +K+K\int_{0}^{u}E|x_{\kappa(n,s)}^n|^p ds. \label{eq:C6}
\end{align}
Hence, on substituting estimates from \eqref{eq:C1}, \eqref{eq:C2}, \eqref{eq:C3}, \eqref{eq:C4}, \eqref{eq:C5} and \eqref{eq:C6} in \eqref{eq:C1+...+C7}, one obtains due to \eqref{eq:scheme:mb:finite}
\begin{align}
E \sup_{0 \leq t \leq u}|x_t^n|^p \leq \frac{1}{2}E \sup_{0 \leq t \leq u}|x_t^n|^p + E|\xi|^p+K+ K  \int_{0}^{u} E\sup_{0 \leq r \leq s}|x_r^n|^pds < \infty \notag
\end{align}
which on the application of Gronwall's lemma completes the proof.
\end{proof}
%------------------------------------------------------------------
\section{Rate of Convergence} \label{sec:rate}
%------------------------------------------------------------------

In what follows, we assume that $6(\chi+2) \leq p$ and $\delta \in (4/(p-2), 1)$.
%-----------------------------------------------------------------
%                Lemma : one-step : Rate
%-----------------------------------------------------------------
\begin{lem} \label{lem:one-step:rate}
Let Assumptions A-1 to A-4 hold, then
$$
\sup_{0 \leq t \leq T}E|x_t^n-x_{\kappa(n,t)}^n|^r \leq K n^{-1}
$$
for any $2 \leq r \leq p$.
\end{lem}
\begin{proof}
The proof follows due to Lemmas [\ref{lem:one-step:norate}, \ref{lem:scheme:momentbound}].
\end{proof}
%-----------------------------------------------------------------
%                Lemma : one-step : sigma: Rate
%-----------------------------------------------------------------
\begin{lem} \label{lem:one-step:sigma:rate}
Let Assumptions A-1 to
%A-4 and
A-5  hold, then
$$
\sup_{0 \leq t \leq T} E|\sigma(x_{\kappa(n,t)}^n)|^2|x_t^n-x_{\kappa(n,t)}^n|^2 \leq K n^{-1}.
$$
\end{lem}
\begin{proof}
By Remark \ref{as:sde:growth} and H\"older's inequality,
\begin{align*}
E|&\sigma(x_{\kappa(n,t)}^n)|^2|x_t^n-x_{\kappa(n,t)}^n|^2  \leq K E\Big|\int^{t}_{\kappa(n,t)} (1+|x_{\kappa(n,t)}^n|^2) \tilde{b}^n(x_{\kappa(n,r)}^n)dr\Big|^2
\\
&  +K E\Big|\int^{t}_{\kappa(n,t)} (1+|x_{\kappa(n,t)}^n|^2) \tilde{\sigma}(x_{\kappa(n,r)}^n)dw_r\Big|^2
\\
& + K E\Big|\int^{t}_{\kappa(n,t)}\int_Z (1+|x_{\kappa(n,t)}^n|^2) \tilde{\gamma}(x_{\kappa(n,r)}^n, z_2)\tilde N(dr,dz_2)\Big|^2
\\
& \leq K n^{-1} E\int^{t}_{\kappa(n,t)} (1+|x_{\kappa(n,t)}^n|^2)^2 |\tilde{b}^n(x_{\kappa(n,r)}^n)|^2dr
\\
& +K E\int^{t}_{\kappa(n,t)} (1+|x_{\kappa(n,t)}^n|^2)^2 |\tilde{\sigma}(x_{\kappa(n,r)}^n)|^2 dr
\\
& \quad  + K E\int^{t}_{\kappa(n,t)}\int_Z (1+|x_{\kappa(n,t)}^n|^2)^2 |\tilde{\gamma}(x_{\kappa(n,r)}^n, z_2)|^2 \nu(dz_2)dr.
\end{align*}
Then, on further application of H\"older's inequality along with Remark \ref{as:sde:poly} and Corollaries [\ref{lem:tildeb:momentbound}, \ref{lem:tilde c:momentbound}], one completes the proof.
 \end{proof}
%-----------------------------------------------------------------
%                Lemma : one-step : gamma: Rate
%-----------------------------------------------------------------
\begin{lem} \label{lem:one-step:gamma:rate}
Let Assumptions A-1 to %A-4 and
A-5  hold, then
$$
\sup_{0 \leq t \leq T}E\int_Z|\gamma(x_{\kappa(n,t)}^n,z_2)|^2|x_t^n-x_{\kappa(n,t)}^n|^2 \nu(dz_2)\leq K n^{-1}.
$$
\end{lem}
\begin{proof}
The proof follows the same arguments as used in Lemma \ref{lem:one-step:sigma:rate}.
\end{proof}
%-----------------------------------------------------------------
%                Lemma : a -tilde a : Rate
%-----------------------------------------------------------------
\begin{lem} \label{lem:form:rate}
%------------------------------------------------------------
Consider equation \eqref{eq:an:tilde} and let Assumptions A-1 to %A-4 and
 A-5 be satisfied, then
$$
\sup_{0\leq t \leq T} E |b(x_t^n)-\tilde{b}^n(x_t^n)|^2  \leq K n^{-2}.
$$
\end{lem}
\begin{proof}
The proof immediately follows due to equation \eqref{eq:an:tilde}, Remarks [\ref{as:scheme:a^n}, \ref{as:sde:poly}]  and Lemma \ref{lem:scheme:momentbound}.
\end{proof}
%-----------------------------------------------------------------
%                Lemma : c - tilde c: Rate
%-----------------------------------------------------------------
\begin{lem} \label{c-tilde c:rate}
Let Assumptions A-1 to %A-4 and
A-5 hold, then
\begin{align*}
\sup_{0 \leq t \leq T}E\int_Z|\gamma(x_t^n, z_2) - \tilde \gamma(x_{\kappa(n,t)}^n, z_2)|^2 \nu(dz_2) \leq K n^{-2}.
\end{align*}
\end{lem}
\begin{proof}
By It\^o's formula,
$$
\gamma^i(x_t^n,  z_2) = \gamma^i(x_{\kappa(n,t)}^n, z_2) + \sum_{u=1}^d \int_{\kappa(n,t)}^t \frac{\partial \gamma^i(x_r^n, z_2)}{\partial x^u} \tilde{b}^{n,u}(x_{\kappa(n,r)}^n)dr \notag
$$
\begin{align}
&+\frac{1}{2}\sum_{u_1,u_2=1}^{d} \sum_{l_1=1}^{m}\int_{\kappa(n,t)}^t \frac{\partial^2 \gamma^i(x_r^n, z_2)}{\partial x^{u_1} \partial x^{u_2}}  \tilde{\sigma}^{(u_1,l_1)}(x_{\kappa(n,r)}^n)\tilde{\sigma}^{(u_2,l_1)}(x_{\kappa(n,r)}^n)dr \notag
\\
&+\sum_{u=1}^d \sum_{j=1}^m \int_{\kappa(n,t)}^t \frac{\partial \gamma^i(x_r^n, z_2)}{\partial x^u}  \tilde{\sigma}^{(u,j)}(x_{\kappa(n,r)}^n)dw_r^j \notag
\\
& +\sum_{u=1}^d \int_{\kappa(n,t)}^t \int_Z \frac{\partial \gamma^i(x_r^n, z_2)}{\partial x^u}  \tilde{\gamma}^{u}(x_{\kappa(n,r)}^n,z_1)\tilde N(dr, dz_1) \notag
\\
&+ \int_{\kappa(n,t)}^t \int_Z \Big(\gamma^i(x_r^n+\tilde{\gamma}(x_{\kappa(n,r)}^n,z_1), z_2)- \gamma^i(x_r^n, z_2) \notag
\\
& \qquad-\sum_{u=1}^{d}\frac{\partial \gamma^i(x_r^n, z_2)}{\partial x^{u}} {\tilde{\gamma}}^{u}(x_{\kappa(n,r)}^n,z_1)\Big)N(dr, dz_1). \label{eq:c:ito}
\end{align}
%Also, one recalls equation \eqref{eq:cn:tilde},
%\begin{align}
%\tilde{\gamma}^i&(x_{\kappa(n,t)}^n, z_2)  = \gamma^i(x_{\kappa(n,t)}^n, z_2)  + \sum_{j=1}^m \sum_{u=1}^d  \int_{\kappa(n,t)}^t \frac{\partial \gamma^{i}(x_{\kappa(n,r)}^n, z_2)}{\partial x^u}\sigma^{(u,j)}(x_{\kappa(n,r)}^n)dw_r^j \notag
%\\
%&\hspace{2cm}+\sum_{u=1}^d \int_{\kappa(n, t)}^{t} \int_Z \frac{\partial \gamma^{i}(x_{\kappa(n,r)}^n, z_2)}{\partial x^u} \gamma^u(x_{\kappa(n,r)}^n,z_1) \tilde N(dr, dz_1) \notag
%\\
%&+\int_{\kappa(n,t)}^t \int_{Z} \Big( \gamma^{i}(x_{\kappa(n,r)}^n+\gamma(x_{\kappa(n,r)}^n, z_1), z_2)- \gamma^{i}(x_{\kappa(n,r)}^n, z_2) -\sum_{u=1}^d \frac{\partial \gamma^{i}(x_{\kappa(n,r)}^n,z_2)}{\partial x^u} \gamma^u(x_{\kappa(n,r)}^n,z_1) \Big) N(dr,dz_1). \label{eq:nn}
%\end{align}
Then, on taking expectation of square of difference between \eqref{eq:c:ito} and \eqref{eq:cn:tilde}, One obtains
\begin{align*}
& E\int_Z|\gamma^i(x_t^n, z_2)-\tilde{\gamma}^i(x_{\kappa(n,t)}^n, z_2)|^2\nu(dz_2)
\\
& \leq K \int_Z E\Big|\sum_{u=1}^d \int_{\kappa(n,t)}^t \frac{\partial \gamma^i(x_r^n, z_2)}{\partial x^u} \tilde{b}^{n,u}(x_{\kappa(n,r)}^n)dr\Big|^2\nu(dz_2) \notag
\\
&+K \int_Z E\Big|\sum_{u_1,u_2=1}^{d} \sum_{l_1=1}^{m}\int_{\kappa(n,t)}^t \frac{\partial^2 \gamma^i(x_r^n, z_2)}{\partial x^{u_1} \partial x^{u_2}}  \tilde{\sigma}^{(u_1,l_1)}(x_{\kappa(n,r)}^n)\tilde{\sigma}^{(u_2,l_1)}(x_{\kappa(n,r)}^n)dr \Big|^2 \nu(dz_2) \notag
\\
&+K \int_Z E\Big|\sum_{u=1}^d \sum_{j=1}^m \int_{\kappa(n,t)}^t \Big(\frac{\partial \gamma^i(x_r^n, z_2)}{\partial x^u} - \frac{\partial \gamma^i(x_{\kappa(n,r)}^n, z_2)}{\partial x^u}\Big) \sigma^{(u,j)}(x_{\kappa(n,r)}^n)dw_r^j \Big|^2 \nu(dz_2)
\\
& + K \int_Z E\Big|\sum_{u=1}^d \sum_{j=1}^m \int_{\kappa(n,t)}^t \frac{\partial \gamma^i(x_r^n, z_2)}{\partial x^u} \sigma_1^{(u,j)}(x_{\kappa(n,r)}^n)dw_r^j \Big|^2 \nu(dz_2) \notag
\\
&+K \int_Z E\Big|\sum_{u=1}^d \sum_{j=1}^m \int_{\kappa(n,t)}^t \frac{\partial \gamma^i(x_r^n, z_2)}{\partial x^u} \sigma_2^{(u,j)}(x_{\kappa(n,r)}^n)dw_r^j\Big|^2 \nu(dz_2)  \notag
\\
& + K \int_Z E\Big|\sum_{u=1}^d \sum_{j=1}^m \int_{\kappa(n,t)}^t \frac{\partial \gamma^i(x_r^n, z_2)}{\partial x^u} \sigma_3^{(u,j)}(x_{\kappa(n,r)}^n)dw_r^j \Big|^2 \nu(dz_2) \notag
\\
&+K \int_Z E\Big|\sum_{u=1}^d \int_{\kappa(n,t)}^t \int_Z \Big[\frac{\partial \gamma^i(x_r^n, z_2)}{\partial x^u} - \frac{\partial \gamma^i(x_{\kappa(n,r)}^n, z_2)}{\partial x^u}\Big]\gamma^{u}(x_{\kappa(n,r)}^n,z_1)\tilde N(dr, dz_1)  \Big|^2 \nu(dz_2) \notag
\\
& + K \int_Z E\Big| \sum_{u=1}^d \int_{\kappa(n,t)}^t \int_Z \frac{\partial \gamma^i(x_r^n, z_2)}{\partial x^u} \gamma_1^{u}(x_{\kappa(n,r)}^n,z_1)\tilde N(dr, dz_1) \Big|^2 \nu(dz_2) \notag
\\
& + K \int_Z E\Big|\sum_{u=1}^d \int_{\kappa(n,t)}^t \int_Z \frac{\partial \gamma^i(x_r^n, z_2)}{\partial x^u} \gamma_2^{u}(x_{\kappa(n,r)}^n,z_1)\tilde N(dr, dz_1) \Big|^2 \nu(dz_2) \notag
\\
&+K \int_Z E\Big|\sum_{u=1}^d \int_{\kappa(n,t)}^t \int_Z \frac{\partial \gamma^i(x_r^n, z_2)}{\partial x^u} \gamma_3^{u}(x_{\kappa(n,r)}^n,z_1)\tilde N(dr, dz_1)\Big|^2 \nu(dz_2) \notag
\end{align*}
\begin{align*}
& +K \int_Z E\Big[\int_{\kappa(n,t)}^t \int_Z \Big(|\gamma^i(x_r^n+\tilde{\gamma}(x_{\kappa(n,r)}^n,z_1), z_2)-\gamma^{i} (x_{\kappa(n,r)}^n+\gamma(x_{\kappa(n,r)}^n, z_1), z_2)|
\\
&+|\gamma^i(x_r^n, z_2)-\gamma^{i}(x_{\kappa(n,r)}^n, z_2)|+\sum_{u=1}^d\Big| \frac{\partial \gamma^{i}(x_r^n,z_2)}{\partial x^u}- \frac{\partial \gamma^{i}(x_{\kappa(n,r)}^n,z_2)}{\partial x^u}\Big| |\gamma^u(x_{\kappa(n,r)}^n,z_1)|
\\
&+\sum_{u=1}^{d}\Big|\frac{\partial \gamma^i(x_r^n, z_2)}{\partial x^{u}}\Big||\gamma_1^u(x_{\kappa(n,r)}^n,z_1)|+\sum_{u=1}^{d}\Big|\frac{\partial \gamma^i(x_r^n, z_2)}{\partial x^{u}}\Big||\gamma_2^u(x_{\kappa(n,r)}^n,z_1)|
\\
&+\sum_{u=1}^{d}\Big|\frac{\partial \gamma^i(x_r^n, z_2)}{\partial x^{u}}\Big||\gamma_3^u(x_{\kappa(n,r)}^n,z_1)|\Big) N(dr, dz_1)\Big]^2 \nu(dz_2)
\end{align*}
which due to H\"{o}lder's inequality, Fubini Theorem and Lemma \ref{lem:BGD:jump} gives,
\begin{align*}
&E\int_Z|\gamma^i(x_t^n, z_2)-\tilde{\gamma}^i(x_{\kappa(n,t)}^n, z_2)|^2 \nu(dz_2)
\\
& \leq Kn^{-1} E\sum_{u=1}^d \int_{\kappa(n,t)}^t \Big(\int_Z \Big|\frac{\partial \gamma^i(x_r^n, z_2)}{\partial x^u}\Big|^2\nu(dz_2)\Big) |\tilde{b}^{n,u}(x_{\kappa(n,r)}^n)|^2 dr \notag
\\
&+K n^{-1}  E\sum_{u_1,u_2=1}^{d} \sum_{l_1=1}^{m}\int_{\kappa(n,t)}^t \Big(\int_Z \Big|\frac{\partial^2 \gamma^i(x_r^n, z_2)}{\partial x^{u_1} \partial x^{u_2}} \Big|^2 \nu(dz_2) \Big)\big|\tilde{\sigma}^{(u_1,l_1)}(x_{\kappa(n,r)}^n)\tilde{\sigma}^{(u_2,l_1)}(x_{\kappa(n,r)}^n)\big|^2 dr  \notag
\\
&+K E\sum_{u=1}^d \sum_{j=1}^m \int_{\kappa(n,t)}^t \Big(\int_Z\Big|\frac{\partial \gamma^i(x_r^n, z_2)}{\partial x^u} - \frac{\partial \gamma^i(x_{\kappa(n,r)}^n, z_2)}{\partial x^u} \Big|^2\nu(dz_2) \Big) |\sigma^{(u,j)}(x_{\kappa(n,r)}^n)|^2dr
\\
& + K E\sum_{u=1}^d \sum_{j=1}^m \int_{\kappa(n,t)}^t \Big(\int_Z\Big|\frac{\partial \gamma^i(x_r^n, z_2)}{\partial x^u} \Big|^2\nu(dz_2)\Big) \big|\sigma_1^{(u,j)}(x_{\kappa(n,r)}^n) \big|^2 dr \notag
\\
&+K E\sum_{u=1}^d \sum_{j=1}^m \int_{\kappa(n,t)}^t \Big(\int_Z\Big|\frac{\partial \gamma^i(x_r^n, z_2)}{\partial x^u} \Big|^2 \nu(dz_2)\Big)\big|\sigma_2^{(u,j)}(x_{\kappa(n,r)}^n)\big|^2dr
\\
& + K E\sum_{u=1}^d \sum_{j=1}^m \int_{\kappa(n,t)}^t \Big(\int_Z\Big|\frac{\partial \gamma^i(x_r^n, z_2)}{\partial x^u} \Big|^2 \nu(dz_2)\Big)|\sigma_3^{(u,j)}(x_{\kappa(n,r)}^n)|^2 dr \notag
\\
&+K E\sum_{u=1}^d \int_{\kappa(n,t)}^t \int_Z \Big(\int_Z\Big|\frac{\partial \gamma^i(x_r^n, z_2)}{\partial x^u} - \frac{\partial \gamma^i(x_{\kappa(n,r)}^n, z_2)}{\partial x^u} \Big|^2 \nu(dz_2) \Big)|\gamma^{u}(x_{\kappa(n,r)}^n,z_1)|^2  \nu(dz_1) dr
\\
& + K E \sum_{u=1}^d \int_{\kappa(n,t)}^t \int_Z \Big(\int_Z\Big|\frac{\partial \gamma^i(x_r^n, z_2)}{\partial x^u}\Big|^2 \nu(dz_2)\Big)|\gamma_1^{u}(x_{\kappa(n,r)}^n,z_1)|^2 \nu(dz_1) dr
\\
& + K E\sum_{u=1}^d \int_{\kappa(n,t)}^t \int_Z \Big(\int_Z\Big|\frac{\partial \gamma^i(x_r^n, z_2)}{\partial x^u}\Big|^2 \nu(dz_2) \Big) |\gamma_2^{u}(x_{\kappa(n,r)}^n,z_1)|^2  \nu(dz_1) dr  \notag
\\
&+K E\sum_{u=1}^d \int_{\kappa(n,t)}^t \int_Z \Big(\int_Z\Big|\frac{\partial \gamma^i(x_r^n, z_2)}{\partial x^u}\Big|^2 \nu(dz_2)\Big) |\gamma_3^{u}(x_{\kappa(n,r)}^n,z_1)|^2\nu(dz_1) dr \notag
\\
& +K E\int_{\kappa(n,t)}^t \int_Z \Big[\Big(\int_Z|\gamma^i(x_r^n+\tilde{\gamma}(x_{\kappa(n,r)}^n,z_1), z_2)-\gamma^{i}(x_{\kappa(n,r)}^n+\gamma(x_{\kappa(n,r)}^n, z_1), z_2)|^2 \nu(dz_2)\Big)
\\
& \qquad+ \Big(\int_Z |\gamma^i(x_r^n, z_2)-\gamma^{i}(x_{\kappa(n,r)}^n, z_2)|^2 \nu(dz_2)\Big)
\end{align*}
\begin{align*}
&+ \sum_{u=1}^d \Big(\int_Z \Big|\frac{\partial \gamma^{i}(x_r^n,z_2)}{\partial x^u}-\frac{\partial \gamma^{i}(x_{\kappa(n,r)}^n,z_2)}{\partial x^u}\Big|\nu(dz_2)\Big) |\gamma^u(x_{\kappa(n,r)}^n,z_1)|^2
\\
&+ \sum_{u=1}^{d}\Big(\int_Z\Big|\frac{\partial \gamma^i(x_r^n, z_2)}{\partial x^{u}}\Big|^2\nu(dz_2)\Big) |\gamma_1^{u}(x_{\kappa(n,r)}^n,z_1)|^2
\\
& + \sum_{u=1}^{d}\Big(\int_Z\Big|\frac{\partial \gamma^i(x_r^n, z_2)}{\partial x^{u}}\Big|^2\nu(dz_2)\Big) |\gamma_2^{u}(x_{\kappa(n,r)}^n,z_1)|^2
\\
&+ \sum_{u=1}^{d}\Big(\int_Z\Big|\frac{\partial \gamma^i(x_r^n, z_2)}{\partial x^{u}}\Big|^2\nu(dz_2)\Big) |\gamma_3^{u}(x_{\kappa(n,r)}^n,z_1)|^2\Big] \nu(dz_1) dr.
\end{align*}
Thus,  due to Assumption A-3 and Remark \ref{rem:der:bounded},
\begin{align}  \label{eq:star}
E\int_Z|\gamma&(x_t^n, z_2)-\tilde{\gamma}(x_{\kappa(n,t)}^n, z_2)|^2\nu(dz_2)\notag
\\
& \leq K n^{-1} E\int_{\kappa(n,t)}^t |\tilde{b}^{n}(x_{\kappa(n,r)}^n)|^2dr+K n^{-1}  E\int_{\kappa(n,t)}^t |\tilde \sigma(x_{\kappa(n,r)}^n)|^2 dr\notag
\\
&+K E\int_{\kappa(n,t)}^t  |\sigma(x_{\kappa(n,r)}^n)|^2|x_r^n-x_{\kappa(n,r)}^n|^2 dr \notag
\\
& +K E\int_{\kappa(n,t)}^t |\sigma_1(x_{\kappa(n,r)}^n)|^2 dr +K E\int_{\kappa(n,t)}^t |\sigma_2(x_{\kappa(n,r)}^n)|^2 dr \notag
\\
& +KE\int_{\kappa(n,t)}^t  |\sigma_3(x_{\kappa(n,r)}^n)|^2dr \notag
\\
& +K\int_{\kappa(n,t)}^t  \int_Z |x_r^n-x^n_{\kappa(n,r)}|^2 |\gamma(x_{\kappa(n,r)}^n,z_2)|^2  \nu(dz_2) dr \notag
\\
&+K E\int_{\kappa(n,t)}^t \int_Z  |\gamma_1(x_{\kappa(n,r)}^n,z_2)|^2 \nu(dz_2)dr \notag
\\
& +KE\int_{\kappa(n,t)}^t \int_Z  |\gamma_2(x_{\kappa(n,r)}^n,z_2)|^2 \nu(dz_2)dr \notag
\\
&+K E\int_{\kappa(n,t)}^t \int_Z |\gamma_3(x_{\kappa(n,r)}^n,z_2)|^2 \nu(dz_2)dr
\end{align}
for any $t \in [0,T]$.  Finally, one completes the proof by using Remarks [\ref{as:sde:growth},  \ref{as:scheme:a^n}, \ref{as:sde:poly}], Lemmas [\ref{lem:b1:rate}, \ref{lem:b2:rate}, \ref{lem:b3:rate},  \ref{lem:c1:rate}, \ref{lem:c2:rate}, \ref{lem:c3:rate},  \ref{lem:scheme:momentbound}, \ref{lem:one-step:rate}, \ref{lem:one-step:sigma:rate}, \ref{lem:one-step:gamma:rate}] along with Corollaries [\ref{lem:tildeb:momentbound}, \ref{lem:tilde c:momentbound}].
\end{proof}
%-----------------------------------------------------------------
%                Lemma : b - tilde b: Rate
%-----------------------------------------------------------------
\begin{lem} \label{lem:b-tilde b:rate}
Let Assumptions A-1 to %A-4 and
A-5 hold, then
\begin{align*}
\sup_{0 \leq t \leq T}E|\sigma(x_{t}^n)& - \tilde{\sigma}(x_{\kappa(n,t)}^n)|^2\leq K n^{-2}.
\end{align*}
\end{lem}
\begin{proof} The lemma can be proved by adopting  similar arguments as used in the proof of Lemma \ref{c-tilde c:rate}.
\end{proof}
%-----------------------------------------------------------------
%                Lemma : a - tilde a: Rate
%-----------------------------------------------------------------
Before proceeding further, let us define $e_t^n := x_t-x_t^n$ i.e.
\begin{align*}
e_t^n & =  \int_{0}^{t}(b(x_s)-\tilde{b}^n(x_{\kappa(n,s)}^n))ds+\int_{0}^{t}(\sigma(x_s)-\tilde{\sigma}(x_{\kappa(n,s)}^n))dw_s
\\
& +\int_{0}^{t}\int_Z (\gamma(x_s,z_2)-\tilde \gamma(x_{\kappa(n,s)}^n,z_2))\tilde N(ds, dz_2)
\end{align*}
almost surely for any $t \in [0, T]$ and $n \in \mathbb{N}$. Also, the notation $e_t^{n,k}$ stands for the $k$-th element of $e_t^n$ for every $k=1,\ldots,d$, $t \in [0,T]$ and $n \in \mathbb{N}$.
\begin{lem} \label{lem:a-tilde a:rate}
Let Assumptions A-1 to %A-4 and
A-5 hold, then
\begin{equation*}
E\int_{0}^{t} e_s^n(b(x_s^n)-b(x_{\kappa(n,s)}^n))ds\leq K  \int_{0}^t \sup_{0 \leq r \leq s}E|e_r^n|^2   ds + K n^{-\frac{2}{2+\delta}-1}
\end{equation*}
for any $t \in [0,T]$.
\end{lem}
\begin{proof}
First, one applies It\^{o}'s formula to obtain $b^k(x_s^n)-b^k(x_{\kappa(n,s)}^n)$ and then writes the following,
%By It\^{o}'s formula,
%\begin{align}
%b^k(x_s^n)&-b^k(x_{\kappa(n,s)}^n) =\sum_{i=1}^d \int_{\kappa(n,s)}^s \frac{\partial b^k(x_r^n)}{\partial x^i} \tilde{b}^{n,i}(x_{\kappa(n,r)}^n)dr \notag
%\\
%&+\frac{1}{2}\sum_{i,j=1}^{d}\int_{\kappa(n,s)}^s \frac{\partial^2 b^k(x_r^n)}{\partial x^i \partial x^j} \sum_{l=1}^{m} \tilde{\sigma}^{(i,l)}(x_{\kappa(n,r)}^n)\tilde{\sigma}^{(j,l)}(x_{\kappa(n,r)}^n)dr \notag
%\\
%&+\sum_{i=1}^d \int_{\kappa(n,s)}^s \frac{\partial b^k(x_r^n)}{\partial x^i} \sum_{l=1}^m \tilde{\sigma}^{(i,l)}(x_{\kappa(n,r)}^n)dw_r^l +\sum_{i=1}^d \int_{\kappa(n,s)}^s \int_Z \frac{\partial b^k(x_r^n)}{\partial x^i} \tilde{\gamma}^{i}(x_{\kappa(n,r)}^n,z_2)\tilde N(dr, dz_2) \notag
%\\
%&+ \int_{\kappa(n,s)}^s \int_Z \Big(b^k(x_r^n+\tilde{\gamma}(x_{\kappa(n,r)}^n,z_2))-b^k(x_r^n)-\sum_{i=1}^{d}\frac{\partial b^k(x_r^n)}{\partial x^{i}}{\tilde{\gamma}}^{i}(x_{\kappa(n,r)}^n,z_2)\Big)N(dr, dz_2) \label{eq:a:ito}
%\end{align}
%and thus one could write
\begin{align}  \label{eq:F1+.+F5}
E\int_{0}^t & e_s^n(b(x_s^n)-b(x_{\kappa(n,s)}^n))ds=E\sum_{k=1}^d\int_{0}^t e_s^{n,k}\{b^k(x_s^n)-b^k(x_{\kappa(n,s)}^n)\}ds \notag
\\
= &  E\sum_{k,i=1}^d  \int_{0}^t e_s^{n,k} \int_{\kappa(n,s)}^s \frac{\partial b^k(x_r^n)}{\partial x^i} \tilde{b}^{n,i}(x_{\kappa(n,r)}^n)dr ds \notag
\\
 &+\frac{1}{2}\sum_{k,i,j=1}^{d} E\int_{0}^t e_s^{n,k}  \int_{\kappa(n,s)}^s \frac{\partial^2 b^k(x_r^n)}{\partial x^i \partial x^j} \sum_{l=1}^{m} \tilde{\sigma}^{(i,l)}(x_{\kappa(n,r)}^n)\tilde{\sigma}^{(j,l)}(x_{\kappa(n,r)}^n)dr ds \notag
 \\
 & +\sum_{k,i=1}^d E \int_{0}^t e_s^{n,k}\int_{\kappa(n,s)}^s \frac{\partial b^k(x_r^n)}{\partial x^i} \sum_{l=1}^m \tilde{\sigma}^{(i,l)}(x_{\kappa(n,r)}^n)dw_r^l ds \notag
\\
&+ \sum_{k,i=1}^d E\int_{0}^t  e_s^{n,k} \int_{\kappa(n,s)}^s \int_Z \frac{\partial b^k(x_r^n)}{\partial x^i} \tilde{\gamma}^{i}(x_{\kappa(n,r)}^n,z_2)\tilde N(dr, dz_2) ds \notag
\\
&+ \sum_{k=1}^d E\int_{0}^t  e_s^{n,k} \int_{\kappa(n,s)}^s \int_Z \Big(b^k(x_r^n+\tilde{\gamma}(x_{\kappa(n,r)}^n,z_2)) \notag
\\
& \qquad\qquad -b^k(x_r^n)-\sum_{i=1}^{d}\frac{\partial b^k(x_r^n)}{\partial x^{i}}{\tilde{\gamma}}^{i}(x_{\kappa(n,r)}^n,z_2)\Big)N(dr, dz_2) ds \notag
\\
=:&  F_1+F_2+F_3+F_4+F_5
\end{align}
for any $t \in [0,T]$. By using Young's and H\"older's inequalities, $F_1$ can be estimated as
\begin{align*}
F_1 &:=  \sum_{k,i=1}^d  E \int_{0}^t e_s^{n,k} \int_{\kappa(n,s)}^s \frac{\partial b^k(x_r^n)}{\partial x^i} \tilde{b}^{n,i}(x_{\kappa(n,r)}^n)dr ds
\\
& \leq  K  E \int_{0}^t  |e_s^{n}|^2 ds +K n^{-1}\sum_{k,i=1}^d  E \int_{0}^t  \int_{\kappa(n,s)}^s \Big|\frac{\partial b^k(x_r^n)}{\partial x^i}\Big|^2 |\tilde{b}^{n,i}(x_{\kappa(n,r)}^n)|^2 dr ds
\end{align*}
and then due to Remarks [\ref{as:scheme:a^n}, \ref{as:sde:poly:Der}, \ref{as:sde:poly}] and Lemma \ref{lem:scheme:momentbound}, one obtains
\begin{align}
F_1 \leq K   \int_{0}^t  \sup_{0 \leq r \leq s}E |e_r^{n}|^2 ds +K n^{-2}. \label{eq:F1}
\end{align}
Similarly, $F_2$ can be estimated by using Young's and H\"older's inequalities along with Remarks [\ref{as:sde:growth},  \ref{as:sde:poly:Der}], Corollary \ref{lem:tildeb:momentbound} and Lemma \ref{lem:scheme:momentbound} as,
\begin{align}
F_2&:=\frac{1}{2}\sum_{k,i,j=1}^{d}E\int_{0}^t e_s^{n,k}  \int_{\kappa(n,s)}^s \frac{\partial^2 b^k(x_r^n)}{\partial x^i \partial x^j} \sum_{l=1}^{m} \tilde{\sigma}^{(i,l)}(x_{\kappa(n,r)}^n)\tilde{\sigma}^{(j,l)}(x_{\kappa(n,r)}^n)dr ds \notag
\\
& \leq K   \int_{0}^t  \sup_{0 \leq r \leq s}E |e_r^{n}|^2 ds +K n^{-2}. \label{eq:F2}
\end{align}
For the estimation of $F_3$, one uses
\begin{align*}
e_s^{n,k}=&e_{\kappa(n,s)}^{n,k}+ \int_{\kappa(n,s)}^s (b^k(x_r)-\tilde{b}^{n,k}(x_{\kappa(n,r)}^n) ) dr
\\
& + \sum_{v=1}^m\int_{\kappa(n,s)}^s\big(\sigma^{(k,v)}(x_r)-\tilde{\sigma}^{(k,v)}(x_{\kappa(n,r)}^n)\big)dw_r^v \notag
\\
&\hspace{1cm}+\int_{\kappa(n,s)}^s \int_Z \big(\gamma^k(x_r, z_2)-\tilde{\gamma}^{k}(x_{\kappa(n,r)}^n, z_2)\big)\tilde N(dr,dz_2)
\end{align*}
to obtain the following,
\begin{align}
F_3&  :=K\sum_{k,i=1}^d  \sum_{l=1}^{m} E \int_{0}^t e_s^{n,k}\int_{\kappa(n,s)}^s\frac{\partial b^k(x_r^n)}{\partial x^i}  \tilde{\sigma}^{(i,l)}(x_{\kappa(n,r)}^n)dw_r^l ds  \notag
\\
& \leq  K \sum_{k,i=1}^d   \sum_{l=1}^{m} E \int_{0}^t e_{\kappa(n,s)}^{n,k}  \int_{\kappa(n,s)}^s \frac{\partial b^k(x_r^n)}{\partial x^i}  \tilde{\sigma}^{(i,l)}(x_{\kappa(n,r)}^n)dw_r^l ds \notag
\\
 & + K \sum_{k,i=1}^d  \sum_{l=1}^{m} E \int_{0}^t \int_{\kappa(n,s)}^s (b^k(x_r)-\tilde{b}^{n,k}(x_{\kappa(n,r)}^n) ) dr \notag
 \\
& \qquad \times \int_{\kappa(n,s)}^s\frac{\partial b^k(x_r^n)}{\partial x^i}  \tilde{\sigma}^{(i,l)}(x_{\kappa(n,r)}^n)dw_r^l ds \notag
\\
&+ K \sum_{k,i=1}^d   \sum_{l,v=1}^{m} E \int_{0}^t \int_{\kappa(n,s)}^s(\sigma^{(k,v)}(x_r)-\tilde{\sigma}^{(k,v)}(x_{\kappa(n,r)}^n))dw_r^v  \notag
\\
& \qquad \times \int_{\kappa(n,s)}^s\frac{\partial b^k(x_r^n)}{\partial x^i}  \tilde{\sigma}^{(i,l)}(x_{\kappa(n,r)}^n)dw_r^l ds \notag
\\
& +K \sum_{k,i=1}^d  \sum_{l=1}^{m} E \int_{0}^t \int_{\kappa(n,s)}^s \int_Z \big(\gamma^k(x_r, z_2)-\tilde{\gamma}^{k}(x_{\kappa(n,r)}^n, z_2)\big)\tilde{N}(dr,dz_2)\notag
 \\
& \qquad \times \int_{\kappa(n,s)}^s\frac{\partial b^k(x_r^n)}{\partial x^i}  \tilde{\sigma}^{(i,l)}(x_{\kappa(n,r)}^n)dw_r^l ds \notag
\\
&=:  F_{31}+F_{32}+F_{33}+F_{34}. \label{eq:F31+.+E34}
\end{align}
Here, $F_{31}$ is given by
\begin{align*}
F_{31}:= & K \sum_{k,i=1}^d   \sum_{l=1}^{m} E \int_{0}^t e_{\kappa(n,s)}^{n,k}  \int_{\kappa(n,s)}^s \frac{\partial b^k(x_r^n)}{\partial x^i}  \tilde{\sigma}^{(i,l)}(x_{\kappa(n,r)}^n)dw_r^l ds  \notag
\\
= & K \sum_{k,i=1}^d   \sum_{l=1}^{m} E \int_{0}^t e_{\kappa(n,s)}^{n,k}  \int_{\kappa(n,s)}^s \frac{\partial b^k(x_r^n)}{\partial x^i}  \sigma^{(i,l)}(x_{\kappa(n,r)}^n)dw_r^l ds
\\
+& K \sum_{k,i=1}^d   \sum_{l=1}^{m} E \int_{0}^t e_{\kappa(n,s)}^{n,k}  \int_{\kappa(n,s)}^s \frac{\partial b^k(x_r^n)}{\partial x^i}  \sigma_1^{(i,l)}(x_{\kappa(n,r)}^n)dw_r^l ds
\\
+& K \sum_{k,i=1}^d   \sum_{l=1}^{m} E \int_{0}^t e_{\kappa(n,s)}^{n,k}  \int_{\kappa(n,s)}^s \frac{\partial b^k(x_r^n)}{\partial x^i}  \sigma_2^{(i,l)}(x_{\kappa(n,r)}^n)dw_r^l ds
\\
+& K \sum_{k,i=1}^d   \sum_{l=1}^{m} E \int_{0}^t e_{\kappa(n,s)}^{n,k}  \int_{\kappa(n,s)}^s \frac{\partial b^k(x_r^n)}{\partial x^i}  \sigma_3^{(i,l)}(x_{\kappa(n,r)}^n)dw_r^l ds
\end{align*}
which on the application of Young's inequality gives,
\begin{align}
F_{31} & \leq   K  E \int_{0}^t |e_{\kappa(n,s)}^{n}|^2 ds +  K \sum_{k,i=1}^d   \sum_{l=1}^{m} E \int_{0}^t \Big|\int_{\kappa(n,s)}^s \frac{\partial b^k(x_r^n)}{\partial x^i}  \sigma_1^{(i,l)}(x_{\kappa(n,r)}^n)dw_r^l\Big|^2 ds  \notag
\\
&\quad +  K \sum_{k,i=1}^d   \sum_{l=1}^{m} E \int_{0}^t \Big|\int_{\kappa(n,s)}^s \frac{\partial b^k(x_r^n)}{\partial x^i}  \sigma_2^{(i,l)}(x_{\kappa(n,r)}^n)dw_r^l\Big|^2 ds  \notag
\\
& \quad+  K \sum_{k,i=1}^d   \sum_{l=1}^{m} E \int_{0}^t \Big|\int_{\kappa(n,s)}^s \frac{\partial b^k(x_r^n)}{\partial x^i}  \sigma_3^{(i,l)}(x_{\kappa(n,r)}^n)dw_r^l\Big|^2 ds  \notag
\end{align}
and then on using an elementary inequality of stochastic integrals along with Remark \ref{as:sde:poly:Der}, one obtains
\begin{align}
F_{31} & \leq   K  E \int_{0}^t |e_{\kappa(n,s)}^{n}|^2 ds +  K  E \int_{0}^t \int_{\kappa(n,s)}^s (1+|x_r^n|^\chi)^2  |\sigma_1(x_{\kappa(n,r)}^n)|^2 dr ds  \notag
\\
& \quad +  K  E \int_{0}^t \int_{\kappa(n,s)}^s (1+|x_r^n|^\chi)^2  |\sigma_2(x_{\kappa(n,r)}^n)|^2 dr ds   \notag
\\
& +  K  E \int_{0}^t \int_{\kappa(n,s)}^s (1+|x_r^n|^\chi)^2  |\sigma_3(x_{\kappa(n,r)}^n)|^2 dr ds.  \notag
\end{align}
Also, by using H\"older's inequality and Lemmas [\ref{lem:b1:rate}, \ref{lem:b2:rate}, \ref{lem:b3:rate}, \ref{lem:scheme:momentbound}],
\begin{align} \label{eq:F31}
F_{31} & \leq   K  E \int_{0}^t |e_{\kappa(n,s)}^{n}|^2 ds \notag
\\
& +  K  \int_{0}^t \int_{\kappa(n,s)}^s (E (1+|x_r^n|^\chi)^\frac{2(2+\delta)}{\delta})^\frac{\delta}{2+\delta}  (E|\sigma_1(x_{\kappa(n,r)}^n)|^{2+\delta})^\frac{2}{2+\delta} dr ds  \notag
\\
& \quad +  K  \int_{0}^t \int_{\kappa(n,s)}^s (E (1+|x_r^n|^\chi)^\frac{2(2+\delta)}{\delta})^\frac{\delta}{2+\delta}  (E|\sigma_2(x_{\kappa(n,r)}^n)|^{2+\delta})^\frac{2}{2+\delta} dr ds   \notag
\\
&\quad+  K  \int_{0}^t \int_{\kappa(n,s)}^s (E (1+|x_r^n|^\chi)^\frac{2(2+\delta)}{\delta})^\frac{\delta}{2+\delta}  (E|\sigma_3(x_{\kappa(n,r)}^n)|^{2+\delta})^\frac{2}{2+\delta} dr ds  \notag
\\
& \leq   K  \int_{0}^t \sup_{0 \leq r \leq s}E |e_r^{n}|^2 ds +  K n^{-\frac{2}{2+\delta}-1}.
\end{align}
Further, due to Remark \ref{as:sde:poly},
\begin{align} \label{eq:b:split:new}
|b^k(x_r)-\tilde{b}^{n,k}(x_{\kappa(n,r)}^n)| &\leq  L(1+|x_r|^\chi+|x_r^n|^\chi) |x_r-x_r^n|  \notag
\\
& +L(1+|x_r^n|^\chi+|x_{\kappa(n,r)}^n|^\chi) |x_r^n-x_{\kappa(n,r)}^n| \notag
\\
&+|b^k(x_{\kappa(n,r)}^n)-\tilde{b}^{n,k}(x_{\kappa(n,r)}^n)|
\end{align}
for any $r \in [0,T]$, $k=1,\ldots,d$,  which on substituting in $F_{32}$ gives
\begin{align*}
F_{32}& :=K \sum_{k,i=1}^d  \sum_{l=1}^{m} E \int_{0}^t \int_{\kappa(n,s)}^s (b^k(x_r)-\tilde{b}^{n,k}(x_{\kappa(n,r)}^n) ) dr
\\
& \qquad \times \int_{\kappa(n,s)}^s \frac{\partial b^k(x_r^n)}{\partial x^i}  \tilde{\sigma}^{(i,l)}(x_{\kappa(n,r)}^n)dw_r^l ds \notag
\\
& \leq   K \sum_{k,i=1}^d  \sum_{l=1}^{m} E \int_{0}^t\int_{\kappa(n,s)}^s(1+|x_r|^\chi+|x_r^n|^\chi) |x_r-x_r^n| dr
\\
& \qquad \times \Big| \int_{\kappa(n,s)}^s \frac{\partial b^k(x_r^n)}{\partial x^i}  \tilde{\sigma}^{(i,l)}(x_{\kappa(n,r)}^n)dw_r^l \Big|ds \notag
\\
& +K \sum_{k,i=1}^d  \sum_{l=1}^{m} E \int_{0}^t\int_{\kappa(n,s)}^s (1+|x_r^n|^\chi+|x_{\kappa(n,r)}^n|^\chi) |x_r^n-x_{\kappa(n,r)}^n| dr\Big|
\\
&\qquad \times \int_{\kappa(n,s)}^s \frac{\partial b^k(x_r^n)}{\partial x^i}  \tilde{\sigma}^{(i,l)}(x_{\kappa(n,r)}^n)dw_r^l \Big|ds \notag
\\
& +K \sum_{k,i=1}^d  \sum_{l=1}^{m} E \int_{0}^t\int_{\kappa(n,s)}^s|b^k(x_{\kappa(n,r)}^n)-\tilde{b}^{n,k}(x_{\kappa(n,r)}^n)| dr \Big|
\\
&\qquad \times\int_{\kappa(n,s)}^s \frac{\partial b^k(x_r^n)}{\partial x^i}  \tilde{\sigma}^{(i,l)}(x_{\kappa(n,r)}^n)dw_r^l\Big| ds.
\end{align*}
Now, one uses H\"{o}lder's inequality with exponent $\sqrt{2}$ and $2+\sqrt{2}$ to obtain,
\begin{align*}
& F_{32}\leq   K \sum_{k,i=1}^d  \sum_{l=1}^{m}  \int_{0}^t \Big(n^{-\sqrt{2}+1}E \int_{\kappa(n,s)}^s(1+|x_r|^\chi+|x_r^n|^\chi)^{\sqrt{2}} |x_r-x_r^n|^{\sqrt{2}} dr\Big)^{\frac{1}{\sqrt{2}}} \notag
\\
& \qquad \times \Big(E\Big|\int_{\kappa(n,s)}^s \frac{\partial b^k(x_r^n)}{\partial x^i}  \tilde{\sigma}^{(i,l)}(x_{\kappa(n,r)}^n)dw_r^l\Big|^{2+\sqrt{2}}\Big)^\frac{1}{2+\sqrt{2}} ds \notag
\\
& +K \sum_{k,i=1}^d  \sum_{l=1}^{m}  \int_{0}^t \Big(n^{-\sqrt{2}+1} E\int_{\kappa(n,s)}^s (1+|x_r^n|^\chi+|x_{\kappa(n,r)}^n|^\chi)^{\sqrt{2}} |x_r^n-x_{\kappa(n,r)}^n|^{\sqrt{2}} dr \Big)^{\frac{1}{\sqrt{2}}} \notag
\\
& \qquad \times \Big(E\Big|\int_{\kappa(n,s)}^s \frac{\partial b^k(x_r^n)}{\partial x^i}  \tilde{\sigma}^{(i,l)}(x_{\kappa(n,r)}^n)dw_r^l\Big|^{2+\sqrt{2}}\Big)^\frac{1}{2+\sqrt{2}} ds \notag
\\
& +K \sum_{k,i=1}^d  \sum_{l=1}^{m}  \int_{0}^t \Big(n^{-1} E\int_{\kappa(n,s)}^s|b^k(x_{\kappa(n,r)}^n)-\tilde{b}^{n,k}(x_{\kappa(n,r)}^n)|^2 dr
\\
& \qquad \times E\Big|\int_{\kappa(n,s)}^s \frac{\partial b^k(x_r^n)}{\partial x^i}  \tilde{\sigma}^{(i,l)}(x_{\kappa(n,r)}^n)dw_r^l\Big|^2 \Big)^\frac{1}{2}ds \notag
\end{align*}
which on the application of H\"older's inequality along with Remark  \ref{as:sde:poly:Der}, Lemmas [\ref{lem:sde:momentbound},  \ref{lem:scheme:momentbound}, \ref{lem:form:rate}] and Corollary \ref{lem:tildeb:momentbound} yields the following,
\begin{align*}
F_{32}& \leq   K n^{-\frac{1}{2}} \int_{0}^t \Big(n^{-\sqrt{2}+1} \int_{\kappa(n,s)}^s \big(E(1+|x_r|^\chi+|x_r^n|^\chi)^{\sqrt{2}(2+\sqrt{2})}\big)^\frac{1}{2+\sqrt{2}}
\\
& \qquad \times \big(E|e_r^n|^2\big)^\frac{1}{\sqrt{2}} dr\Big)^{\frac{1}{\sqrt{2}}} ds \notag
\\
& +K n^{-\frac{1}{2}}  \int_{0}^t \Big(n^{-\sqrt{2}+1} \int_{\kappa(n,s)}^s \big(E(1+|x_r^n|^\chi+|x_{\kappa(n,r)}^n|^\chi)^{\sqrt{2}(2+\sqrt{2})}\big)^\frac{1}{2+\sqrt{2}}
\\
&\qquad \times\big(E|x_r^n-x_{\kappa(n,r)}^n|^2\big)^\frac{1}{\sqrt{2}} dr \Big)^{\frac{1}{\sqrt{2}}}  ds  +K  n^{-\frac{5}{2}}\notag
\\
& \leq   K n^{-\frac{3}{2}} \int_{0}^t \big(\sup_{0 \leq r \leq s}E|e_r^n|^2\big)^\frac{1}{2}  ds +K n^{-2} +K  n^{-\frac{5}{2}}. \notag
\end{align*}
Thus, due to Young's inequality, one obtains
\begin{align} \label{eq:F32}
F_{32}\leq   K \int_{0}^t \sup_{0 \leq r \leq s}E|e_r^n|^2 ds + K n^{-2}.
\end{align}
Again, for any $r \in [0,T]$, $k=1,\ldots,d$ and $v=1,\ldots,m$,  one uses
\begin{align} \label{eq:sigma:split:new}
\sigma^{(k,v)}&(x_r)-\tilde{\sigma}^{(k,v)}(x_{\kappa(n,r)}^n)=(\sigma^{(k,v)}(x_r)-\sigma^{(k,v)}(x_r^n)) \notag
\\
& \qquad \qquad +(\sigma^{(k,v)}(x_r^n)-\tilde{\sigma}^{(k,v)}(x_{\kappa(n,r)}^n))
\end{align}
to express $F_{33}$ as below,
\begin{align*}
F_{33}&  := K \sum_{k,i=1}^d   \sum_{l,v=1}^{m} E \int_{0}^t \int_{\kappa(n,s)}^s(\sigma^{(k,v)}(x_r)-\tilde{\sigma}^{(k,v)}(x_{\kappa(n,r)}^n))dw_r^v
\\
&\qquad \times \int_{\kappa(n,s)}^s \frac{\partial b^k(x_r^n)}{\partial x^i}  \tilde{\sigma}^{(i,l)}(x_{\kappa(n,r)}^n)dw_r^l ds
\\
& = K \sum_{k,i=1}^d   \sum_{l,v=1}^{m} E \int_{0}^t \int_{\kappa(n,s)}^s(\sigma^{(k,v)}(x_r)-{\sigma}^{(k,v)}(x_r^n))dw_r^v
\\
&\qquad \times \int_{\kappa(n,s)}^s \frac{\partial b^k(x_r^n)}{\partial x^i}  \tilde{\sigma}^{(i,l)}(x_{\kappa(n,r)}^n)dw_r^l ds
\\
&+ K \sum_{k,i=1}^d   \sum_{l,v=1}^{m} E \int_{0}^t \int_{\kappa(n,s)}^s(\sigma^{(k,v)}(x_r^n)-\tilde{\sigma}^{(k,v)}(x_{\kappa(n,r)}^n))dw_r^v
\\
&\qquad \times \int_{\kappa(n,s)}^s \frac{\partial b^k(x_r^n)}{\partial x^i}  \tilde{\sigma}^{(i,l)}(x_{\kappa(n,r)}^n)dw_r^l ds.
\end{align*}
Further, H\"{o}lder's inequality implies
\begin{align*}
F_{33} & \leq K \sum_{k,i=1}^d   \sum_{l,v=1}^{m}  \int_{0}^t \Big(E\Big|\int_{\kappa(n,s)}^s(\sigma^{(k,v)}(x_r)-{\sigma}^{(k,v)}(x_r^n))dw_r^v\Big|^2
\\
&\qquad \times E\Big| \int_{\kappa(n,s)}^s \frac{\partial b^k(x_r^n)}{\partial x^i}  \tilde{\sigma}^{(i,l)}(x_{\kappa(n,r)}^n)dw_r^l\Big|^2\Big)^{\frac{1}{2}} ds
\\
&+ K \sum_{k,i=1}^d   \sum_{l,v=1}^{m}  \int_{0}^t \Big(E\Big|\int_{\kappa(n,s)}^s(\sigma^{(k,v)}(x_r^n)-\tilde{\sigma}^{(k,v)}(x_{\kappa(n,r)}^n))dw_r^v\Big|^2
\\
&\qquad \times E\Big| \int_{\kappa(n,s)}^s \frac{\partial b^k(x_r^n)}{\partial x^i}  \tilde{\sigma}^{(i,l)}(x_{\kappa(n,r)}^n)dw_r^l\Big|^2\Big)^\frac{1}{2} ds
\end{align*}
which again due to H\"older's inequality, Remark  \ref{as:sde:poly:Der}, Lemma \ref{lem:scheme:momentbound}  and Corollary \ref{lem:tildeb:momentbound} yields the following,
\begin{align*}
F_{33} & \leq K n^{-\frac{1}{2}}  \int_{0}^t \Big(E\int_{\kappa(n,s)}^s|\sigma(x_r)-\sigma(x_r^n)|^2dr\Big)^{\frac{1}{2}} ds
\\
& + K n^{-\frac{1}{2}} \int_{0}^t \Big(E\int_{\kappa(n,s)}^s|\sigma(x_r^n)-\tilde \sigma(x_{\kappa(n,r)}^n)|^2dr\Big)^\frac{1}{2} ds.
\end{align*}
and then Assumption A-2 and Lemma \ref{lem:b-tilde b:rate} gives
\begin{align}
F_{33}  \leq K  n^{-1} \int_{0}^t \Big( \sup_{0 \leq r \leq s} E|e_r^n|^2 \Big)^{\frac{1}{2}}  ds + K n^{-2}. \notag
\end{align}
Thus, on the application of Young's inequality, one obtains
\begin{align}
F_{33} \leq K  \int_{0}^t  \sup_{0 \leq r \leq s} E|e_r^n|^2 ds + K n^{-2}. \label{eq:F33}
\end{align}
Finally, for any $r \in [0,T]$ $k=1,\ldots,d$ and $z_2 \in Z$,  one uses
$$
\gamma^k(x_r, z_2)-\tilde{\gamma}^{k}(x_{\kappa(n,r)}^n, z_2)=(\gamma^k(x_r, z_2)-\gamma^k(x_r^n, z_2))+(\gamma^k(x_r^n, z_2)-\tilde{\gamma}^{k}(x_{\kappa(n,r)}^n, z_2))
$$
to express $F_{34}$ as the following
\begin{align}
F_{34}& := K \sum_{k,i=1}^d  \sum_{l=1}^{m} E \int_{0}^t \int_{\kappa(n,s)}^s \int_Z \big(\gamma^k(x_r, z_2)-\tilde{\gamma}^{k}(x_{\kappa(n,r)}^n, z_2)\big)\tilde{N}(dr,dz_2) \notag
\\
&\qquad \times \int_{\kappa(n,s)}^s \frac{\partial b^k(x_r^n)}{\partial x^i}  \tilde{\sigma}^{(i,l)}(x_{\kappa(n,r)}^n)dw_r^l ds \notag
\\
& =  K \sum_{k,i=1}^d  \sum_{l=1}^{m}  \int_{0}^t E\int_{\kappa(n,s)}^s \int_Z \big(\gamma^k(x_r, z_2)-\gamma^k(x_r^n, z_2)\big)\tilde{N}(dr,dz_2) \notag
\\
&\qquad \times \int_{\kappa(n,s)}^s \frac{\partial b^k(x_r^n)}{\partial x^i}  \tilde{\sigma}^{(i,l)}(x_{\kappa(n,r)}^n)dw_r^l ds \notag
\\
& + K \sum_{k,i=1}^d  \sum_{l=1}^{m}  \int_{0}^tE \int_{\kappa(n,s)}^s \int_Z \big(\gamma^k(x_r^n, z_2)-\tilde{\gamma}^{k}(x_{\kappa(n,r)}^n, z_2)\big)\tilde{N}(dr,dz_2)\notag
 \\
&\qquad \times\int_{\kappa(n,s)}^s \frac{\partial b^k(x_r^n)}{\partial x^i}  \tilde{\sigma}^{(i,l)}(x_{\kappa(n,r)}^n)dw_r^l ds \notag
\end{align}
which due to H\"older's inequality gives,
\begin{align}
F_{34}\leq & K \sum_{k,i=1}^d  \sum_{l=1}^{m}  \int_{0}^t \Big(E\Big|\int_{\kappa(n,s)}^s \int_Z \big(\gamma^k(x_r, z_2)-\gamma^k(x_r^n, z_2)\big)\tilde{N}(dr,dz_2)\Big|^2 \notag
\\
& \qquad \times E\Big| \int_{\kappa(n,s)}^s \frac{\partial b^k(x_r^n)}{\partial x^i}  \tilde{\sigma}^{(i,l)}(x_{\kappa(n,r)}^n)dw_r^l\Big|^2\Big)^\frac{1}{2} ds \notag
\\
& + K \sum_{k,i=1}^d  \sum_{l=1}^{m}  \int_{0}^t \Big(E\Big|\int_{\kappa(n,s)}^s \int_Z \big(\gamma^k(x_r^n, z_2)-\tilde{\gamma}^{k}(x_{\kappa(n,r)}^n, z_2)\big)\tilde{N}(dr,dz_2)\Big|^2\notag
\\
& \qquad \times E\Big| \int_{\kappa(n,s)}^s \frac{\partial b^k(x_r^n)}{\partial x^i}  \tilde{\sigma}^{(i,l)}(x_{\kappa(n,r)}^n)dw_r^l\Big|^2\Big)^\frac{1}{2} ds. \notag
\end{align}
Further, the application of Assumption A-2,  Remark \ref{as:sde:poly:Der}, Lemmas [\ref{lem:BGD:jump}, \ref{c-tilde c:rate}] and Corollary  \ref{lem:tildeb:momentbound} gives
\begin{align}
F_{34}\leq & K n^{-\frac{1}{2}} \int_{0}^t \Big(E\int_{\kappa(n,s)}^s \int_Z |\gamma(x_r, z_2)-\gamma(x_r^n, z_2)|^2\nu(dz_2)dr \Big)^\frac{1}{2}  ds \notag
\\
& + K n^{-\frac{1}{2}}  \int_{0}^t \Big(E\int_{\kappa(n,s)}^s \int_Z |\gamma(x_r^n, z_2)-\tilde{\gamma}(x_{\kappa(n,r)}^n, z_2)|^2\nu(dz_2)dr\Big)^\frac{1}{2} ds \notag
\\
\leq & K  n^{-1} \int_{0}^t \Big( \sup_{0 \leq r \leq s}E|e_r^n|^2  \Big)^\frac{1}{2}  ds  + K n^{-2}. \notag
\end{align}
Thus,  on the application of Young's inequality, one obtains
\begin{align}
F_{34} \leq  K  \int_{0}^t \sup_{0 \leq r \leq s}E|e_r^n|^2   ds  + K n^{-2}. \label{eq:F34}
\end{align}
Thus substituting the estimates from \eqref{eq:F31}, \eqref{eq:F32}, \eqref{eq:F33} and \eqref{eq:F34} in \eqref{eq:F31+.+E34}, one obtains
\begin{align}
F_3 \leq  K  \int_{0}^t \sup_{0 \leq r \leq s}E|e_r^n|^2   ds +  K n^{-\frac{2}{2+\delta}-1}. \label{eq:F3}
\end{align}
By adopting the same approach as followed in the estimation of $F_3$, one could estimate $F_4$ and $F_5$ as
\begin{align}
F_4 & \leq  K  \int_{0}^t \sup_{0 \leq r \leq s}E|e_r^n|^2   ds +  K n^{-\frac{2}{2+\delta}-1}, \label{eq:F4}
\\
F_5 & \leq  K  \int_{0}^t \sup_{0 \leq r \leq s}E|e_r^n|^2   ds +  K n^{-\frac{2}{2+\delta}-1}. \label{eq:F5}
\end{align}
Finally, combining estimates from \eqref{eq:F1}, \eqref{eq:F2}, \eqref{eq:F3}, \eqref{eq:F4} and \eqref{eq:F5} in equation \eqref{eq:F1+.+F5} completes the proof.
\end{proof}
%-----------------------------------------------------------------
%                   Main Theorem
%-----------------------------------------------------------------
\begin{proof}[\textbf{Proof of Theorem \ref{thm:main:thm}}]
By It\^{o}'s formula,
\begin{align*}
|e_t^n|^2& = 2 \int_{0}^{t} e_s^n(b(x_s)-\tilde{b}^n(x_{\kappa(n,s)}^n)) ds + 2 \int_{0}^{t} e_s^n(\sigma(x_s)-\tilde{\sigma}(x_{\kappa(n,s)}^n)) dw_s
\\
&  +   \int_{0}^{t} |\sigma(x_s)-\tilde{\sigma}(x_{\kappa(n,s)}^n)|^2 ds+ 2 \int_{0}^{t} \int_Z e_s^n  (\gamma(x_s,z_2)-\tilde \gamma(x_{\kappa(n,s)}^n,z_2))\tilde N(ds, dz_2)
\\
& +  \int_{0}^{t} \int_Z \big(  |e_s^n+\gamma(x_s,z_2)-\tilde \gamma(x_{\kappa(n,s)}^n,z_2)|^2
\\
& \qquad -|e_s^n|^2-2e_s^n(\gamma(x_s,z_2)-\tilde \gamma(x_{\kappa(n,s)}^n,z_2))\big) N(ds, dz_2)
\end{align*}
almost surely for any $t \in [0,T]$,  which on taking expectation implies
\begin{align}
E|e_t^n|^2 = & 2 E\int_{0}^{t} e_s^n\{b(x_s)-\tilde{b}^n(x_{\kappa(n,s)}^n)\} ds   +  E \int_{0}^{t} |\sigma(x_s)-\tilde{\sigma}(x_{\kappa(n,s)}^n)|^2 ds \notag
\\
 & +  E \int_{0}^{t} \int_Z \big(  |e_s^n+\gamma(x_s,z_2)-\tilde \gamma(x_{\kappa(n,s)}^n,z_2)|^2 \notag
 \\
 & \qquad-|e_s^n|^2-2e_s^n\{\gamma(x_s,z_2)-\tilde \gamma(x_{\kappa(n,s)}^n,z_2)\}\big)\nu(dz_2)ds \notag
 \\
 & \qquad \qquad \qquad =:W_1+W_2+W_3 \label{eq:W1+W2+W3}
\end{align}
for any $t \in [0, T]$.

Now, for estimating $W_1$, one uses the following,
\begin{align*}
e_s^n(b(x_s)-\tilde{b}^n(x_{\kappa(n,s)}^n))=&e_s^n( b(x_s)-b(x_s^n))+e_s^n(b(x_s^n)-b(x_{\kappa(n,s)}^n))
\\
& +e_s^n(b(x_{\kappa(n,s)}^n)- \tilde{b}^n(x_{\kappa(n,s)}^n))
\end{align*}
which on using Assumption A-2 and Young's inequality gives
\begin{align} \label{eq:enbn:rate}
e_s^n(b(x_s)-\tilde{b}^n(x_{\kappa(n,s)}^n)) & \leq  K|e_s^n|^2+e_s^n(b(x_s^n)-b(x_{\kappa(n,s)}^n)) \notag
\\
& +K|b(x_{\kappa(n,s)}^n)- \tilde{b}^n(x_{\kappa(n,s)}^n)|^2
\end{align}
for any $s \in [0,T]$. Thus, $W_1$ can be estimated by
\begin{align}
W_1:=&2E\int_0^t e_s^n(b(x_s)-\tilde{b}^n(x_{\kappa(n,s)}^n))ds \notag
\\
\leq & KE\int_0^t |e_s^n|^2ds+2E\int_0^t e_s^n(b(x_s^n)-b(x_{\kappa(n,s)}^n))ds \notag
\\
& +KE\int_0^t |b(x_{\kappa(n,s)}^n)- \tilde{b}^n(x_{\kappa(n,s)}^n)|^2ds  \notag
\end{align}
which on the application of Lemma [\ref{lem:form:rate}, \ref{lem:a-tilde a:rate}] implies
\begin{align}
W_1 \leq K \int_0^t \sup_{0 \leq r \leq s}E|e_r^n|^2ds+ Kn^{-\frac{2}{2+\delta}-1}. \label{eq:W1}
\end{align}
Further, for estimating $W_2$, one writes,
$$
\sigma(x_s)-\tilde{\sigma}(x_{\kappa(n,s)}^n)=\sigma(x_s)-\sigma(x_s^n)+\sigma(x_s^n)-\tilde{\sigma}(x_{\kappa(n,s)}^n)
$$
which on using Assumption A-2 yields,
\begin{align} \label{eq:sigm:split:rate1}
|\sigma(x_s)-\tilde{\sigma}(x_{\kappa(n,s)}^n)|^2\leq K|e_s^n|^2+2|\sigma(x_s^n)-\tilde{\sigma}(x_{\kappa(n,s)}^n)|^2
\end{align}
for any $s \in [0,T]$. Thus, due to Lemma \ref{lem:b-tilde b:rate}, one obtains
\begin{align}
W_2 & := E \int_{0}^{t} |\sigma(x_s)-\tilde{\sigma}(x_{\kappa(n,s)}^n)|^2 ds \leq  K \int_{0}^{t} \sup_{0 \leq r \leq s} E|e_r^n|^2 ds + K n^{-2}. \label{eq:W2}
\end{align}
Finally, for estimating $W_3$, one uses equation \eqref{eq:y1y2} to write
\begin{align*}
W_3&:=E \int_{0}^{t} \int_Z \big(  |e_s^n+\gamma(x_s,z_2)-\tilde \gamma(x_{\kappa(n,s)}^n,z_2)|^2-|e_s^n|^2
\\
& \qquad -2e_s^n(\gamma(x_s,z_2)-\tilde \gamma(x_{\kappa(n,s)}^n,z_2))\big)\nu(dz_2)ds
\\
& \leq K E \int_{0}^{t} \int_Z  |\gamma(x_s,z_2)-\tilde \gamma(x_{\kappa(n,s)}^n,z_2)|^2 \nu(dz_2)ds
\end{align*}
and then applying the following splitting
\begin{align*}
\gamma(x_s,z_2) & -\tilde \gamma(x_{\kappa(n,s)}^n,z_2)=\gamma(x_s,z_2)- \gamma(x_s^n,z_2)
\\
& +\gamma(x_s^n,z_2)-\tilde \gamma(x_{\kappa(n,s)}^n,z_2),
\end{align*}
one could write
\begin{align*}
W_3& \leq K E \int_{0}^{t} \int_Z  |\gamma(x_s,z_2)-\gamma(x_s^n,z_2)|^2  \nu(dz_2)ds
\\
& + K E \int_{0}^{t} \int_Z |\gamma(x_s^n,z_2)-\tilde \gamma(x_{\kappa(n,s)}^n,z_2)|^2 \nu(dz_2)ds
\end{align*}
which due to Assumption A-2 and Lemma \ref{c-tilde c:rate} yields
\begin{align}
W_3& \leq  K \int_{0}^{t} \sup_{0 \leq r \leq s}E|e_r^n|^2 ds  +  K n^{-2}. \label{eq:W3}
\end{align}
Thus by substituting estimates from \eqref{eq:W1}, \eqref{eq:W2} and \eqref{eq:W3} in \eqref{eq:W1+W2+W3}, one obtains,
\begin{align}
\sup_{0 \leq t \leq u}E|e_t^n|^2 \leq  K \int_{0}^{u} \sup_{0 \leq r \leq s}E|e_r^n|^2 ds  +  K n^{-\frac{2}{2+\delta}-1} < \infty \notag
\end{align}
for any $u \in [0,T]$. The Gronwall's inequality completes the proof.
\end{proof}
%\begin{cor}
%Let Assumptions A-1 to
% A-5 hold, then
%$$
%E\sup_{0 \leq t \leq T}|x_t-x_t^n|^2 \leq K n^{-\frac{2}{2+\delta}-1}
%$$
%with $0< \delta < 1$.
%\end{cor}
%\begin{proof}
%This is an immediate consequence of Theorem \ref{thm:main:thm} and Lemma \ref{lem:yor}.
%\end{proof}
%---------------------------------------------------------------
\section{Continuous Case} \label{sec:con}
%---------------------------------------------------------------
In this section, the case when $\nu \equiv 0$ or $\gamma\equiv0$ in SDE \eqref{eq:sde} is discussed, i.e. one considers the following SDE,
\begin{align} \label{eq:sde:continuous}
x_t=& \xi + \int_{0}^t b(x_s)ds+\int_{0}^t \sigma(x_s)dw_s
\end{align}
almost surely for any $t \in [0,T] $, where $\xi$ is an $\mathscr{F}_{0}$-measurable  random variable in $\mathbb{R}^d$. Further, the scheme \eqref{eq:milstein} is replaced by the following,
\begin{align} \label{eq:milstein:continuous}
x_t^n=& \xi + \int_{0}^t \tilde b^n(x_{\kappa(n,s)}^n)ds+\int_{0}^t \tilde \sigma(x_{\kappa(n,s)}^n)dw_s,
\end{align}
almost surely for any $t \in [0,T]$. The drift coefficient $\tilde b^n$ in scheme \eqref{eq:milstein:continuous} is given by equation \eqref{eq:an:tilde} and the diffusion coefficient $\tilde \sigma$ is defined as
\begin{align} \label{eq:tilde:sigma}
\tilde \sigma(x_{\kappa(n,s)}^n):=\sigma(x_{\kappa(n,s)}^n)+\sigma_1(x_{\kappa(n,s)}^n)
\end{align}
where $\sigma_1$ is given by \eqref{eq:sigma:1}. Also, Assumptions A-2 to A-4 hold since  $\gamma\equiv0$. This implies that Lemma \ref{lem:sde:momentbound} holds under Assumptions A-1 and  A-2 with $b(x)$ being a continuous function in $x \in \mathbb{R}^d$. Further, Lemma \ref{lem:b1:rate} and Corollary \ref{lem:tildeb:momentbound} hold under Assumptions A-2 and A-3  while Lemma \ref{lem:scheme:momentbound} holds under Assumptions A-1 to A-3. One also notes that Lemma \ref{lem:b2:rate} to Lemma \ref{lem:c3:rate} and Corollary \ref{lem:tilde c:momentbound} are not required for this section since $\gamma\equiv0$.

We now proceed for the derivation of the rate of convergence of scheme \eqref{eq:milstein:continuous} and prove that this is same as that obtained by  \cite{wanga2013}. One can notice in the following calculations that a rate of convergence is achieved for any  $ q  \leq p/(3\chi+6)$ when either $p/(3\chi+6)\geq 4$ or $p/(3\chi+6)=2$.

\begin{lem} \label{lem:one-step:rate:new}
Consider Remark \ref{as:scheme:a^n} and  let Assumptions A-1 to A-3 (with $\gamma\equiv0$) hold, then
$$
\sup_{0 \leq t \leq T}E|x_t^n-x_{\kappa(n,t)}^n|^q \leq K n^{-\frac{q}{2}}.
$$
\end{lem}
\begin{proof}
This follows due to Lemmas [\ref{lem:one-step:norate}, \ref{lem:scheme:momentbound}].
\end{proof}

\begin{lem} \label{lem:form:rate:new}
Consider equation \eqref{eq:an:tilde} and let Assumptions A-1 to A-3 (with $\gamma\equiv0$) be satisfied, then
$$
\sup_{0 \leq t \leq T}E |b(x_t^n)-\tilde{b}^n(x_t^n)|^q dt  \leq K n^{-q}.
$$
\end{lem}
\begin{proof}
This is an immediate consequence of equation \eqref{eq:an:tilde}, Remark \ref{as:sde:growth} and Lemma \ref{lem:scheme:momentbound}.
\end{proof}

\begin{lem} \label{lem:b-tilde b:rate:new}
Let Assumptions A-1 to A-3 (with $\gamma\equiv0$) hold, then
$$
\sup_{0 \leq t \leq T} E|\sigma(x_{t}^n) - \tilde{\sigma}(x_{\kappa(n,t)}^n)|^q \leq K n^{-q}.
$$
\end{lem}
\begin{proof}
One uses It\^o's formula to estimate the difference $\sigma^{(k,v)}(x_t^n) - \sigma^{(k,v)}(x_{\kappa(n,t)}^n)$ and \eqref{eq:tilde:sigma} to obtain the following,
\begin{align*}
&E|\sigma(x_t^n) - \tilde{\sigma}(x_{\kappa(n,t)}^n)|^q=E \Big(\sum_{k=1}^{d}\sum_{v=1}^{m}|\sigma^{(k,v)}(x_t^n) - \tilde{\sigma}^{(k,v)}(x_{\kappa(n,t)}^n)|^2\Big)^\frac{q}{2} \notag
\\
&\leq  K  \sum_{u,k=1}^d \sum_{v=1}^{m}E\Big|\int_{\kappa(n,t)}^t \frac{\partial \sigma^{(k,v)}(x_r^n)}{\partial x^u} \tilde{b}^{n,u}(x_{\kappa(n,r)}^n)dr \Big|^q \notag
\\
&+K \sum_{u,k=1}^d \sum_{v=1}^{m} \sum_{u_1,u_2=1}^{d} \sum_{l_1=1}^{m} E\Big|\int_{\kappa(n,t)}^t \frac{\partial^2 \sigma^{(k,v)}(x_r^n)}{\partial x^{u_1} \partial x^{u_2}}  \tilde{\sigma}^{(u_1,l_1)}(x_{\kappa(n,r)}^n)\tilde{\sigma}^{(u_2,l_1)}(x_{\kappa(n,r)}^n)dr \Big|^q \notag
\end{align*}
\begin{align}
&+K  \sum_{u,k=1}^d \sum_{v=1}^{m} \sum_{u=1}^d \sum_{j=1}^m E\Big|\int_{\kappa(n,t)}^t  \sigma^{(u,j)}(x_{\kappa(n,r)}^n)\Big\{\frac{\partial \sigma^{(k,v)}(x_r^n)}{\partial x^u} -\frac{\partial \sigma^{(k,v)} (x_{\kappa(n,r)}^n)}{\partial x^u}\Big\}dw_r^j\Big|^q \notag
\\
& +K  \sum_{u,k=1}^d \sum_{v=1}^{m}\sum_{u=1}^d \sum_{j=1}^m E\Big|\int_{\kappa(n,t)}^t \frac{\partial \sigma^{(k,v)}(x_r^n)}{\partial x^u}  \sigma_1^{(u,j)}(x_{\kappa(n,r)}^n)dw_r^j\Big|^q \notag
\end{align}
and the application of Remarks [\ref{as:sde:growth} ,\ref{rem:der:bounded}, \ref{as:scheme:a^n}, \ref{as:sde:poly}],  Assumption A-3, Lemmas [\ref{lem:b1:rate}, \ref{lem:one-step:rate:new}] and Corollary \ref{lem:tildeb:momentbound}   completes the proof.
\end{proof}

\begin{lem} \label{lem:a-tilde a:rate:new}
Let Assumptions A-1 to
%, A-4, A-2, A-5,
A-3 (with $\gamma\equiv0$)%,A- \ref{as:a:continuity}
 and A-5 hold, then
\begin{equation*}
E\int_{0}^{t} |e_s^n|^{q-2}e_s^n(b(x_s^n)-b(x_{\kappa(n,s)}^n))ds\leq K  \int_{0}^t \sup_{0 \leq r \leq s}E|e_r^n|^q   ds + K n^{-q}
\end{equation*}
for any $t \in [0,T]$.
\end{lem}
\begin{proof}
The proof of the lemma is done for any $q \geq 4$ since the results for the case $q=2$ is a direct consequence of Lemma \ref{lem:a-tilde a:rate} when $\gamma\equiv0$.
As before, one uses Ito's formula to write $b^k(x_s^n)-b^k(x_{\kappa(n,s)}^n)$ and then obtains the following,
\begin{align}
&E\int_{0}^t |e_s^n|^{q-2} e_s^n(b(x_s^n)-b(x_{\kappa(n,s)}^n))ds  = \sum_{k=1}^d E\int_{0}^t |e_s^n|^{q-2} e_s^{n,k}(b^k(x_s^n)-b^k(x_{\kappa(n,s)}^n))ds \notag
\\
&=\sum_{k,i=1}^d  E \int_{0}^t |e_s^n|^{q-2} e_s^{n,k} \int_{\kappa(n,s)}^s \frac{\partial b^k(x_r^n)}{\partial x^i} \tilde{b}^{n,i}(x_{\kappa(n,r)}^n)dr ds \notag
\\
 &+\frac{1}{2}\sum_{k,i,j=1}^{d}E\int_{0}^t |e_s^n|^{q-2} e_s^{n,k}  \int_{\kappa(n,s)}^s \frac{\partial^2 b^k(x_r^n)}{\partial x^i \partial x^j} \sum_{l=1}^{m} \tilde{\sigma}^{(i,l)}(x_{\kappa(n,r)}^n)\tilde{\sigma}^{(j,l)}(x_{\kappa(n,r)}^n)dr ds \notag
 \\
 & +\sum_{k,i=1}^d E\int_{0}^t |e_s^n|^{q-2} e_s^{n,k}\int_{\kappa(n,s)}^s \frac{\partial b^k(x_r^n)}{\partial x^i} \sum_{l=1}^m \tilde{\sigma}^{(i,l)}(x_{\kappa(n,r)}^n)dw_r^l ds \notag
 \\
 & =:T_1+T_2+T_3. \label{eq:T1+T3}
\end{align}
By using Schwarz and Young's   inequalities, $T_1$ can be estimated as
\begin{align*}
T_1:= & \sum_{k,i=1}^d  E \int_{0}^t |e_s^n|^{q-2}e_s^{n,k} \int_{\kappa(n,s)}^s \frac{\partial b^k(x_r^n)}{\partial x^i} \tilde{b}^{n,i}(x_{\kappa(n,r)}^n)dr ds
\\
\leq & K  E \int_{0}^t  |e_s^{n}|^q ds +K \sum_{k,i=1}^d  E \int_{0}^t  \Big|\int_{\kappa(n,s)}^s \frac{\partial b^k(x_r^n)}{\partial x^i} \tilde{b}^{n,i}(x_{\kappa(n,r)}^n)dr \Big|^q ds
\end{align*}
which on the application of  Remarks [\ref{as:scheme:a^n},  \ref{as:sde:poly:Der}, \ref{as:sde:poly}] and Lemma \ref{lem:scheme:momentbound} implies
\begin{align}
T_1 \leq K   \int_{0}^t  \sup_{0 \leq r \leq s}E |e_r^{n}|^q ds +K n^{-q}. \label{eq:T1}
\end{align}
By using Schwarz, Young's and H\"older's inequalities, Remark \ref{as:sde:poly:Der}, Corollary \ref{lem:tildeb:momentbound} and Lemma \ref{lem:scheme:momentbound}, $T_2$ can be estimated as,
\begin{align}
T_2&:=\frac{1}{2}\sum_{k,i,j=1}^{d}E\int_{0}^t |e_s^n|^{q-2} e_s^{n,k}  \int_{\kappa(n,s)}^s \frac{\partial^2 b^k(x_r^n)}{\partial x^i \partial x^j} \sum_{l=1}^{m} \tilde{\sigma}^{(i,l)}(x_{\kappa(n,r)}^n)\tilde{\sigma}^{(j,l)}(x_{\kappa(n,r)}^n)dr ds \notag
\\
& \leq K   \int_{0}^t  \sup_{0 \leq r \leq s}E |e_r^{n}|^q ds +K n^{-q}. \label{eq:T2}
\end{align}
For the last term of equation \eqref{eq:T1+T3}, one uses It\^o's formula to obtain the following,
\begin{align} \label{eq:ito:new1}
|e_s^n|^{q-2} & e_s^{n,k}=|e_{\kappa(n,s)}^n|^{q-2} e_{\kappa(n,s)}^{n,k} + \int_{\kappa(n,s)}^s |e_r^n|^{q-2}(b^k(x_r)-\tilde{b}^{n,k}(x_{\kappa(n,r)}^n) ) dr \notag
\\
&+  \int_{\kappa(n,s)}^s |e_r^n|^{q-2} \sum_{v=1}^m\big(\sigma^{(k,v)}(x_r)-\tilde{\sigma}^{(k,v)}(x_{\kappa(n,r)}^n)\big)dw_r^v \notag
\\
& + (q-2)  \int_{\kappa(n,s)}^s e_r^{n,k} |e_r^n|^{q-4} e_r^n (b(x_r)-\tilde{b}^n(x_{\kappa(n,r)}^n)) dr \notag
\\
&+ (q-2) \int_{\kappa(n,s)}^s e_r^{n,k} |e_r^n|^{q-4} e_r^n (\sigma(x_r)-\tilde{\sigma}(x_{\kappa(n,r)}^n))  dw_r  \notag
\\
& + \frac{(q-2)(q-4)}{2}  \int_{\kappa(n,s)}^s e_r^{n,k} |e_r^n|^{q-6}| (\sigma(x_r)-\tilde{\sigma}(x_{\kappa(n,r)}^n))^{*}e_r^n|^2 dr  \notag
\\
& + \frac{q-2}{2}  \int_{\kappa(n,s)}^s e_r^{n,k}|e_r^n|^{q-4}|\sigma(x_r)-\tilde{\sigma}(x_{\kappa(n,r)}^n)|^2 dr \notag
\\
& + (q-2) \int_{\kappa(n,s)}^s \sum_{v=1}^m (\sigma^{(k,v)}(x_r)-\tilde{\sigma}^{(k,v)}(x_{\kappa(n,r)}^n))|e_r^n|^{q-4} \notag
\\
& \qquad \times\sum_{u=1}^d e_r^{n,u} (\sigma^{(u,v)}(x_r)-\tilde{\sigma}^{(u,v)}(x_{\kappa(n,r)}^n)) dr
\end{align}
almost surely for any $s \in [0,T]$. Hence, on substituting the values from equation \eqref{eq:ito:new1} in $T_3$ of equation  \eqref{eq:T1+T3}
and then on using Schwarz inequality, one obtains
\begin{align*}
T_3 &= \sum_{k,i=1}^d E\int_{0}^t |e_{\kappa(n,s)}^n|^{q-2} e_{\kappa(n,s)}^{n,k} \int_{\kappa(n,s)}^s \frac{\partial b^k(x_r^n)}{\partial x^i} \sum_{l=1}^m \tilde{\sigma}^{(i,l)}(x_{\kappa(n,r)}^n)dw_r^l ds \notag
\\
&+\sum_{k,i=1}^d E\int_{0}^t \int_{\kappa(n,s)}^s |e_r^n|^{q-2}|b(x_r)-\tilde{b}^{n}(x_{\kappa(n,r)}^n) | dr \notag
\\
& \qquad\times \Big|\int_{\kappa(n,s)}^s \frac{\partial b^k(x_r^n)}{\partial x^i} \sum_{l=1}^m \tilde{\sigma}^{(i,l)}(x_{\kappa(n,r)}^n)dw_r^l \Big|ds \notag
\\
& + \sum_{k,i=1}^d E\int_{0}^t \int_{\kappa(n,s)}^s |e_r^n|^{q-2} \sum_{v=1}^m\big(\sigma^{(k,v)}(x_r)-\tilde{\sigma}^{(k,v)}(x_{\kappa(n,r)}^n)\big)dw_r^v \notag
\\
& \qquad\times \int_{\kappa(n,s)}^s \frac{\partial b^k(x_r^n)}{\partial x^i} \sum_{l=1}^m \tilde{\sigma}^{(i,l)}(x_{\kappa(n,r)}^n)dw_r^l ds \notag
%\\
%& + K \sum_{k,i=1}^d E\int_{0}^t   \int_{\kappa(n,s)}^s e_r^{n,k} |e_r^n|^{q-4} e_r^n (b(x_r)-\tilde{b}^n(x_{\kappa(n,r)}^n)) dr \notag\times\int_{\kappa(n,s)}^s \frac{\partial b^k(x_r^n)}{\partial x^i} \sum_{l=1}^m \tilde{\sigma}^{(i,l)}(x_{\kappa(n,r)}^n)dw_r^l ds \notag
\\
& + K \sum_{k,i=1}^d E\int_{0}^t  \int_{\kappa(n,s)}^s e_r^{n,k} |e_r^n|^{q-4} e_r^n (\sigma(x_r)-\tilde{\sigma}(x_{\kappa(n,r)}^n))  dw_r \notag
\\
& \qquad \times \int_{\kappa(n,s)}^s \frac{\partial b^k(x_r^n)}{\partial x^i} \sum_{l=1}^m \tilde{\sigma}^{(i,l)}(x_{\kappa(n,r)}^n)dw_r^l ds \notag
\end{align*}
\begin{align} \label{eq:T31+T35}
& + K \sum_{k,i=1}^d E\int_{0}^t   \int_{\kappa(n,s)}^s |e_r^n|^{q-3}|\sigma(x_r)-\tilde{\sigma}(x_{\kappa(n,r)}^n)|^2 dr \Big| \notag
\\
& \qquad \times\int_{\kappa(n,s)}^s \frac{\partial b^k(x_r^n)}{\partial x^i} \sum_{l=1}^m \tilde{\sigma}^{(i,l)}(x_{\kappa(n,r)}^n)dw_r^l \Big|ds \notag
\\
&=:T_{31}+T_{32}+T_{33}+T_{34}+T_{35}.
\end{align}
Furthermore, for estimating $T_{31}$, one writes,
\begin{align*}
T_{31}&:=K\sum_{k,i=1}^d  \sum_{l=1}^{m} E \int_{0}^t |e_{\kappa(n,s)}^n|^{q-2} e_{\kappa(n,s)}^{n,k}\int_{\kappa(n,s)}^s\frac{\partial b^k(x_r^n)}{\partial x^i}  \tilde{\sigma}^{(i,l)}(x_{\kappa(n,r)}^n)dw_r^l ds\notag
\\
& =K\sum_{k,i=1}^d  \sum_{l=1}^{m} E \int_{0}^t |e_{\kappa(n,s)}^n|^{q-2} e_{\kappa(n,s)}^{n,k}\int_{\kappa(n,s)}^s\frac{\partial b^k(x_r^n)}{\partial x^i}  \sigma^{(i,l)}(x_{\kappa(n,r)}^n)dw_r^l ds\notag
\\
& +K\sum_{k,i=1}^d  \sum_{l=1}^{m} E \int_{0}^t |e_{\kappa(n,s)}^n|^{q-2} e_{\kappa(n,s)}^{n,k}\int_{\kappa(n,s)}^s\frac{\partial b^k(x_r^n)}{\partial x^i}  \sigma_1^{(i,l)}(x_{\kappa(n,r)}^n)dw_r^l ds\notag
\end{align*}
which on using the Young's inequality, H\"older's inequality and an elementary inequality of stochastic integrals as well as by noticing that the first term is zero, one obtains the following,
\begin{align*}
T_{31} & \leq K  \int_{0}^t \sup_{0 \leq r \leq s}E|e_r^n|^q ds
\\
& \qquad+K\sum_{k,i=1}^d  \sum_{l=1}^{m} E \int_{0}^t \Big|\int_{\kappa(n,s)}^s\frac{\partial b^k(x_r^n)}{\partial x^i}  \sigma_1^{(i,l)}(x_{\kappa(n,r)}^n)dw_r^l\Big|^q ds\notag
\\
& \leq K  \int_{0}^t \sup_{0 \leq r \leq s}E|e_r^n|^q ds
\\
& \qquad+Kn^{-\frac{q}{2}+1} E \int_{0}^t \int_{\kappa(n,s)}^s(1+|x_r^n|^\chi)^q  |\sigma_1(x_{\kappa(n,r)}^n)|^q dr ds\notag
\\
& \leq K  \int_{0}^t \sup_{0 \leq r \leq s}E|e_r^n|^q ds
\\
& \qquad +Kn^{-\frac{q}{2}+1} E \int_{0}^t \int_{\kappa(n,s)}^s (E(1+|x_r^n|^\chi)^{\frac{qp}{p-q}})^\frac{p-q}{p}  (E|\sigma_1(x_{\kappa(n,r)}^n)|^p)^\frac{q}{p} dr ds\notag
\end{align*}
and then the application of Lemmas [\ref{lem:b1:rate}, \ref{lem:scheme:momentbound}] yields,
\begin{align} \label{eq:T31}
T_{31} \leq K  \int_{0}^t \sup_{0 \leq r \leq s}E|e_r^n|^q ds +Kn^{-q}.
\end{align}
Also, to estimate $T_{32}$, one uses the splitting \eqref{eq:b:split:new} to obtain,
\begin{align*}
T_{32}& :=\sum_{k,i=1}^d E\int_{0}^t \int_{\kappa(n,s)}^s |e_r^n|^{q-2}|b(x_r)-\tilde{b}^{n}(x_{\kappa(n,r)}^n) | dr
\\
& \qquad \times \Big|\int_{\kappa(n,s)}^s \frac{\partial b^k(x_r^n)}{\partial x^i} \sum_{l=1}^m \tilde{\sigma}^{(i,l)}(x_{\kappa(n,r)}^n)dw_r^l \Big|ds
\\
& \leq K \sum_{k,i=1}^d E\int_{0}^t \int_{\kappa(n,s)}^s \big(n^\frac{1}{q}|e_r^n|\big)^{q-1}(1+|x_r|^{\chi}+|x_r^n|^{\chi}) dr
\\
& \qquad \times n^{-\frac{q-1}{q}} \Big|\int_{\kappa(n,s)}^s \frac{\partial b^k(x_r^n)}{\partial x^i} \sum_{l=1}^m \tilde{\sigma}^{(i,l)}(x_{\kappa(n,r)}^n)dw_r^l \Big|ds
\\
& + K \sum_{k,i=1}^d E\int_{0}^t \int_{\kappa(n,s)}^s \big(n^\frac{1}{q}|e_r^n|\big)^{q-2}(1+|x_r^n|^\chi+|x_{\kappa(n,r)}^n|^\chi)|x_r^n-x_{\kappa(n,r)}^n| dr
\\
 & \qquad\times n^{-\frac{q-2}{q}} \Big|\int_{\kappa(n,s)}^s \frac{\partial b^k(x_r^n)}{\partial x^i} \sum_{l=1}^m \tilde{\sigma}^{(i,l)}(x_{\kappa(n,r)}^n)dw_r^l \Big|ds
 \\
 & + K \sum_{k,i=1}^d E\int_{0}^t \int_{\kappa(n,s)}^s \big(n^\frac{1}{q}|e_r^n|\big)^{q-2}|b(x_{\kappa(n,r)}^n)-\tilde{b}^{n}(x_{\kappa(n,r)}^n) | dr
 \\
 & \qquad\times n^{-\frac{q-2}{q}} \Big|\int_{\kappa(n,s)}^s \frac{\partial b^k(x_r^n)}{\partial x^i} \sum_{l=1}^m \tilde{\sigma}^{(i,l)}(x_{\kappa(n,r)}^n)dw_r^l \Big|ds.
\end{align*}
Due to Young's inequality and Remark \ref{as:sde:poly},
\begin{align*}
T_{32} & \leq K E \int_{0}^t n \int_{\kappa(n,s)}^s |e_r^n|^{q} dr ds +  K n^{-q+1} \sum_{k,i=1}^d E \int_{0}^t  \int_{\kappa(n,s)}^s (1+|x_r|^{\chi}+|x_r^n|^{\chi})^q dr
\\
& \hspace{3cm} \times \Big|\int_{\kappa(n,s)}^s \frac{\partial b^k(x_r^n)}{\partial x^i} \sum_{l=1}^m \tilde{\sigma}^{(i,l)}(x_{\kappa(n,r)}^n)dw_r^l \Big|^q ds
\\
&  + K n^{-\frac{q-2}{2}} \sum_{k,i=1}^d E \int_{0}^t  \int_{\kappa(n,s)}^s (1+|x_r^n|^{\chi}+|x_{\kappa(n,r)}^n|^{\chi})^\frac{q}{2} |x_r^n-x_{\kappa(n,r)}^n|^\frac{q}{2} dr
\\
& \hspace{3cm} \times \Big|\int_{\kappa(n,s)}^s \frac{\partial b^k(x_r^n)}{\partial x^i} \sum_{l=1}^m \tilde{\sigma}^{(i,l)}(x_{\kappa(n,r)}^n)dw_r^l \Big|^\frac{q}{2} ds
\\
 & + K n^{-q+1} \sum_{k,i=1}^d E \int_{0}^t \int_{\kappa(n,s)}^s (1+|x_{\kappa(n,r)}^n|^{\chi+2})^q dr
\\
& \hspace{3cm} \times \Big|\int_{\kappa(n,s)}^s \frac{\partial b^k(x_r^n)}{\partial x^i} \sum_{l=1}^m \tilde{\sigma}^{(i,l)}(x_{\kappa(n,r)}^n)dw_r^l \Big|^\frac{q}{2} ds.
\end{align*}
which on using H\"older's inequality yields,
\begin{align*}
T_{32} & \leq K   \int_{0}^t \sup_{0 \leq r \leq s} E|e_r^n|^{q}  ds +  K n^{-q+1} \sum_{k,i=1}^d  \int_{0}^t  \Big[ E \Big(\int_{\kappa(n,s)}^s (1+|x_r|^{\chi}+|x_r^n|^{\chi})^q  dr \Big)^2\Big]^\frac{1}{2}
\\
& \hspace{3cm} \times \Big[E\Big|\int_{\kappa(n,s)}^s \frac{\partial b^k(x_r^n)}{\partial x^i} \sum_{l=1}^m \tilde{\sigma}^{(i,l)}(x_{\kappa(n,r)}^n)dw_r^l \Big|^{2q} \Big]^\frac{1}{2} ds
\\
&  + K n^{-\frac{q-2}{2}} \sum_{k,i=1}^d  \int_{0}^t  \Big[E \Big(\int_{\kappa(n,s)}^s (1+|x_r^n|^{\chi}+|x_{\kappa(n,r)}^n|^{\chi})^\frac{q}{2} |x_r^n-x_{\kappa(n,r)}^n|^\frac{q}{2} dr \Big)^2\Big]^\frac{1}{2}
\\
& \hspace{3cm} \times \Big[E\Big|\int_{\kappa(n,s)}^s \frac{\partial b^k(x_r^n)}{\partial x^i} \sum_{l=1}^m \tilde{\sigma}^{(i,l)}(x_{\kappa(n,r)}^n)dw_r^l \Big|^q\Big]^\frac{1}{2} ds
 \\
 & + K n^{-q+1} \sum_{k,i=1}^d  \int_{0}^t \Big[E \Big(\int_{\kappa(n,s)}^s (1+|x_{\kappa(n,r)}^n|^{\chi+2})^q dr \Big)^2\Big]^\frac{1}{2}
\\
& \hspace{3cm} \times \Big[E\Big|\int_{\kappa(n,s)}^s \frac{\partial b^k(x_r^n)}{\partial x^i} \sum_{l=1}^m \tilde{\sigma}^{(i,l)}(x_{\kappa(n,r)}^n)dw_r^l \Big|^q\Big]^\frac{1}{2} ds
\end{align*}
and then due to Remark \ref{as:sde:poly:Der}, one obtains
\begin{align*}
T_{32} & \leq K   \int_{0}^t \sup_{0 \leq r \leq s} E|e_r^n|^{q}  ds +  K n^{-q+1} \sum_{k,i=1}^d  \int_{0}^t  \Big[ n^{-1}E \int_{\kappa(n,s)}^s (1+|x_r|^{\chi}+|x_r^n|^{\chi})^{2q}  dr \Big]^\frac{1}{2}
\\
& \hspace{3cm} \times \Big[n^{-q+1}E\int_{\kappa(n,s)}^s (1+|x_r^n|^{\chi+1})^{2q}| \tilde{\sigma}(x_{\kappa(n,r)}^n)|^{2q} dr \Big]^\frac{1}{2} ds
\\
&  + K n^{-\frac{q-2}{2}}   \int_{0}^t  \Big[n^{-1}E \int_{\kappa(n,s)}^s (1+|x_r^n|^{\chi}+|x_{\kappa(n,r)}^n|^{\chi})^q |x_r^n-x_{\kappa(n,r)}^n|^q dr \Big]^\frac{1}{2}
\\
& \hspace{3cm} \times \Big[ n^{-\frac{q}{2}+1} E\int_{\kappa(n,s)}^s (1+|x_r^n|^{\chi+1})^q |\tilde{\sigma}(x_{\kappa(n,r)}^n)|^q dr \Big]^\frac{1}{2} ds
 \\
 & + K n^{-q+1}   \int_{0}^t \Big[n^{-1}E\int_{\kappa(n,s)}^s (1+|x_{\kappa(n,r)}^n|^{\chi+2})^{2q} dr \Big]^\frac{1}{2}
\\
& \hspace{3cm} \times \Big[n^{-\frac{q}{2}+1}E\int_{\kappa(n,s)}^s (1+|x_r^n|^{\chi+1})^q |\tilde{\sigma}(x_{\kappa(n,r)}^n)|^q dr \Big]^\frac{1}{2} ds.
\end{align*}
Thus, by the application of H\"older's inequality and Lemmas [\ref{lem:sde:momentbound}, \ref{lem:scheme:momentbound}, \ref{lem:one-step:rate:new}] and Corollary \ref{lem:tildeb:momentbound}, one obtains,
\begin{align}
T_{32} & \leq K \int_{0}^t  \sup_{0 \leq r \leq s}E|e_r^n|^q  ds+ K n^{-q}.\label{eq:T32}
\end{align}
For estimating $T_{33}$, one uses the splitting \eqref{eq:sigma:split:new} to write
\begin{align*}
T_{33} & :=\sum_{k,i=1}^d E\int_{0}^t \int_{\kappa(n,s)}^s |e_r^n|^{q-2} \sum_{v=1}^m\big(\sigma^{(k,v)}(x_r)-\tilde{\sigma}^{(k,v)}(x_{\kappa(n,r)}^n)\big)dw_r^v
\\
& \qquad \times\int_{\kappa(n,s)}^s \frac{\partial b^k(x_r^n)}{\partial x^i} \sum_{l=1}^m \tilde{\sigma}^{(i,l)}(x_{\kappa(n,r)}^n)dw_r^l ds
\end{align*}

\begin{align*}
& =\sum_{k,i=1}^d \sum_{v,l=1}^mE\int_{0}^t \int_{\kappa(n,s)}^s |e_r^n|^{q-2} \big(\sigma^{(k,v)}(x_r)-\sigma^{(k,v)}(x_r^n)\big)dw_r^v
\\
& \qquad \times \int_{\kappa(n,s)}^s \frac{\partial b^k(x_r^n)}{\partial x^i} \tilde{\sigma}^{(i,l)}(x_{\kappa(n,r)}^n)dw_r^l ds
\\
& + \sum_{k,i=1}^d \sum_{v,l=1}^m E\int_{0}^t \int_{\kappa(n,s)}^s |e_r^n|^{q-2} \big(\sigma^{(k,v)}(x_r^n)-\tilde{\sigma}^{(k,v)}(x_{\kappa(n,r)}^n)\big)dw_r^v
\\
& \qquad \times \int_{\kappa(n,s)}^s \frac{\partial b^k(x_r^n)}{\partial x^i} \tilde{\sigma}^{(i,l)}(x_{\kappa(n,r)}^n)dw_r^l ds
\\
& =\sum_{k,i=1}^d \sum_{v=1}^m E\int_{0}^t \int_{\kappa(n,s)}^s |e_r^n|^{q-2} \big(\sigma^{(k,v)}(x_r)-\sigma^{(k,v)}(x_r^n)\big)
\\
&\qquad \times \frac{\partial b^k(x_r^n)}{\partial x^i} \tilde{\sigma}^{(i,v)}(x_{\kappa(n,r)}^n)dr ds
\\
& + \sum_{k,i=1}^d \sum_{v,l=1}^m E\int_{0}^t \int_{\kappa(n,s)}^s |e_r^n|^{q-2} \big(\sigma^{(k,v)}(x_r^n)-\tilde{\sigma}^{(k,v)}(x_{\kappa(n,r)}^n)\big)
\\
& \qquad \times \frac{\partial b^k(x_r^n)}{\partial x^i} \tilde{\sigma}^{(i,v)}(x_{\kappa(n,r)}^n)dr ds
\end{align*}
and then  Assumption A-2 and Remark \ref{as:sde:poly} give
\begin{align} \label{eq:T33:older}
T_{33} &\leq K  E\int_{0}^t \int_{\kappa(n,s)}^s \big(n^\frac{1}{q}|e_r^n|\big)^{q-1} n^{-\frac{q-1}{q}}(1+|x_r^n|^{\chi+1}) |\tilde{\sigma}(x_{\kappa(n,r)}^n)|dr ds \notag
\\
& + K E\int_{0}^t \int_{\kappa(n,s)}^s \big(n^\frac{1}{q} |e_r^n|\big)^{q-2} n^{-\frac{q-2}{q}}|\sigma(x_r^n)-\tilde{\sigma} \notag
\\
& \qquad \times (x_{\kappa(n,r)}^n)|(1+|x_r^n|^{\chi+1})|\tilde{\sigma}(x_{\kappa(n,r)}^n)| dr ds
\end{align}
%which due to H\"older's inequality becomes
%\begin{align*}
%T_{33} & \leq K  \int_{0}^t \int_{\kappa(n,s)}^s   (E |e_r^n|^q)^\frac{q-1}{q}  \big(E(1+|x_r^n|^{\chi+1})^q |\tilde{\sigma}(x_{\kappa(n,r)}^n)|^q\big)^\frac{1}{q} dr ds
%\\
%& + K \int_{0}^t \int_{\kappa(n,s)}^s  (E|e_r^n|^q)^\frac{q-2}{q} (E |\sigma(x_r^n)-\tilde{\sigma}(x_{\kappa(n,r)}^n)|^\frac{q}{2}(1+|x_r^n|^{\chi+1})^\frac{q}{2}|\tilde{\sigma}(x_{\kappa(n,r)}^n)|^\frac{q}{2} )^\frac{2}{q} dr ds
%\\
%& \leq K  \int_{0}^t (\sup_{0 \leq r \leq s}E |e_r^n|^q)^\frac{q-1}{q} \int_{\kappa(n,s)}^s     \big(E(1+|x_r^n|^{\chi+1})^q |\tilde{\sigma}(x_{\kappa(n,r)}^n)|^q\big)^\frac{1}{q} dr ds
%\\
%& + K \int_{0}^t (\sup_{0 \leq r \leq s}E|e_r^n|^q)^\frac{q-2}{q} \int_{\kappa(n,s)}^s   (E |\sigma(x_r^n)-\tilde{\sigma}(x_{\kappa(n,r)}^n)|^\frac{q}{2}(1+|x_r^n|^{\chi+1})^\frac{q}{2}|\tilde{\sigma}(x_{\kappa(n,r)}^n)|^\frac{q}{2} )^\frac{2}{q} dr ds.
%\end{align*}
%Also,
which due to Young's inequality and H\"older's inequality yields
\begin{align*}
& T_{33} %& \leq K  \int_{0}^t  E |e_r^n|^q ds + K  n^{-q+1} \int_{0}^t \int_{\kappa(n,s)}^s  E(1+|x_r^n|^{\chi+1})^q |\tilde{\sigma}(x_{\kappa(n,r)}^n)|^q dr ds
%\\
%& + K n^{-\frac{q}{2}+1} \int_{0}^t \int_{\kappa(n,s)}^s E |\sigma(x_r^n)-\tilde{\sigma}(x_{\kappa(n,r)}^n)|^\frac{q}{2}(1+|x_r^n|^{\chi+1})^\frac{q}{2}|\tilde{\sigma}(x_{\kappa(n,r)}^n)|^\frac{q}{2}  dr ds
 \leq K  \int_{0}^t \sup_{0 \leq r \leq s} E |e_r^n|^q ds
\\
& + K n^{-q+1} \int_{0}^t \int_{\kappa(n,s)}^s   (E(1+|x_r^n|^{\chi+1})^{2q} E|\tilde{\sigma}(x_{\kappa(n,r)}^n)|^{2q})^\frac{1}{2} dr ds
\\
&  + K n^{-\frac{q}{2}+1} \int_{0}^t \int_{\kappa(n,s)}^s   (E|\sigma(x_r^n)-\tilde{\sigma}(x_{\kappa(n,r)}^n)|^qE(1+|x_r^n|^{\chi+1})^q|\tilde{\sigma}(x_{\kappa(n,r)}^n)|^q)^\frac{1}{2}   dr ds.
\end{align*}
and on the application of Lemmas [\ref{lem:scheme:momentbound}, \ref{lem:b-tilde b:rate:new}] and Corollary \ref{lem:tildeb:momentbound}, one obtains
\begin{align} \label{eq:T33}
T_{33} & \leq K  \int_{0}^t \sup_{0 \leq r \leq s} E |e_r^n|^q ds + K  n^{-q}.
\end{align}
To estimate $T_{34}$, one observes that
\begin{align*}
T_{34} & :=K \sum_{k,i=1}^d E\int_{0}^t  \int_{\kappa(n,s)}^s e_r^{n,k} |e_r^n|^{q-4} e_r^n (\sigma(x_r)-\tilde{\sigma}(x_{\kappa(n,r)}^n))  dw_r
\\
& \qquad \times \int_{\kappa(n,s)}^s \frac{\partial b^k(x_r^n)}{\partial x^i} \sum_{l=1}^m \tilde{\sigma}^{(i,l)}(x_{\kappa(n,r)}^n)dw_r^l ds \notag
%\\
%& =K \sum_{k,i=1}^d E\int_{0}^t  \int_{\kappa(n,s)}^s e_r^{n,k} |e_r^n|^{q-4} \sum_{u=1}^d \sum_{v=1}^m e_r^{n,u} (\sigma^{(u,v)}(x_r)-\tilde{\sigma}^{(u,v)}(x_{\kappa(n,r)}^n))  dw_r^v
%\\
%& \hspace{3cm} \times \int_{\kappa(n,s)}^s \frac{\partial b^k(x_r^n)}{\partial x^i} \sum_{l=1}^m \tilde{\sigma}^{(i,l)}(x_{\kappa(n,r)}^n)dw_r^l ds \notag
\\
& =K \sum_{k,i,u=1}^d  \sum_{v,l=1}^m  E\int_{0}^t  \int_{\kappa(n,s)}^s e_r^{n,k} |e_r^n|^{q-4} e_r^{n,u} (\sigma^{(u,v)}(x_r)-\tilde{\sigma}^{(u,v)}(x_{\kappa(n,r)}^n))  dw_r^v
\\
& \hspace{1cm} \times \int_{\kappa(n,s)}^s \frac{\partial b^k(x_r^n)}{\partial x^i}  \tilde{\sigma}^{(i,l)}(x_{\kappa(n,r)}^n)dw_r^l ds \notag
\\
& =K \sum_{k,i,u=1}^d  \sum_{v=1}^m  E\int_{0}^t  \int_{\kappa(n,s)}^s e_r^{n,k} |e_r^n|^{q-4} e_r^{n,u} (\sigma^{(u,v)}(x_r)-\tilde{\sigma}^{(u,v)}(x_{\kappa(n,r)}^n))
\\
& \qquad \times \frac{\partial b^k(x_r^n)}{\partial x^i}  \tilde{\sigma}^{(i,v)}(x_{\kappa(n,r)}^n)dr ds \notag
\end{align*}
and then due to splitting \eqref{eq:sigma:split:new} along with Remark \ref{as:sde:poly:Der}, one obtains,
\begin{align*}
T_{34} &\leq K  E\int_{0}^t \int_{\kappa(n,s)}^s |e_r^n|^{q-1} (1+|x_r^n|^{\chi+1}) |\tilde{\sigma}(x_{\kappa(n,r)}^n)|dr ds
\\
& + K E\int_{0}^t \int_{\kappa(n,s)}^s |e_r^n|^{q-2} |\sigma(x_r^n)-\tilde{\sigma}(x_{\kappa(n,r)}^n)|(1+|x_r^n|^{\chi+1})|\tilde{\sigma}(x_{\kappa(n,r)}^n)| dr ds
\end{align*}
which is exactly the same as $T_{33}$ given in equation \eqref{eq:T33:older}. Thus, by proceeding exactly the same, one obtains,
\begin{align} \label{eq:T34}
T_{34} & \leq K  \int_{0}^t \sup_{0 \leq r \leq s} E |e_r^n|^q ds + K  n^{-q}.
\end{align}
Moreover, for estimating $T_{35}$, one again uses the splitting \eqref{eq:sigma:split:new} to obtain,
\begin{align*}
T_{35}& :=K \sum_{k,i=1}^d E\int_{0}^t   \int_{\kappa(n,s)}^s |e_r^n|^{q-3}|\sigma(x_r)-\tilde{\sigma}(x_{\kappa(n,r)}^n)|^2 dr
\\
& \qquad\times \Big|\int_{\kappa(n,s)}^s \frac{\partial b^k(x_r^n)}{\partial x^i} \sum_{l=1}^m \tilde{\sigma}^{(i,l)}(x_{\kappa(n,r)}^n)dw_r^l \Big|ds \notag
\\
& \leq K \sum_{k,i=1}^d E\int_{0}^t   \int_{\kappa(n,s)}^s |e_r^n|^{q-1}dr \Big|\int_{\kappa(n,s)}^s \frac{\partial b^k(x_r^n)}{\partial x^i} \sum_{l=1}^m \tilde{\sigma}^{(i,l)}(x_{\kappa(n,r)}^n)dw_r^l \Big|ds \notag
\\
&+ K \sum_{k,i=1}^d E\int_{0}^t   \int_{\kappa(n,s)}^s |e_r^n|^{q-3}|\sigma(x_r^n)-\tilde{\sigma}(x_{\kappa(n,r)}^n)|^2 dr
\\
& \qquad\times \Big|\int_{\kappa(n,s)}^s \frac{\partial b^k(x_r^n)}{\partial x^i} \sum_{l=1}^m \tilde{\sigma}^{(i,l)}(x_{\kappa(n,r)}^n)dw_r^l \Big|ds \notag
\end{align*}
which  due to H\"older's inequality implies
\begin{align*}
&T_{35}  \leq K \sum_{k,i=1}^d \int_{0}^t  \Big( E\Big[\int_{\kappa(n,s)}^s |e_r^n|^{q-1}dr \Big]^\frac{q}{q-1}\Big)^\frac{q-1}{q}
\\
& \times \Big(E\Big|\int_{\kappa(n,s)}^s \frac{\partial b^k(x_r^n)}{\partial x^i} \sum_{l=1}^m \tilde{\sigma}^{(i,l)}(x_{\kappa(n,r)}^n)dw_r^l \Big|^q\Big)^\frac{1}{q} ds \notag
\\
&+ K \sum_{k,i=1}^d \int_{0}^t   \Big(E\Big[\int_{\kappa(n,s)}^s |e_r^n|^{q-3} |\sigma(x_r^n)-\tilde{\sigma}(x_{\kappa(n,r)}^n)|^2 dr \Big]^\frac{q}{q-1}\Big)^\frac{q-1}{q}
\\
& \times \Big(E\Big|\int_{\kappa(n,s)}^s \frac{\partial b^k(x_r^n)}{\partial x^i} \sum_{l=1}^m \tilde{\sigma}^{(i,l)}(x_{\kappa(n,r)}^n)dw_r^l \Big|^q\Big)^\frac{1}{q} ds \notag
\end{align*}
and then an elementary inequality of stochastic integrals along with H\"older's inequality yields
\begin{align*}
 T_{35}  & \leq K n^{-\frac{1}{2}} \int_{0}^t  \Big(n^{-\frac{q}{q-1}+1} E\int_{\kappa(n,s)}^s |e_r^n|^q dr \Big)^\frac{q-1}{q}  ds \notag
\\
&+ K n^{-\frac{1}{2}} \int_{0}^t   \Big(n^{-\frac{q}{q-1}+1} E\int_{\kappa(n,s)}^s |e_r^n|^\frac{(q-3)q}{q-1} |\sigma(x_r^n)-\tilde{\sigma}(x_{\kappa(n,r)}^n)|^\frac{2q}{q-1} dr  \Big)^\frac{q-1}{q}  ds \notag
\\
& \leq K n^{-\frac{3}{2}} \int_{0}^t  \Big( \sup_{0 \leq r \leq s}E|e_r^n|^q  \Big)^\frac{q-1}{q}  ds  \notag
\\
&+ K n^{-\frac{1}{2}} \int_{0}^t   \Big(n^{-\frac{q}{q-1}+1} \int_{\kappa(n,s)}^s (E|e_r^n|^q)^\frac{q-3}{q-1} (E|\sigma(x_r^n)-\tilde{\sigma}(x_{\kappa(n,r)}^n)|^q)^\frac{2}{q-1} dr  \Big)^\frac{q-1}{q}  ds \notag
\\
& \leq K n^{-\frac{3}{2}} \int_{0}^t  \Big( \sup_{0 \leq r \leq s}E|e_r^n|^q  \Big)^\frac{q-1}{q}  ds + K n^{-\frac{7}{2}} \int_{0}^t    \Big(\sup_{0 \leq r \leq s} E|e_r^n|^q\Big)^\frac{q-3}{q}   ds. \notag
\end{align*}
Thus, by Young's inequality,
\begin{align} \label{eq:T35}
T_{35} \leq  K\int_{0}^t  \sup_{0 \leq r \leq s}E|e_r^n|^q   ds + K n^{-\frac{3q}{2}} + K n^{-\frac{7q}{6}}.
\end{align}
Hence, on substituting the estimates from \eqref{eq:T31}, \eqref{eq:T32}, \eqref{eq:T33}, \eqref{eq:T34} and \eqref{eq:T35} in \eqref{eq:T31+T35}, one obtains
\begin{align} \label{eq:T3}
T_{3} \leq  K\int_{0}^t  \sup_{0 \leq r \leq s}E|e_r^n|^q   ds + K n^{-q}.
\end{align}
Finally, on combining the estimates from \eqref{eq:T1}, \eqref{eq:T2} and \eqref{eq:T3} in \eqref{eq:T1+T3}, one completes the proof.
\end{proof}

\begin{proof}[\textbf{Proof of Theorem \ref{thm:main:theorem}}]
Let us take $e_t^n := x_t-x_t^n$ i.e.,
\begin{align*}
e_t^n =  \int_{0}^{t}(b(x_s)-\tilde{b}^n(x_{\kappa(n,s)}^n))ds+\int_{0}^{t}(\sigma(x_s)-\tilde{\sigma}(x_{\kappa(n,s)}^n))dw_s
\end{align*}
for any $t \in [0, T]$. By It\^{o}'s formula, for any $q \geq 2$,
\begin{align*}
|e_t^n|^q& = q \int_{0}^{t} |e_s^n|^{q-2} e_s^n(b(x_s)-\tilde{b}^n(x_{\kappa(n,s)}^n)) ds
\\
& + q \int_{0}^{t} |e_s^n|^{q-2} e_s^n(\sigma(x_s)-\tilde{\sigma}(x_{\kappa(n,s)}^n)) dw_s
\\
&  +   \frac{q(q-2)}{2} \int_{0}^{t} |e_s^n|^{q-4}|(\sigma(x_s)-\tilde{\sigma}(x_{\kappa(n,s)}^n))^*e_s^n|^2 ds
\\
& +   \frac{q}{2} \int_{0}^{t} |e_s^n|^{q-2}|\sigma(x_s)-\tilde{\sigma}(x_{\kappa(n,s)}^n)|^2 ds
\end{align*}
almost surely  for any $t \in [0,T]$,  which on taking expectation and using Schwarz inequality implies
\begin{align} \label{eq:Eito:}
E|e_t^n|^q& \leq  q E\int_{0}^{t} |e_s^n|^{q-2} e_s^n(b(x_s)-\tilde{b}^n(x_{\kappa(n,s)}^n)) ds   \notag
\\
&+   \frac{q(q-1)}{q}  \int_{0}^{t} |e_s^n|^{q-2}|\sigma(x_s)-\tilde{\sigma}(x_{\kappa(n,s)}^n)|^2 ds.
\end{align}
Thus, by using the estimates from \eqref{eq:enbn:rate} and \eqref{eq:sigm:split:rate1}, one obtains,
\begin{align}
E|e_t^n|^q & \leq  K E\int_{0}^{t} |e_s^n|^q ds + K E\int_{0}^{t} |e_s^n|^{q-2} e_s^n(b(x_s^n)-b^n(x_{\kappa(n,s)}^n)) ds \notag
\\
& + K E\int_{0}^{t} |b(x_{\kappa(n,s)}^n)-\tilde{b}^n(x_{\kappa(n,s)}^n)|^q ds  +   K  E \int_{0}^{t} |\sigma(x_s^n)-\tilde{\sigma}(x_{\kappa(n,s)}^n)|^q ds
\end{align}
for any $t \in [0,T]$. Further, on the application of Lemmas [\ref{lem:form:rate:new}, \ref{lem:b-tilde b:rate:new}, \ref{lem:a-tilde a:rate:new}] gives
\begin{align*}
\sup_{0 \leq t \leq u}E|e_t^n|^q \leq K \int_{0}^{u} \sup_{0 \leq r \leq s}E|e_r^n|^q ds+Kn^{-q}
\end{align*}
for any $u \in [0,T]$ and hence Gronwall's lemma completes the proof.
\end{proof}
\begin{proof}[\textbf{Proof of Theorem \ref{thm:main:theorem:sup:inside}}]
This follows due to Theorem \ref{thm:main:theorem} and Lemma  \ref{lem:yor}.
\end{proof}
\section{Numerical Simulation}
In this section, we demonstrate theoretical results obtained in this article with the help of some examples. In order to implement the scheme \eqref{eq:milstein} of SDE \eqref{eq:sde} and scheme \eqref{eq:milstein:continuous} of SDE \eqref{eq:sde:continuous}, one requires commutative conditions on  diffusion and jump coefficients. One can refer to \cite{kloeden1999} and \cite{platen2010}  for details. For the purpose of this section, we consider one-dimensional SDEs defined on interval $[0,1]$, which is partitioned into $2^n$ sub-intervals of equal length (step-size) $h=2^{-n}$ for some $n \in \mathbb{N}$. In what follows, $x_{lh}^h$ denotes the tamed Milstein scheme at $lh$-th grid point and $\Delta w_{lh}:=w_{(l+1)h}-w_{lh}$ for $l=0,\ldots,2^n$.
\begin{example}
%-----------------------------------------
Let us consider one-dimensional SDE defined by,
\begin{align} \label{eq:sim:m1a}
dx_t=(x_t-x_t^5) dt+x_t dw_t
\end{align}
for any $t \in [0, 1]$  with $x_0 =1$. The tamed Milstein scheme of SDE \eqref{eq:sim:m1a} at $(l+1)h$-th grid is given by
\begin{align} \label{eq:sim:tmils1a}
x_{(l+1)h}^h=x_{lh}^h+\frac{x_{lh}^h-(x_{lh}^h)^5}{1+h|x_{lh}^h-(x_{lh}^h)^5|^2}h+x_{lh}^h\Delta w_{lh}+\frac{1}{2}x_{lh}^h\{(\Delta w_{lh})^2-h\}
\end{align}
for $l=0,\ldots,2^n-1$ with $h=2^{-n}$. Here, the above scheme with $h=2^{-21}$ is taken as true solution of SDE \eqref{eq:sim:m1a}. From Table \ref{tab:m1a} and Figure \ref{fig:m1a}, one can notice that  $\mathcal{L}^q$ convergence rates of scheme \eqref{eq:sim:tmils1a} are approximately $1.0$ for all $q=1,\ldots, 5$, which are consistent with our theoretical findings. The number of paths considered is $60,000$.

\begin{figure}[h]
\centering
%\begin{minipage}[b]{0.45\linewidth}
%\includegraphics[scale=.60]{tmil.eps}
%\caption{\em  $\mathcal{L}^q$-convergence rate of  tamed Milstein scheme \eqref{eq:sim:tmils1} of SDE \eqref{eq:sim:m1}}
%\label{fig:m1}
%\end{minipage}
%\quad
%\begin{minipage}[b]{0.45\linewidth}
\includegraphics[scale=.50]{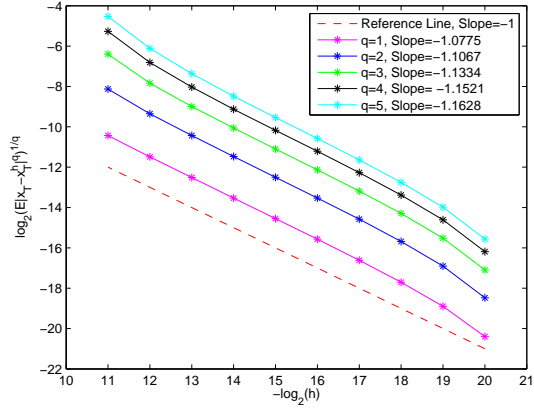}
\caption{\em  $\mathcal{L}^q$-convergence rate of  tamed Milstein scheme \eqref{eq:sim:tmils1a} of SDE \eqref{eq:sim:m1a}}
\label{fig:m1a}
%\end{minipage}
\end{figure}

\begin{table}[h]
\centering
\begin{tabular}{|c|c|c|c|c|c|} \hline
$h=2^{-n}$ & $E|x_T-x_T^h|$   & $\sqrt{E|x_T-x_T^h|^2}$ & $\sqrt[3]{E|x_T-x_T^h|^3}$ & $\sqrt[4]{E|x_T-x_T^h|^4}$ & $\sqrt[5]{E|x_T-x_T^h|^5}$ \\ \hline
$2^{-20}$& 0.0000007223 & 0.0000027396 & 0.0000071338  & 0.0000133872 & 0.0000206505 \\
$2^{-19}$& 0.0000020488 & 0.0000081546 & 0.0000213011    & 0.0000400356 & 0.0000618759 \\
$2^{-18}$& 0.0000046829 & 0.0000190372 & 0.0000498670  & 0.0000937473 & 0.0001448460 \\
$2^{-17}$& 0.0000099526 & 0.0000408359 & 0.0001072262 & 0.0002022094 & 0.0003133946 \\
$2^{-16}$& 0.0000205589 & 0.0000844630 & 0.0002227601  & 0.0004218232 & 0.0006555927 \\
$2^{-15}$& 0.0000417394 & 0.0001723833 & 0.0004554010  & 0.0008642163 & 0.0013460486 \\
$2^{-14}$& 0.0000843948 & 0.0003519474 & 0.0009360518  & 0.0017870537 & 0.0027962853 \\
$2^{-13}$& 0.0001710052 & 0.0007232200 & 0.0019649384  & 0.0038259755 & 0.0060684654 \\
$2^{-12}$& 0.0003479789 & 0.0015293072 & 0.0043796657  & 0.0089031484 & 0.0144907359 \\
$2^{-11}$& 0.0007231189 & 0.0035802581 & 0.0118774764  & 0.0259649292 & 0.0432914310  \\ \hline
\end{tabular}
\caption{\em Tamed Milstein scheme \eqref{eq:sim:tmils1a} of SDE \eqref{eq:sim:m1a}}
\label{tab:m1a}
\end{table}

\hfill $\square$
\end{example}
%--------------------------
%----------------------------------------
\begin{example} %Example-1
%-----------------------------------------
Let us now consider the following one-dimensional SDE,
\begin{align} \label{eq:sim:m2}
dx_t =& -0.10 x_t^3 dt + x_t dw_t + \int_{\mathbb{R}} x_t z \tilde N(dt, dz)
\end{align}
for any $t \in [0, 1]$  with $x_0 =1$. The tamed Milstein scheme of SDE \eqref{eq:sim:m2} is given by,
\begin{align}  \label{eq:sim:tmils2}
x_{(l+1)h}^h = & x_{lh}^h  + \frac{-0.10 (x_{lh}^{h})^3}{1+ h |0.10(x_{lh}^{h})^3|^2}h + x_{lh}^h\Delta w_{lh} \notag
\\
& + x_{lh}^h \sum_{i=N((h, \mathbb{R})+1}^{N((l+1)h,\mathbb{R})} z_i  +   \frac{1}{2} x_{lh}^h \{(\Delta w_{lh})^2-h\} \notag
\\
&+x_{lh}^h \sum_{i=N(lh,\mathbb{R})+1}^{N((l+1)h,\mathbb{R})} z_i (w_{(l+1)h}-w_{\tau_i}) + x_{lh}^h \sum_{i=N(lh,\mathbb{R})+1}^{N((l+1)h,\mathbb{R})} z_i (w_{\tau_i}-w_{lh})  \notag
\\
&+x_{lh}^h \sum_{j=N(lh,\mathbb{R})+1}^{N((l+1)h,\mathbb{R})} \sum_{i=N(lh,\mathbb{R})+1}^{N(\tau_j,\mathbb{R})} z_i z_j
\end{align}
for $l=0,\ldots, 2^{n}-1$. Notice that $N(t, \mathbb{R})$ is a Poisson process with (jump) intensity $\lambda$.

The case $\lambda=3$ in Table \ref{tab:mils2a} and its corresponding plot in Figure \ref{fig:mils2a} are based on $60,000$ trajectories while case $\lambda=5$ of Table \ref{tab:mils2a} and its corresponding plot in Figure \ref{fig:mils2b} are based on $360,000$ paths. In both the cases, the mark random variables  $z_i$s (jump-sizes) are assumed to follow normal distribution with mean $0$  and variance $0.125$.

Similarly, both the cases $\lambda=3$ and $\lambda=5$ in Table \ref{tab:mils2b} and their corresponding plots in Figure \ref{fig:milsc} and Figure \ref{fig:milsd} are based on $240,000$ simulations. Here, the mark random variables $z_i$s are assumed to follow uniform distribution on $[-1/4,1/4]$.

\begin{figure}[h]
\centering
\begin{minipage}[b]{0.48\linewidth}
\includegraphics[scale=.48]{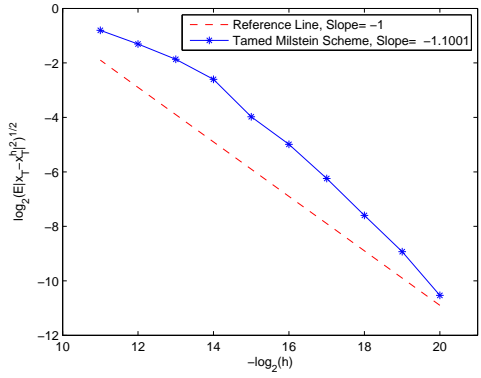}
\subcaption{\em  $\lambda=3$ and mark is normal with mean $0$ and variance $0.125$}
\label{fig:mils2a}
\end{minipage}
\hfill
\begin{minipage}[b]{0.48\linewidth}
\includegraphics[scale=.48]{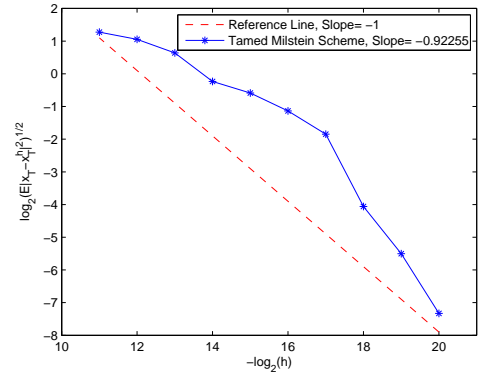}
\subcaption{\em  $\lambda=5$ and mark is normal with mean $0$ and variance $0.125$}
\label{fig:mils2b}
\end{minipage}
\\
\centering
\begin{minipage}[b]{0.48\linewidth}
\includegraphics[scale=.48]{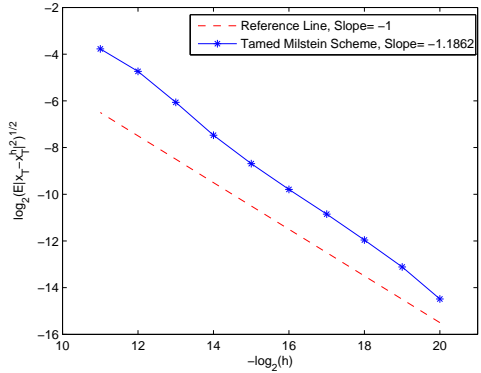}
\subcaption{\em $\lambda=3$ and mark is uniform on $[-1/4,1/4]$}
\label{fig:milsc}
\end{minipage}
\hfill
\begin{minipage}[b]{0.48\linewidth}
\includegraphics[scale=.48]{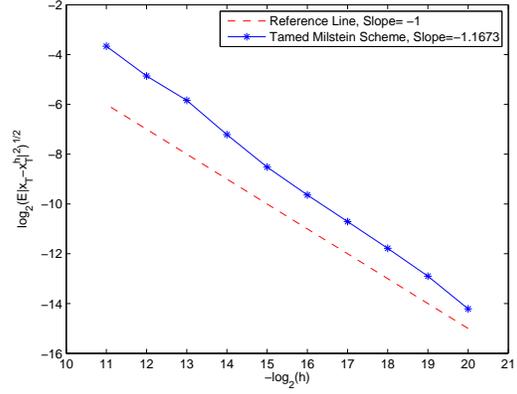}
\subcaption{\em \footnotesize $\lambda=5$ and mark is uniform on $[-1/4,1/4]$}
\label{fig:milsd}
\end{minipage}
\caption{\em $\mathcal{L}^2$-convergence rate of tamed Milstein scheme \eqref{eq:sim:tmils2} of SDE \eqref{eq:sim:m2}.}
\label{fig:mils2}
\end{figure}

\begin{table}[h]
\centering
\begin{minipage}{0.45\textwidth}
\centering
\begin{tabular}{|c|c|c|}
\hline %
$h=2^{-n}$    & \multicolumn{2}{|c|}{$\sqrt{E|x_T-x_T^h|^2}$ } \\ \cline{2-3}
          & $\lambda=3.0$   &  $\lambda=5.0$ \\ \hline
$2^{-20}$ & 0.00067484 & 0.00621555 \\
$2^{-19}$ & 0.00204889 & 0.02203719 \\
$2^{-18}$ & 0.00515874 & 0.06003098 \\
$2^{-17}$ & 0.01321011 & 0.27830910 \\
$2^{-16}$ & 0.03146860 & 0.45542612 \\
$2^{-15}$ & 0.06349005 & 0.66561201 \\
$2^{-14}$ & 0.16459365 & 0.85082102 \\
$2^{-13}$ & 0.27426757 & 1.55620505 \\
$2^{-12}$ & 0.40437133 & 2.07850380 \\
$2^{-11}$ & 0.57251083 & 2.41922833 \\ \hline
\end{tabular}
\subcaption{\em \footnotesize Mark is normal with mean $0$ and variance $0.125$.}
\label{tab:mils2a}
\end{minipage}%
\hfill
\begin{minipage}{0.480\textwidth}
\centering
\begin{tabular}{|c|c|c|}
\hline
$h=2^{-n}$    & \multicolumn{2}{|c|}{$\sqrt{E|x_T-x_T^h|^2}$ } \\ \cline{2-3}
          & $\lambda=3.0$   &  $\lambda=5.0$ \\ \hline
$2^{-20}$ & 0.00004365 & 0.00005260 \\
$2^{-19}$ & 0.00011286 & 0.00013058 \\
$2^{-18}$ & 0.00025133 & 0.00028420 \\
$2^{-17}$ & 0.00054134 & 0.00059787  \\
$2^{-16}$ & 0.00112643 & 0.00125615 \\
$2^{-15}$ & 0.00242178 & 0.00272857 \\
$2^{-14}$ & 0.00563126 & 0.00673629 \\
$2^{-13}$ & 0.01497166 & 0.01747566 \\
$2^{-12}$ & 0.03745749 & 0.03448135 \\
$2^{-11}$ & 0.07302734 & 0.07900926 \\ \hline
\end{tabular}
\subcaption{\em  Mark is uniform on $[-1/4,1/4]$.}
\label{tab:mils2b}
\end{minipage}
\caption{\em Tamed Milstein scheme \eqref{eq:sim:tmils2}  of SDE \eqref{eq:sim:m2}.}
\label{tab:mils2}
\end{table}

It is evident from plots in Figure \ref{fig:mils2} that rate of convergence of the tamed Milstein scheme of SDE driven by L\'evy noise depends significantly on the distribution of mark  random variable and the jump intensity.

\hfill $\square$
\end{example}
\newpage
\appendix
\section*{Appendix A}
In this section, one discusses about implementing the scheme \eqref{eq:milstein} on computer. For the purpose of simplicity, let us assume that
\begin{align} \label{eq:gam:mean}
\int_Z \gamma(x,z)\nu(dz) =0
\end{align}
for any $x \in \mathbb{R}^d$. Thus, SDE \eqref{eq:sde} becomes
\begin{align} \label{eq:sde:modi}
x_t=\xi+\int_0^t b(x_s)ds +\int_0^t\sigma(x_s)dw_s +\int_0^t\int_Z \gamma(x_s,z)N(ds,dz)
\end{align}
for any $t \in [0,T]$. One can partition the interval $[0,T]$ into $H$ subintervals of length $h$ such that $Hh=T$, in particular $H=2^{n}T$ and $h=2^{-n}$ for any $n \in \mathbb{N}$. For simplicity, one can take $T=1$. As a consequence, the tamed Milstein scheme  \eqref{eq:milstein} at the $(l+1)h$-th grid can be written as,
\begin{align}  \label{eq:milstein:modif}
x_{(l+1)h}^h &=x_{lh}^h  + \tilde{b}^n(x_{lh}^h)h + \sum_{k=1}^m\sigma^{(k)}(x_{lh}^h)\Delta w_{lh}^k + \int_{lh}^{(l+1)h}\int_Z \gamma(x_{lh}^h, z_2) N(ds,dz_2) \notag
\\
& + \sum_{j,k=1}^m\sum_{u=1}^d  \int_{lh}^{(l+1)h} \int_{lh}^s  \sigma^{(u,j)}(x_{lh}^h)\frac{\partial \sigma^{(k)} (x_{lh}^h)}{\partial x^u} dw_r^j dw_s^k \notag
\\
&+\sum_{j=1}^m\int_{lh}^{(l+1)h} \int_{lh}^s \int_{Z}\big\{\sigma^{(j)}(x_{lh}^h+\gamma(x_{lh}^h, z_1))-\sigma^{(j)}(x_{lh}^h)\big\} N(dr,dz_1) dw_s^j \notag
\\
&+ \sum_{j=1}^m \sum_{u=1}^d \int_{lh}^{(l+1)h}\int_{lh}^s \int_Z   \frac{\partial \gamma(x_{lh}^h, z_2)}{\partial x^u} \sigma^{(u,j)}(x_{lh}^h)dw_r^j N(ds,dz_2) \notag
\\
&\hspace{-5mm}+\int_{lh}^{(l+1)h}\int_Z \int_{lh}^s \int_{Z} \big\{ \gamma(x_{lh}^h+\gamma(x_{lh}^h, z_1), z_2)-\gamma(x_{lh}^h, z_2) \big\} N(dr,dz_1)  N(ds,dz_2)
\end{align}
for any $l=0,\ldots,2^n-1$, where $h=2^{-n}$ for $n \in \mathbb{N}$. The drift coefficient $\tilde{b}^n$ of scheme \eqref{eq:milstein:modif} is defined in equation \eqref{eq:an:tilde}. The one-dimensional case of scheme \eqref{eq:milstein:modif} is given by
\begin{align}  \label{eq:milstein:modif:1d}
x_{lh+h}^h &=x_{lh}^h  + \tilde{b}^n(x_{lh}^h)h + \sigma(x_{lh}^h)\Delta w_{lh} + \sum_{i=N(lh, Z)+1}^{N(lh+h,Z)} \gamma(x_{lh}^h, z_i) \notag
\\
& +   \frac{1}{2} \sigma(x_{lh}^h) \frac{\partial \sigma(x_{lh}^h)}{\partial x^u} \{(\Delta w_{lh})^2-h\} \notag
\\
&+\sum_{i=N(lh,Z)+1}^{N(lh+h,Z)} \big\{\sigma(x_{lh}^h+\gamma(x_{lh}^h, z_i))-\sigma(x_{lh}^h)\big\} (w_{lh+h}-w_{\tau_i}) \notag
\\
&+ \sigma(x_{lh}^h) \sum_{i=N(lh,Z)+1}^{N(lh+h,Z)} \frac{\partial \gamma(x_{lh}^h, z_i)}{\partial x^u} (w_{\tau_i}-w_{lh})  \notag
\\
&+\sum_{j=N(lh,Z)+1}^{N(lh+h,Z)} \sum_{i=N(lh,Z)+1}^{N(\tau_j,Z)} \big\{ \gamma(x_{lh}^h+\gamma(x_{lh}^h, z_i), z_j)-\gamma(x_{lh}^h, z_j) \big\}
\end{align}
for $l=0,\ldots, 2^{n}-1$. The multi-dimensional case of scheme \eqref{eq:milstein:modif} with mark-dependent jump coefficient is computationally difficult to simulate because one requires to keep track of the jump times in each sub-interval. However, the mark-independent case i.e. when $\gamma(x,z)=\gamma(x)$ for any $x \in \mathbb{R}^d$ and $z\in Z$, can be simplified by imposing diffusion and jump commutative conditions as respectively given below,
\begin{align} \label{eq:sim:comutative}
\sum_{u=1}^{d} \sigma^{(u,j)}(x) \frac{\partial \sigma^{(i,k)}(x)}{\partial x^u} =\sum_{u=1}^{d} \sigma^{(u,k)}(x) \frac{\partial \sigma^{(i,j)}(x)}{\partial x^u}
\end{align}
for any $k,j=1,\ldots,m$, $i=1,\ldots,d$ and $x \in \mathbb{R}^d$, and
\begin{align} \label{eq:jum:comm}
\sigma^{(k,j)}(x+\gamma(x))-\sigma^{(k,j)}(x)= \sum_{u=1}^d \frac{\partial \gamma^k(x)}{\partial x^u} \sigma^{(u,j)}(x)
\end{align}
for any $x \in \mathbb{R}^d$, $k=1,\ldots,d$ and $j=1,\ldots,m$. A detailed discussion on commutative conditions can be found in \cite{platen2010}. Thus, the tamed Milstein scheme \eqref{eq:milstein:modif} of SDE \eqref{eq:sde:modi} with diffusion and jump coefficients satisfying commutative conditions \eqref{eq:sim:comutative} and \eqref{eq:jum:comm}, can be written as
\begin{align}  \label{eq:milstein:modif:com}
x_{lh+h}^h &=x_{lh}^h  + \tilde{b}^n(x_{lh}^h)h + \sum_{k=1}^m\sigma^{(k)}(x_{lh}^h)\Delta w_{lh}^k + \gamma(x_{lh}^h) \Delta N(lh,Z) \notag
\\
& + \frac{1}{2}\sum_{j,k=1}^m\sum_{u=1}^d   \sigma^{(u,j)}(x_{lh}^h)\frac{\partial \sigma^{(k)} (x_{lh}^h)}{\partial x^u} \{\Delta w_{lh}^j \Delta w_{lh}^k-hI_{\{j=k\}}\} \notag
\\
&+\sum_{j=1}^m \big\{\sigma^{(j)}(x_{lh}^h+\gamma(x_{lh}^h))-\sigma^{(j)}(x_{lh}^h)\big\} \Delta N(lh,Z) \Delta w_{lh}^j \notag
\\
&+\frac{1}{2} \big\{ \gamma(x_{lh}^h+\gamma(x_{lh}^h))-\gamma(x_{lh}^h) \big\} \{(\Delta N(lh,Z))^2- \Delta N(lh,Z)\}
\end{align}
where $\Delta N(lh,Z)=N(lh+h,Z)-N(lh,Z)$ (i.e. the number of jumps in $[lh,lh+h]$) for any $l=0,\ldots,2^n-1$. Also, notice that when assumption \eqref{eq:gam:mean} does not hold, one can implement the scheme \eqref{eq:milstein:modif} by replacing the drift coefficient $b^n(x)$ with $b^n(x)-\int_Z \gamma(x,z)\nu(dz)$.

\section*{Acknowledgments} This work has made use of the resources provided by the Edinburgh Compute and Data Facility (ECDF) \url{http://www.ecdf.ed.ac.uk/}. All the simulations included in this article are performed in c programming language using MPI parallel system.

%-----------------------------------------------------------------

\end{document}